\documentclass[leqno, a4paper]{amsart}
\usepackage{amsmath}
\usepackage{amssymb, amsmath, amsthm, bm, cite}
\usepackage{mathrsfs, mathtools, float, pgfplots}
\usepackage[utf8]{inputenc}
\usepackage[ruled]{algorithm2e}
\usepackage[noend]{algpseudocode}
\pgfplotsset{compat=1.14}
\usepackage[colorlinks=true]{hyperref}
\usepackage[utf8]{inputenc}
\usepackage[margin=1.1in]{geometry}
\usepackage{xcolor}
\usepackage{thm-restate}

\newtheorem{theorem}{Theorem}
\newtheorem{definition}{Definition}

\newtheorem{lemma}{Lemma}
\definecolor{mygreen}{HTML}{43a047}
\newcommand{\longweak}{\relbar\joinrel\rightharpoonup}
\newcommand{\eps}{\varepsilon}
\newcommand{\dd}{\textup{d}}

\renewcommand{\div}{\text{div}}
\DeclareMathOperator*{\esssup}{ess\,sup}
\hypersetup{urlcolor=black, citecolor=red}

\numberwithin{equation}{section}

\pgfplotsset{colormap={rainbow}{[1pt]
		rgb(0pt)    =(0,0,1),
		rgb(250pt) =(0,1,1),
		rgb(500pt) =(0,1,0),
		rgb(750pt) =(1,1,0),
		rgb(1000pt)=(1,0,0),
	},
}

\pgfplotsset{colormap={blue}{[1pt]
		rgb(0pt)    =(1,1,1),
		rgb(50pt) =(0.75,0.9,1),
		rgb(120pt) =(0.5, 0.8, 1),
		rgb(180pt) =(0.25, 0.7, 1),
		rgb(250pt) =(0, 0.6, 1),
		rgb(300pt) =(0, 0.5, 0.9),
		rgb(370pt) =(0, 0.5, 0.85),
		rgb(430pt) =(0, 0.3, 0.75),
		rgb(500pt) =(0,0.2,0.7),
		rgb(560pt) =(0,0.1,0.6),
		rgb(620pt) =(0,0,0.5),
		rgb(680pt) =(0,0,0.4),
		rgb(810pt) =(0,0,0.25),
		rgb(870pt) =(0,0,0.15),
		rgb(930pt) =(0,0,0.1),
		rgb(1000pt)=(0,0,0),
	},
}

\pgfplotsset{colormap={red}{[1pt]
		rgb(0pt)    =(0,0,0),
		rgb(60pt) =(0.2,0,0),
		rgb(120pt) =(0.4, 0, 0),
		rgb(190pt) =(0.55, 0, 0),
		rgb(250pt) =(0.7, 0, 0),
		rgb(310pt) =(0.85, 0, 0),
		rgb(380pt) =(0.9, 0, 0),
		rgb(440pt) =(1, 0.2, 0),
		rgb(1000pt)=(1,1,1),
	},
}

\title[Phase-field Models of Tumor Growth and Invasion due to ECM Degradation]{Local and Nonlocal Phase-field Models of Tumor Growth and Invasion due to ECM Degradation}
\author[M. Fritz, E. A. B. F. Lima, V. Nikoli\'c, J. T. Oden, and B. Wohlmuth]{Marvin Fritz$^{1,*}$, Ernesto A. B. F. Lima$^2$, Vanja Nikoli\'c$^1$,\\ J. Tinsley Oden$^2$, and Barbara Wohlmuth$^1$}
\thanks{${}^*$Corresponding author}
\email{\{marvin.fritz, vanja.nikolic, barbara.wohlmuth\}@ma.tum.de}
\email{\{lima, oden\}@oden.utexas.edu}
\begin{document} 
\maketitle
\vspace*{-2mm}
{\footnotesize
   \centerline{$^1$Department of Mathematics, Technical University of Munich, Germany \\[1mm]}
   \centerline{$^2$Oden Institute for Computational Engineering and Sciences, The University of Texas at Austin, USA}
}

\vspace{8mm}
\begin{abstract}
We present and analyze new multi-species phase-field mathematical models of tumor growth and ECM invasion. The local and nonlocal mathematical models describe the evolution of volume fractions of tumor cells, viable cells (proliferative and hypoxic cells), necrotic cells, and the evolution of MDE and ECM, together with chemotaxis, haptotaxis, apoptosis, nutrient distribution, and cell-to-matrix adhesion. We provide a rigorous proof of the existence of solutions of the coupled system with gradient-based and adhesion-based haptotaxis effects. In addition, we discuss finite element discretizations of the model, and we present the results of numerical experiments designed to show the relative importance and roles of various effects, including cell mobility, proliferation, necrosis, hypoxia, and nutrient concentration on the generation of MDEs and the degradation of the ECM.
\end{abstract}
\vspace{10mm}
\section{Introduction} \label{Sec_Introduction}
An important factor in tumor growth and invasion of healthy tissue in humans, and a first step toward metastasis, is the over expression by tumor cells of matrix degenerative enzymes (MDEs) that erode the extracellular matrix (ECM) and allow the migration of tumor cells into the tissue. The expression of MDEs such as urokinase-plasminogen activator and matrix metalloproteinases lead to the activation of plasminogen and the degrading protein plasmin (see, e.g. \cite{chaplain2005mathematical, nargis2016effects, madsen2015source}). According to \cite{madsen2015source}, "matrix degradation is central to tumor pathogenesis", and the degradation of ECM "makes room for migration as cells cannot move into regions of the tissue which are too dense"~\cite{nargis2016effects}. 

This study complements and extends recent work on general phase-field models reported in \cite{fritz2018unsteady, lima2017selection,lima2015analysis}. The models developed and analyzed there are intended to depict phenomena at the mesoscale and macroscale where tumor constituents are determined by fields representing volume fractions of mass concentrations of various species. Local versions of multiphase models have been proposed by several authors over the last decade, and we mention as examples the papers of Araujo and McElwain~\cite{araujo2004history}, Garcke et al.~\cite{garcke2018multiphase,garcke2016cahn}, Wise et al. \cite{wise2008three}, and Lima et al.~\cite{lima2014hybrid}. Recent literature on models of tumor growth is surveyed in, for example,~\cite{bellomo2008foundations, bellomo2000modelling, deisboeck2010multiscale, oden2016toward}. Among studies of phenomenological models of tumor cell invasion and tumor-host interaction, we mention \cite{gatenby1995models,perumpani1996biological,anderson2000mathematical,perumpanani2000traveling,marchant2001travelling,chaplain2005mathematical,madsen2015source,nargis2016effects,hillen2013convergence,peng2017multiscale}. Typically, in these works the models are characterized by systems of reaction-diffusion partial differential equations describing the evolution of concentrations of densities of tumor cells, ECM, and some form of matrix-degradation agent, such as MDE.

Among other factors influencing tumor cell mobility and migration are long-range interactions due to such phenomena as cell-to-cell adhesion. Cell-to-cell adhesion involves the binding of one or more cells to each other through the reaction of proteins on the cell surfaces and is a key factor in tissue formation, stability, and the breakdown of tissue. This adhesion related deterioration of tissue is a factor contributing to the invasion and metastasis of cancer cells (see e.g. \cite{armstrong2006continuum,chaplain2011mathematical,chaplain2005mathematical}). Several nonlocal mathematical models of adhesion (meaning models in which events or cell concentrations at a point $x$ in the tumor domain depend on events at points distinct from $x$ but within a finite neighborhood of $x$) have been proposed in the literature. For an example of such cell-to-cell adhesion models, see Armstrong et al.~\cite{armstrong2006continuum}, Chaplain et al.~\cite{chaplain2005mathematical,anderson2000mathematical}, Engwer et al.~\cite{engwer2017structured}, and Stinner et al.~\cite{stinner2014global}, the latter two references addressing the effects of adhesion on tumor-cell invasion.

The inclusion of such nonlocal effects in mesoscale models of tumor growth leads to convolution terms in the Ginzburg--Landau free energy functional of the tumor and gives rise to models involving systems of nonlinear integro-differential equations. An analysis of a class of such models is discussed in a recent study \cite{fritz2018unsteady}.

In the current work, we introduce new nonlocal, multi-species, phase-field mathematical models of tumor growth and invasion due to ECM degradation. The models depict the evolution of volume fractions of tumor cells, viable cells (proliferative and hypoxic cells), necrotic cells and the evolution of MDE and ECM, together with chemotaxis, haptotaxis, apoptosis, and nutrient distribution.\\
\indent We then provide a rigorous analysis of existence of solutions of the full model system. To the authors' best knowledge, there has been no prior analytical treatment of a phase-field tumor system with ECM degradation.  In~\cite{engwer2017structured,stinner2014global,chaplain2011mathematical}, diffusion-type tumor models with invasion due to ECM degradation are analyzed. Phase-field tumor systems without ECM degradation are treated in Garcke et al.~\cite{garcke2017well,garcke2018cahn}. We combine these two aspects in one tumor growth model. The main challenge in the analysis is to control the ECM density without having a maximum principle for the phase-field tumor equations, as can be done for diffusion-type tumor models; see, e.g.,~\cite{stinner2014global}.\\
\indent In this work, we also discuss efficient finite element discretizations of the model. We present the results of numerical experiments designed to show the relative importance and roles of various effects, including cell mobility, proliferation, necrosis, hypoxia, and nutrient concentration on the generation of MDEs and the degradation of the ECM.

Following this introduction, we describe two families of haptotaxis effects in tumor models in Section \ref{Sec_Modelling}, and include discussions of the role and interpretation of key terms in mass balance laws and the models of MDE production, and the evolution of ECM. After the mathematical notation is introduced in Section \ref{Section_Notation}, a complete mathematical analysis of a local and nonlocal model is presented in Section \ref{Sec_Analysis}. Finite element approximations and time-marching schemes are presented in Section \ref{Sec_FEM} and results of numerical experiments are collected in Section \ref{Sec_Simulations}. Concluding comments are provided in Section \ref{Sec_Conclusion}.

\section{Models of Tumor Growth and ECM Degradation} \label{Sec_Modelling}
We begin with a generalization of the setting described in~\cite{oden2016toward} in which a tumor mass, contained in a region $\Omega \subset \mathbb{R}^d$, $d\in \{2,3\}$, at time $t \in [0, T]$, is viewed as a mixture of constituents of constant and equal mass density $\varrho_0$ characterized by volume fractions $\phi_{\beta}: \overline{\Omega} \times [0, T] \rightarrow \mathbb{R}$, $\beta \in \{T, P, H, N\}$. The volume fraction of tumor cells $\phi_T$ is made of proliferative cells ($\phi_P$), which have a high probability of migration or growing in density (e.g. through mitosis and cell-to-cell and cell-to-matrix adhesion) in $\Omega$, hypoxic cells ($\phi_H$) that are in a harsh environment with low nutrient availability, and necrotic cells ($\phi_N$) that are cells that died due to the lack of nutrients. The total tumor cell volume fraction is then the sum $\phi_T=\phi_P+\phi_H+\phi_N$. We assume that the tumor growth is logistic, with a proliferation rate $\lambda^{\textup{pro}}_T$, and thus the viable cells ($\phi_V=\phi_P+\phi_H=\phi_T-\phi_N$) can proliferate until the capacity of the domain is reached (i.e., $\phi_T=1$). The tumor volume fraction can decrease due to two phenomena: (1) natural cell death (apoptosis) of viable cells at a rate $\lambda_T^{\textup{apo}}$; (2) degradation of necrotic cells at a rate $\lambda_N^{\textup{deg}}$. \\
\indent The tumor is supplied with nutrients, $\phi_\sigma$, such as oxygen or glucose by the vascular system that nourishes both healthy and tumor cells and which dictates the process of chemotaxis whereby cells migrate in the direction of increasing gradient of the nutrient. Here we characterize the nutrient concentration over $\overline{\Omega} \times [0, T]$ by a scalar field $\phi_{\sigma}=\phi_\sigma(x,t)$ governed by a reaction-diffusion equation. \\
\indent The tumor is embedded in a network of macromolecules called the extracellular matrix (ECM), the density of which is represented by a scalar-valued field $\theta=\theta(x,t)$. The ECM is non-diffusible~\cite{nargis2016effects} and its evolution can be modeled by a logistic-type evolution equation which captures the degradation of ECM due to the action of certain matrix-degenerative enzymes (MDEs). When the local nutrient supply (indicated by $\phi_{\sigma}$) drops below a certain threshold, tumor cells may enter a state of hypoxia in which enzymes are released by hypoxic cells that make room for cell migration by eroding the ECM. This process is called haptotaxis. The concentration of matrix degenerative enzymes is characterized here by a field $\phi_M=\phi_M(x,t)$. \\
\indent The mechanical behavior of the tumor mass must obey the balance laws of mechanics, namely the laws of conservation of mass, momentum, and energy. We will ignore thermal effects, and also, for the moment, mechanical deformations, see e.g. \cite{lima2017selection}, as well as convective flow velocities in the material time derivatives, see e.g. \cite{fritz2018unsteady}, concentrating on mass conservation.  

Under these assumptions, tumor mass ($\textup{m}_T= \int_{\Omega} \varrho_0 \phi_T\, \textup{d}x$) is conserved ($\textup{d} \textup{m}_T / \textup{d}t=\Gamma$, $\Gamma$ being the mass supplied to by other constituents). This leads to the evolution equation,
\begin{equation} \label{1}
\begin{aligned}
\partial_t \phi_T= \text{div}J-\textup{div} J_\alpha+\lambda^{\textup{pro}}_T \phi_\sigma \phi_V(1-\phi_T) -\lambda_T^{\textup{apo}}\phi_V -\lambda_N^{\textup{deg}}\phi_N. 
\end{aligned}
\end{equation}
Here, $J$ is the mass flux, $\lambda^{\textup{pro}}_T$ and $\lambda_T^{\textup{apo}}$ are non-negative parameters governing the rate of growth and decline of tumor cell volume due to cell proliferation and apoptosis, respectively, $\lambda_N^{\textup{deg}}$ is the rate in reduction of $\phi_N$ due to the natural removal of necrotic cells and $J_\alpha$ is the adhesion flux (cf. \cite{armstrong2006continuum}) representing the influx of tumor mass due to cell-to-matrix effects, such as haptotaxis and cell--ECM adhesion. We refer to both $J$ and $J_\alpha$ as "mass" fluxes recognizing that they are actually characterized by volume fractions of constituents rather than mass concentrations because the constituent mass densities are assumed to be equal and constant and thus do not appear in the mass balance law. 

According to well-established thermodynamics arguments, the mass flux is of the form,
\begin{equation}
\begin{aligned}
J=m_T(\phi_V)\nabla \mu,
\end{aligned}
\end{equation}
where $m_T$ is the cell mobility matrix (such as $m_T(\phi_V)=M_T \phi_V^2 (1-\phi_V)^2$, $M_T>0$) and $\mu$ is the chemical potential,
\begin{equation}\label{mu}
\begin{aligned} 
\mu= \frac{\delta \mathcal{E}}{\delta \phi_T}=\Psi'(\phi_T)-\varepsilon^2_T \Delta \phi_T-\chi_C \phi_\sigma+\delta_T \phi_T,
\end{aligned}
\end{equation}
$\mathcal{E}$  being the Ginzburg--Landau free energy functional,
\begin{equation}\label{Ginzburg}
\begin{aligned} 
\mathcal{E}(\phi_T,\phi_\sigma)=\int_\Omega \left( \Psi(\phi_T)+\frac{\varepsilon^2_T}{2}|\nabla \phi_T|^2-\chi_C \phi_\sigma \phi_T+\frac{1}{2\delta_\sigma} \phi_\sigma^2 + \frac{\delta_T}{2} \phi_T^2 \right)   \dd x.
\end{aligned}
\end{equation}
$\delta \mathcal{E}/ \delta \phi_T$ denotes the variational or Gateaux derivative of $\mathcal{E}$ with respect to $\phi_T$. In (\ref{Ginzburg}), $\delta_T$ is a positive parameter relating to the level of cell diffusion, $\Psi$ is a double-well potential (such as $\Psi(\phi_T)=\overline{E}\phi_T^2(1-\phi_T)^2$, $\overline{E}>0$), $\varepsilon_T$ is a parameter characterizing surface energy of domains separated by large gradients in $\phi_T$, and $\chi_C$ is the chemotaxis parameter. If $\mu$ is simply $\delta_T \phi_T$, then $\div(J)$ in (\ref{1}) collapses to a classical diffusion term $J=\div(m_T(\phi_V) \delta_T \nabla \phi_T)$. The potential $\Psi$ penalizes the energy (increases it to move the system away from a minimum energy point) when $\phi_T \notin [0,1]$. The presence of the Laplacian in (\ref{mu}) leads to a fourth-order evolution equation of the Cahn--Hilliard type when $\mu$ is introduced into (\ref{1}). The resulting model is a diffused-interface or phase-field model in which the boundary between "phases" ($\phi_T,\phi_V,\phi_N,\dots$) is an implicit part of the solution.

The adhesion flux $J_\alpha$ in \eqref{1} represents either a \emph{local} gradient-based (cf. \cite{stinner2014global,tao2011chemotaxis,walker2007global}) or a \emph{nonlocal} adhesion-based haptotaxis effect; cf.~\cite{armstrong2006continuum,chaplain2011mathematical,gerisch2008mathematical}. Therefore, we consider the cases $\alpha \in \{\text{loc}, \text{nonloc}\}$ with respective fluxes of the form
\begin{equation} \label{adhesion_flux}
J_\alpha = \chi_H \phi_V \cdot \begin{cases} \nabla \theta, &\alpha = \text{loc}, \\ k*\theta, &\alpha=\text{nonloc},\end{cases}
\end{equation}
where $\chi_H$ is the so-called haptotaxis parameter, $k$ is a vector-valued kernel function and $*$ denotes the convolution operator, which is set to zero outside of the domain $\Omega$. We will specify assumptions on $k$ needed in the analysis later.

The rate-of-change of the volume fraction of necrotic cells, $\phi_N$, is assumed to be non-diffusive and increases when the nutrient drops below a threshold $\sigma_{VN}$. Also, some of the necrotic cells are removed from the tumor domain and leave as waste products. We propose to capture these phenomena by the evolution equation,
\begin{equation}\partial_t \phi_N = \lambda_{VN} \mathscr{H}(\sigma_{VN}-\phi_\sigma) \phi_V - \lambda_N^{\textup{deg}} \phi_N,\label{N} \end{equation}
where $\lambda_{VN}$ is a non-negative parameters and $\mathscr{H}$ is the Heaviside step function.

To the mass balance \eqref{1}, we add the equations governing the evolution of the nutrient, the MDE, and the ECM,
\begin{align}
\partial_t \phi_\sigma &=    \div\big(D_\sigma(\phi_\sigma) (\delta_\sigma^{-1} \nabla \phi_\sigma - \chi_C \nabla \phi_T)\big)-\lambda_T^{\textup{pro}} \phi_V \frac{\phi_\sigma}{\phi_\sigma + \lambda_\sigma^{\textup{sat}}}, \label{sigma}  \\
\partial_t \phi_M &= \div(D_M(\phi_M) \nabla \phi_M) - \lambda^{\textup{dec}}_M \phi_M + \lambda^{\textup{pro}}_M \phi_V \theta\frac{\sigma_H}{\sigma_H+\phi_\sigma}  (1-\phi_M) - \lambda^{\textup{dec}}_\theta \theta \phi_M,  \label{MDE} \\
\partial_t \theta &= -\lambda_\theta^{\textup{deg}} \theta \phi_M, \label{ECM}
\end{align}
In (\ref{sigma})--(\ref{ECM}), we assume that the nutrient volume fraction decreases as it is consumed by viable tumor cells. The production of MDE by the viable cells is proportional to the nutrient and ECM concentrations at a rate $\lambda^{\textup{pro}}_M$. We assume that the production is higher at low-nutrient \cite{nargis2016effects} and high ECM concentration environments. The MDE concentration decreases due to a natural decay, $\lambda^{\textup{dec}}_M$, and the decay of the ECM, $\lambda_\theta^{\textup{dec}}$. The quantities $\lambda^{\textup{pro}}_M$, $\lambda^{\textup{dec}}_M$,  $\lambda_\theta^{\textup{dec}}$, and $\lambda_\theta^{\textup{deg}}$ are non-negative parameters governing the rate of growth or decay of the MDE and ECM, as indicated. 


\section{Notation and auxiliary results} \label{Section_Notation}
For notational simplicity, we omit the spatial domain $\Omega$ when denoting various Banach spaces and write only $L^p, H^m, W^{m,p}$, where $1\leq p\leq \infty$ and $1\leq m<\infty$. These spaces are equipped with the norms $|\cdot|_{L^p}$, $|\cdot|_{H^m}$, and $|\cdot|_{W^{m,p}}$. We denote by $(\cdot,\cdot)$ the scalar product in $L^2$. The brackets $\langle \cdot,\cdot \rangle$ stand for the duality pairing on $(H^1)'\times H^1$. In the case of $d$-dimensional vector functions, we  write $[L^p]^d$, $[H^m]^d$ and $[W^{m,p}]^d$.\\
\indent For a given Banach space $X$, we define the Bochner space
$$L^p(0,T;X)=\{ u :(0,T) \to X: u \text{ Bochner measurable, } \int_0^T |u(t)|_X^p \,\text{d$t$} < \infty \},$$
where $1 \leq p < \infty$, with the norm $$\|u\|_{L^p X} = \|u\|_{L^p(0,T;X)} =\left( \int_0^T |u(t)|_X^p \,\dd t \right)^{1/p};$$ see~\cite{evans2010partial,roubicek}. For $p=\infty$, we equip $L^\infty(0,T;X)$ with the norm $$\|u\|_{L^\infty X}=\|u\|_{L^\infty(0,T;X)} = \esssup_{t\in (0,T)} |u(t)|_X$$ and we introduce the Sobolev--Bochner space as
$$W^{1,p}(0,T;X)=\{ u \in L^p(0,T;X) : \partial_t u \in L^p(0,T;X) \}.$$
\indent Throughout this paper, $C<\infty$ stands for a generic positive constant.
\subsection{Helpful inequalities} We recall the Poincar\'e inequality,
\begin{alignat}{2}
|f-\overline{f}|_{L^2} &\leq C_\text{P} |\nabla f|_{L^2} &&\quad\text{for all } f \in H^1, \label{Poincare_1}
\end{alignat}
where $C_\text{P}<\infty$ and $\overline{f}=\frac{1}{|\Omega|}\int_\Omega f(x) \,\dd x$ is the mean of $f$; cf.~\cite{roubicek}. We also recall Young's inequality for convolutions, 
\begin{align} \label{Young_convolution}
    |f *g|_{L^r} \leq |f|_{L^p} |g|_{L^q}, \quad p,q,r \geq 1, \quad 1+\tfrac{1}{r}=\tfrac{1}{p}+\tfrac{1}{q},
\end{align}
where $f \in L^p$, $g \in L^q$; see~\cite[Theorem 4.2]{lieb2001analysis}. Gronwall's inequality will be often employed as well.
\begin{lemma}[Gronwall, cf. Lemma 3.1 in~\cite{garcke2017well}] Let $u,v \in C([0,T];\mathbb{R}_{\geq 0})$. If there are positive constants $C_1,C_2<\infty$ such that
$$u(t)+v(t) \leq C_1 + C_2 \int_0^t u(s) \, \textup{d}s \quad \text{ for all } t\in [0,T],$$
then it holds that $$u(t) + v(t) \leq C_1 e^{C_2 T} \quad \text{for all }\ t \in [0,T].$$
\label{Lem_Gronwall}
\end{lemma}
\subsection{Embedding results} Let $X$, $Y$, $Z$ be Banach spaces such that $X$ is compactly embedded in $Y$ and $Y$ is continuously embedded in $Z$, i.e. $X\hookrightarrow \hookrightarrow Y \hookrightarrow Z$. In the proof of the existence theorem below we will rely on the Aubin--Lions compactness lemma, see~\cite[Corollary 4]{simon1986compact},
\begin{equation}\begin{alignedat}{2} L^p(0,T;X) \cap W^{1,1}(0,T;Z) &\hookrightarrow \hookrightarrow L^p(0,T;Y), &&\quad 1\leq p<\infty, \\
L^\infty(0,T;X) \cap W^{1,r}(0,T;Z) &\hookrightarrow \hookrightarrow C([0,T];Y), &&\quad r >1.
\label{Eq_AubinLions}
\end{alignedat}
\end{equation}
Furthermore, we make use of the following continuous embeddings
\begin{alignat}{2} L^2(0,T;Y) &\cap H^1(0,T;Z) &&\hookrightarrow C([0,T];[Y,Z]_{1/2}), \label{Eq_ContEmbedding1}\\
L^\infty(0,T;Y) &\cap C_w([0,T];Z) &&\hookrightarrow C_w([0,T];Y),\label{Eq_ContEmbedding2}
\end{alignat}
where $[Y,Z]_{1/2}$ denotes the interpolation space between $Y$ and $Z$; cf.~\cite[Theorem 3.1, Chapter 1]{lions2012non} and~\cite[Theorem 2.1]{strauss1966continuity}. We refer to~\cite[Definition 2.1, Chapter 1]{lions2012non} for the definition of the interpolation space. In (\ref{Eq_ContEmbedding2}), $C_w([0,T];Y)$ denotes the space of weakly continuous functions on the interval $[0,T]$ with values in $Y$. 
\subsection{General assumptions}
We make the following assumptions on the domain and parameters throughout the paper.\\
\begin{itemize}
\item[\textbf{(A1)}] $\Omega \subset \mathbb{R}^d$, where $d\in \{2,3\}$, is a bounded domain with Lipschitz boundary and $T>0$ is a fixed time horizon.\\
\item[\textbf{(A2)}] The mobility $m_T \in C_b(\mathbb{R}^2)$ satisfies $$(\exists \, m_0,\, m_\infty >0)\, (\forall x \in \mathbb{R}^2) : \quad  m_0 \leq m_T(x) \leq m_\infty .$$
\item[\textbf{(A3)}] The functions $D_\sigma, D_M \in C_b(\mathbb{R})$ satisfy
$$(\exists \, D_0, D_\infty >0)\, (\forall x \in \mathbb{R}) : \quad D_0 \leq D_\sigma(x) \leq D_\infty, \quad D_0 \leq D_M(x) \leq D_\infty.$$
\item[\textbf{(A4)}] The constants $\eps_T, \delta_\sigma, \lambda^{\textup{pro}}_M$ are positive and fixed, while $\chi_C$, $\delta_T$, $\sigma_{VN}$, $\lambda_T^{\textup{apo}}$, $\lambda_N^{\textup{deg}}$ are non-negative fixed constants. \\
\item[\textbf{(A5)}] The potential $\Psi \in C^{1,1}(\mathbb{R})$ is non-negative, continuously differentiable, with globally Lipschitz derivative, and satisfies
\begin{equation*}
    \begin{aligned}
(\exists  R_1, R_2, R_3>0) \ (\forall    x \in \mathbb{R}): \quad \Psi(x) \geq R_1|x|^2-R_2,\quad
 |\Psi'(x)|\leq R_3(1+|x|).
    \end{aligned}
\end{equation*}
\item[\textbf{(A6)}] The adhesion flux $J_\alpha$, where $\alpha \in \{\text{loc},\text{nonloc}\}$, is of the form $$J_\alpha(\phi_T,\phi_N,\theta)=g(\phi_T,\phi_N) G(\theta)$$ with $g \in C_b(\mathbb{R}^2)$ and $G \in \mathscr{L}(X_\alpha;[L^2]^d)$. The space $X_{\alpha}$ is defined as
\begin{equation} \label{X_alpha}
    \begin{aligned}
        X_{\alpha}=\begin{cases}
        H^1\cap L^\infty, \quad &\alpha=\textup{loc}, \\
        L^2, \quad &\alpha=\textup{nonloc}.
        \end{cases}
    \end{aligned}
\end{equation}
\end{itemize}

The assumptions (A1)--(A5) are typical in tumor growth models; see, e.g.,~\cite{garcke2018cahn,garcke2017well,garcke2016global,garcke2016cahn,fritz2018unsteady}.
Assumption (A6) is satisfied if we modify the adhesion flux in (\ref{adhesion_flux}) by replacing $\phi_V$ with the bounded cut-off functional $\mathcal{C}(\phi_V)=\max(0,\min(1,\phi_V))$. This approach is also common in tumor modelling; cf.~\cite{garcke2018cahn,fritz2018unsteady}. We define $g(\phi_T,\phi_N)=\mathcal{C}(\phi_T-\phi_N)$ and $$G(\theta)=\begin{cases} \nabla \theta, &\alpha=\text{loc}, \\ k*\theta, &\alpha=\text{nonloc},\end{cases}$$ for a kernel function $k\in L^1(\mathbb{R}^d)$, which gives the following estimate on the adhesion flux,
$$\begin{aligned} |J_{\text{loc}}|_{L^2} &\leq \chi_H |\nabla \theta|_{L^2} \leq \chi_H |\theta|_{H^1 \cap L^\infty},
\\ |J_{\text{nonloc}}|_{L^2} &\leq \chi_H |k*\theta|_{L^2} \leq \chi_H |k|_{L^1} |\theta|_{L^2},
\end{aligned}$$
where we applied Young's inequality for convolutions (\ref{Young_convolution}) in the case $\alpha=\text{nonloc}$. Here, we equip the intersection space $X_\text{loc}=H^1\cap L^\infty$ with the norm $|\cdot|_{H^1\cap L^\infty}:=|\cdot|_{H^1}+|\cdot|_{L^\infty}$.

\subsection{Comparison to other tumor growth models}

In Section~\ref{Sec_Analysis}, we provide a rigorous analysis of existence of solutions to a modification of the system governed by the equations (\ref{1}), (\ref{mu}), (\ref{N})--(\ref{ECM}). In our model, we combine the effects of tumor growth and invasion, ECM degradation and the separation of tumor phases into viable and necrotic cells. 

The basis of phase-field tumor models, i.e., a Cahn--Hilliard equation for the tumor volume fraction $\phi_T$ and a reaction-diffusion equation for the nutrient concentration $\phi_\sigma$, has been proposed in~\cite{hawkins2012numerical} and has been extended to general multiphase models in~\cite{garcke2018multiphase}. The existence analysis for this model is provided by Garcke et al. \cite{garcke2017well,colli2015cahn} and additionally, several flow models for the velocity field of the mixture have been proposed and analyzed, e.g., flow models by Darcy \cite{garcke2016global,garcke2016cahn,garcke2018cahn,jiang2015well,dai2017analysis}, Brinkman \cite{ebenbeck2019analysis,ebenbeck2019cahn}, Darcy--Forchheimer--Brinkman \cite{fritz2018unsteady} and Navier--Stokes \cite{lam2017thermodynamically}.

To account for cell-to-matrix and cell-to-cell adhesion effects, nonlocal models have been proposed, see e.g. \cite{chaplain2011mathematical,frigeri2017diffuse}. For the analysis of cell-to-cell adhesion models, we refer to \cite{della2016nonlocal,della2018nonlocal,fritz2018unsteady}. To account for cell-to-matrix adhesion, one has to introduce the extracellular matrix, and up to the authors' knowledge, there has been no coupling of the ECM density to a phase-field type tumor growth model. In \cite{stinner2014global,engwer2017structured,chaplain2011mathematical}, diffusion-type tumor models with ECM degradation have been considered and analyzed. 

Our model combines both the phase-field type and the effect of ECM degradation into one system. The main challenge in the analysis of our system is to control the ECM density without having a maximum principle for the phase-field tumor equations, as can be done for diffusion-type tumor models, see \cite{stinner2014global}.

\section{Analysis of the local and nonlocal model} \label{Sec_Analysis}
We consider the system given by equations (\ref{1}), (\ref{mu}), (\ref{N})--(\ref{ECM}) and modify it to perform the analysis. Since the equation for the ECM density \eqref{ECM} is given by an operator-valued ordinary differential equation, its solution can be expressed via the integral
 \begin{align} \label{ECM_eq}
 \theta(x,t)=\theta(x,0) \exp\left\{-\int_0^t \phi_M(x,s)\, \dd s\right\} \text{ a.e. in } \Omega \times (0,T).
 \end{align}
We will employ equation \eqref{ECM_eq} going forward. Next we eliminate the viable cell volume fraction $\phi_V$ from the system by expressing it in terms of $\phi_T$ and $\phi_N$, i.e. $\phi_V=\phi_T-\phi_N$, which yields the system
\begin{equation} \begin{aligned}
\partial_t \phi_T &=\begin{multlined}[t]  \div (m_T(\phi_T, \phi_N) \nabla \mu) -\div(J_\alpha(\phi_T,\phi_N,\theta))  + \phi_\sigma  f_1(\phi_T,\phi_N)-\lambda^{\textup{apo}}_T\phi_T  -\lambda_N^{\textup{dec}} \phi_N,\end{multlined} \\
\mu &= \Psi'(\phi_T)-\eps^2_T \Delta \phi_T-\chi_C \phi_\sigma+\delta_T \phi_T, \\
\partial_t \phi_N &= \mathscr{S}(\sigma_{VN}-\phi_\sigma) f_2(\phi_T,\phi_N) - \lambda_N^{\textup{deg}} \phi_N, \\
\partial_t \phi_\sigma &=    \div\big(D_\sigma(\phi_\sigma)(\delta_\sigma^{-1} \nabla \phi_\sigma-\chi_C \nabla \phi_T)\big)  + (\phi_T-\phi_N) f_3(\phi_\sigma) , \\
\partial_t \phi_M &=\div(D_M(\phi_M) \nabla \phi_M) +  \theta f_4(\phi_T,\phi_N,\phi_\sigma,\phi_M) - \lambda_M^{\textup{pro}} \phi_M  , \\
\theta(x,t)&=\theta(x,0) \exp\left\{-\int_0^t f_5(\phi_M(x,s))\, \dd s\right\}, 
\end{aligned} \label{mod_problem_loc} \end{equation}
where $\lambda_N^{\textup{dec}}:=\lambda_N^{\textup{deg}}-\lambda_T^{\textup{apo}}$. Note that we have additionally modified the equation for $\phi_N$ by introducing the Sigmoid function $\mathscr{S}$ as a continuous approximation of the Heaviside step function $\mathscr{H}$. This modification is necessary to derive $H^1$-estimates in space of the necrotic tumor volume fraction $\phi_N$. Furthermore, we have generalized the right-hand side terms in \eqref{1}, \eqref{mu}, \eqref{N}--\eqref{MDE}, \eqref{ECM_eq}  by introducing functions $f_i$, $i \in \{1, \dots, 5\}$, on which we make the following assumptions: \\[3mm]

\noindent \textbf{($\text{A7}_{\text{nonloc}}$)} The functions $f_1 \in C_b(\mathbb{R}^2)$, $f_2 \in \text{Lip}(\mathbb{R}^2)\cap PC^1(\mathbb{R}^2)$, $f_3 \in C_b(\mathbb{R})$, $f_4 \in C_b(\mathbb{R}^4)$, and $f_5 \in C_b(\mathbb{R};\mathbb{R}_{\geq 0})$ satisfy 
$$(\exists f_\infty, \bar{f}_\infty>0)\,(\forall x): \quad |f_i(x)| \leq f_\infty, \ \forall i \in \{1,\dots,5\},\quad |D_x f_2(x)|\leq \bar{f}_\infty \text{ a.e.}, $$ 
in the case of the nonlocal model ($\alpha=\text{nonloc}$), or\\

\noindent \textbf{($\text{A7}_{\text{loc}}$)} Let ($\text{A7}_{\text{nonloc}}$) hold. Additionally, let $f_5 \in \text{Lip}(\mathbb{R};\mathbb{R}_{\geq 0})$ such that  $|D_x f_5(x)|\leq \bar{f}_\infty$ a.e., \\

\noindent in case of the local model ($\alpha=\text{loc}$). 

Here, $PC^1$ denotes the space of piecewise continuously differentiable functions, which ensure together with Lipschitz continuity the validity of the chain rule in the situation of a composition with a vector-valued Sobolev function; see \cite{murat2003chain,leoni2007necessary}. We note that the assumption on $f_5$ is strengthened in the local case from continuity to Lipschitz continuity. Since Lipschitz continuous functions are almost everywhere differentiable, the expression $D_x f_5$ is well-defined a.e. for $f_5 \in \text{Lip}(\mathbb{R};\mathbb{R}_{\geq 0})$.

In Section \ref{Sec_FEM}, we give specific and practically relevant examples of functions $f_1$, \dots, $f_5$, which satisfy the assumptions given in ($\text{A7}_\alpha$) and relate the system (\ref{mod_problem_loc}) to the model given by (\ref{1}), (\ref{mu}), (\ref{N})--(\ref{ECM}).


We couple the system of equations \eqref{mod_problem_loc} to the initial data and homogeneous Neumann boundary conditions
\begin{align} \label{Initial_bnd_data_loc}
\begin{cases}
\partial_n \phi_T = \partial_n \phi_\sigma = \partial_n \phi_M = m_T(\phi_T, \phi_N) \partial_n \mu-J_\alpha(\phi_T, \phi_N, \theta) \cdot n =  0 \quad \text{ on } \partial \Omega \times (0,T), \\
(\phi_T, \phi_N, \phi_\sigma, \phi_M) \vert_{t=0}= (\phi_{T,0}, \phi_{N,0}, \phi_{\sigma,0}, \phi_{M,0}).
\end{cases}
\end{align}

We next define the notion of a weak solution of our system. 
\begin{definition}[Weak solution] \label{Definition_loc} Let $\alpha \in \{\textup{loc},\textup{nonloc}\}$ and $\theta_0 \in X_\alpha$, with $X_{\alpha}$ defined as in \eqref{X_alpha}. We call $(\phi_T, \mu, \phi_N, \phi_\sigma, \phi_M, \theta)$ a weak solution of the initial-boundary value problem \eqref{mod_problem_loc}, \eqref{Initial_bnd_data_loc} if 
\begin{equation*}
    \begin{aligned}
&    \phi_T \in L^\infty(0,T;H^1) \cap H^1(0,T; (H^1)'),\quad \mu \in L^2(0,T; H^1), \quad \phi_N \in L^\infty(0,T;H^1)\cap H^1(0,T;L^2), \\
    & \phi_\sigma, \phi_M \in L^2(0,T; H^1) \cap L^\infty(0,T; L^2) \cap H^1(0,T; (H^1)'), \\
    & \theta \in \begin{cases} W^{1, \infty}(0,T; L^\infty) \cap H^1(0,T; H^1)\ & \text{for } \ \alpha=\textup{loc}, \\ W^{1,\infty}(0,T;L^2)\ & \text{for } \ \alpha=\textup{nonloc}, \end{cases}
    \end{aligned}
\end{equation*}
and it holds that
\begin{subequations} \label{Loc:SystemContinuous}
\begin{align}
&\begin{multlined}[t] \langle \partial_t \phi_T, \varphi_1 \rangle + (m_T(\phi_T, \phi_N) \nabla \mu , \nabla  \varphi_1) -(J_\alpha(\phi_T,\phi_N,\theta),\nabla \varphi_1)\\[2mm] - (\phi_\sigma f_1(\phi_T,\phi_N), \varphi_1 ) + (\lambda_T^{\textup{apo}} \phi_T +\lambda_N^{\textup{dec}} \phi_N,\varphi_1)=0,  \end{multlined} \label{cont_loc:1}\\[2mm]
& -(\mu , \varphi_2 )+(\Psi'(\phi_T),  \varphi_2)  +\eps_T^2 (\nabla \phi_T, \nabla  \varphi_2) - \chi_C (\phi_\sigma,  \varphi_2) +\delta_T (\phi_T,\varphi_2)=0 , \label{cont_loc:2}\\[2mm]
& ( \partial_t \phi_N, \varphi_3 ) - (\mathscr{S}(\sigma_{VN}-\phi_\sigma) f_2(\phi_T,\phi_N), \varphi_3)+\lambda_N^{\textup{deg}} (\phi_N,\varphi_3)=0, \label{cont_loc:3} \\[2mm]
& \langle \partial_t \phi_\sigma,  \varphi_4  \rangle +( D_\sigma(\phi_\sigma) \nabla \varphi_4, \delta_\sigma^{-1} \nabla \phi_{\sigma} - \chi_C \nabla \phi_T)  + ((\phi_T-\phi_N) f_3(\phi_\sigma) , \varphi_4 )=0, \label{cont_loc:4} \\[2mm]
&  \langle \partial_t \phi_M,  \varphi_5  \rangle + (D_M(\phi_M) \nabla \phi_M , \nabla  \varphi_5) - (\theta f_4(\phi_T, \phi_N, \phi_M,\phi_\sigma), \varphi_5) +\lambda_M^{\textup{pro}} (\phi_M, \varphi_5 ) =0, \label{cont_loc:5}
\end{align}
\end{subequations}
a.e. in time, for all test functions $ \varphi_1, \varphi_2, \varphi_4, \varphi_5 \in H^1$, $\varphi_3 \in L^2$, and 
\begin{equation}  \theta(x,t)=\theta_0(x) \exp\left\{-\int_0^t f_5(\phi_M(x,s))\, \dd s\right\} \text{ a.e. in } \Omega \times (0,T) \label{cont_loc:6}, \tag{4.4f} 
\end{equation}
 where
$$(\phi_T,\ \phi_N,\ \phi_\sigma,\ \phi_M) \vert_{t=0}= (\phi_{T,0}, \ \phi_{N,0},\ \phi_{\sigma,0}, \ \phi_{M,0}).$$
\end{definition}

\subsection{Existence of solutions} \label{Section:AnalysisLocal}
Our first goal is to prove existence of solutions for the local and nonlocal model. 
\begin{theorem}[Existence of weak solutions] \label{theorem:ExistenceLocal} Let $\alpha \in \{\textup{loc},\textup{nonloc}\}$ and $\theta_0 \in X_\alpha$, with $X_{\alpha}$ defined as in \eqref{X_alpha}. Furthermore, let assumptions \textup{(A1)}--\textup{(A6)}, $(\textup{A7}_\alpha)$ hold and let the initial data have the following regularity
$$ \phi_{T,0} \in H^1,\ \phi_{N,0} \in H^1,\ \phi_{\sigma,0} \in L^2,\ \phi_{M,0} \in L^2.$$
Then there exists a solution $(\phi_T, \mu, \phi_N, \phi_\sigma, \phi_M, \theta)$ of the problem \eqref{mod_problem_loc} in the sense of Definition~\ref{Definition_loc}. Additionally, the following energy estimate holds
\begin{equation*}\begin{aligned} 
&\begin{multlined}  \|\phi_T\|_{L^\infty H^1}^2+ \| \mu\|_{L^2 H^1}^2 + \|\phi_N\|^2_{L^\infty H^1} + \|\phi_\sigma\|_{L^\infty L^2}^2 + \|\phi_\sigma\|_{L^2H^1}^2 \\ +\|\phi_M\|^2_{L^\infty L^2} + \|\phi_M\|_{L^2H^1}^2 +\|\theta\|_{L^\infty X_\alpha}^2 \leq C(T) \, (1+\textup{IC}), \end{multlined}
\end{aligned}
\end{equation*}
where
$$\textup{IC}=|\phi_{T,0}|^2_{H^1}+ |\phi_{N,0}|^2_{H^1} +|\phi_{\sigma,0}|^2_{L^2}  + |\phi_{M,0}|^2_{L^2} + |\theta_0|^2_{X_\alpha}.
$$

\end{theorem}

\subsection{Galerkin approximations in space} 
To prove existence of solutions, we employ Galerkin approximations in space, following the strategy in~\cite{fritz2018unsteady, garcke2017well, garcke2016cahn}. We construct approximate solutions by considering eigenfunctions $\{w_k\}_{k \in \mathbb{N}}$ of the Neumann-Laplacian:
\begin{equation}
\begin{aligned}
\begin{cases}
-\Delta w_k= \lambda _k w_k &\quad \text{in } \Omega,\\[1mm]
\dfrac{\partial w_k}{\partial n}= 0 &\quad \text{on } \partial \Omega. 
\end{cases}
\end{aligned}
\end{equation}
It is known that the eigenfunctions of the Neumann-Laplacian form an orthonormal basis of $L^2$ and an orthogonal basis of $H^1$; cf.~\cite[Theorem II.6.6]{boyer2012mathematical}. We then define the discrete space by 
\begin{align} \label{Vn}
V_n= \textup{span}\{w_1, \dots, w_n\}.
\end{align}
We seek approximate solutions of the form 
\begin{equation} \label{GalerkinApproximations}
\begin{aligned}
\phi_T^n (x,t) &= \sum_{j=1}^n \alpha_j(t) w_j(x),
&&\quad \phi_N^n (x,t) =  \sum_{j=1}^n \beta_j(t) w_j(x), \\
\quad \phi_{\sigma}^n (x,t) &= \sum_{j=1}^n \gamma_j(t) w_j(x), 
&&\quad \phi_M^n (x,t) =  \sum_{j=1}^n \delta_j(t) w_j(x), 
&& 
\end{aligned}
\end{equation}
where  $\alpha_j, \beta_j, \gamma_j, \delta_j : (0,T) \to \mathbb{R}$ will be determined by a system of ordinary differential equations. We choose the approximations of the initial conditions as follows:
\begin{equation} \label{GalerkinInitial}
\begin{alignedat}{2} \phi_{T, 0}^n &=\, \Pi_{V_n} \phi_{T,0}, \quad \phi_{N, 0}^n&&= \Pi_{V_n} \phi_{N,0}, \\
\phi_{\sigma, 0}^n &=\, \Pi_{V_n} \phi_{\sigma, 0} , \quad \phi_{M,0}^n &&= \Pi_{V_n} \phi_{M,0}.
\end{alignedat}
\end{equation}
Above, $\Pi_{V_n}$ denotes the $L^2$ projection operator: $(\Pi_{V_n} u, v)=(u ,v)$ for all $v \in V_n$. Note that for all $n \in \mathbb{N}$ it holds that
\begin{equation} \label{bounds_approx_initial}
\begin{alignedat}{2}
|\phi_{T, 0}^n|_{H^1} &\leq |\phi_{T, 0}|_{H^1}, \quad &&|\phi_{N, 0}^n|_{H^1} \leq |\phi_{N, 0}|_{H^1}, \\ |\phi_{\sigma, 0}^n|_{L^2} &\leq |\phi_{\sigma, 0}|_{L^2},\quad &&|\phi_{M,0}^n|_{L^2} \leq |\phi_{M,0}|_{L^2};
\end{alignedat}
\end{equation}
see, e.g.,~\cite[Lemma 7.5]{robinson2001infinite}. \\
\indent The semi-discretization of the problem \eqref{mod_problem_loc} is then given by
\begin{subequations} \label{GalerkinLocal:System}
\begin{align}
&\begin{multlined}[t] (\partial_t \phi^n_T, \varphi^n ) + (m_T(\phi^n_T, \phi^n_N) \nabla \mu^n , \nabla  \varphi^n ) -(J_\alpha(\phi^n_T, \phi^n_{N},\theta^n),\nabla \varphi^n )\\[0mm] - (\phi^n_\sigma f_1(\phi_T^n,\phi_N^n), \varphi^n ) + (\lambda_T^{\textup{apo}}\phi^n_T+ \lambda_N^{\textup{dec}}\phi^n_N,  \varphi^n )=0, \end{multlined} \label{loc:1} \\[2mm]
& -(\mu^n , \varphi^n )+(\Psi'(\phi^n_T),  \varphi^n )  +\eps_T^2 (\nabla \phi^n_T, \nabla  \varphi^n ) - \chi_C (\phi^n_\sigma,  \varphi^n ) + \delta_T (\phi_T^n,\varphi^n)=0 , \label{loc:2} \\[2mm]
& (\partial_t \phi^n_N, \varphi^n ) -   (\mathscr{S}(\sigma_{VN}-\phi^n_\sigma) f_2(\phi^n_T,\phi^n_N), \varphi^n ) + \lambda_N^{\textup{deg}} (\phi_N^n,\varphi^n)=0,\label{loc:3} \\[2mm]
& ( \partial_t \phi^n_\sigma,  \varphi^n  ) +( D_\sigma(\phi_\sigma^n) (\delta_\sigma^{-1} \nabla \phi^n_{\sigma} - \chi_C \nabla \phi^n_T), \nabla \varphi^n )  +((\phi^n_T-\phi^n_N) f_3(\phi_\sigma^n) , \varphi^n  )=0, \label{loc:4} \\[2mm]
& \begin{multlined}[t] (\partial_t \phi^n_M,  \varphi^n ) + (D_M(\phi_M^n) \nabla \phi^n_M , \nabla \varphi^n )- (\theta^n f_4(\phi_T^n, \phi_N^n, \phi_\sigma^n, \phi_M^n), \varphi^n )  +\lambda^{\textup{pro}}_M (\phi^n_M, \varphi^n ) =0, \end{multlined} \label{loc:5} \\[2mm]
& \theta^n(x,t)=\theta_0(x) \exp\bigg\{-\int_0^t f_5(\phi_M^n(x,s))\, \dd s\bigg\}, \label{loc:6} 
\end{align}
\end{subequations}
for all $\varphi^n \in V_n$, with
\begin{equation} \label{GalerkinLocal:InitialCon}
\begin{aligned}
(\phi^n_T,\ \phi^n_N,\ \phi^n_\sigma,\ \phi^n_M) \vert_{t=0}= (\phi^n_{T,0},\ \phi^n_{N,0},\ \phi^n_{\sigma,0},\ \phi^n_{M,0}).
\end{aligned}
\end{equation}
The system \eqref{GalerkinLocal:System}--\eqref{GalerkinLocal:InitialCon} is equivalent to an initial value problem for a system of integro-differential equations for the unknown function $\xi=(\xi_1,\dots,\xi_n)$, where $\xi_i=(\alpha_i, \beta_i, \gamma_i, \delta_i)$, $i \in \{1,\dots n\}$, which can be equivalently written as
\begin{equation*}
\begin{aligned}
        \partial_t \xi_i(t) &= F^i(t,\xi(t),K\xi(t))\\ &=\widehat{F}^i(t,\xi(t)) + \widetilde{F}^i(t,\xi(t),K\xi(t)),
        \end{aligned}
        \end{equation*}
for all $i \in \{1,\dots,n\}$, where $K\xi(t)=\int_0^t f_5(\sum_{j=1}^n \delta_j(s) w_j) \dd s$ and 
\begin{equation*}
\begin{aligned}
       \widetilde{F}^i_1&=\int_\Omega g\big(\phi_T^n(x,t),\phi_N^n(x,t)\big) G\big(\theta_0 \exp\{-K\xi(t)\}\big) \cdot \nabla w_i \, \dd x,  \\
       \widetilde{F}^i_2 &= \widetilde{F}^i_3 =0,
        \\ \widetilde{F}^i_4 &=\int_\Omega \theta_0 \exp\{-K\xi(t)\}  f_4\big(\phi_T^n(x,t),\phi_N^n(x,t),\phi_\sigma^n(x,t),\phi_M^n(x,t)\big) w_i \, \dd x . 
 \end{aligned}
\end{equation*}
We note that the given functions $\Psi'$, $m_T$, $D_M$, $D_\sigma$, $f_1, \dots, f_5$ are all continuous. Therefore, on account of an extension of the Cauchy--Peano theorem for integro-differential equations, see Theorem \ref{app:theorem} in Appendix below, we obtain a solution of \eqref{GalerkinLocal:System}--\eqref{GalerkinLocal:InitialCon} such that
\begin{equation*}
\begin{aligned}
(\phi^n_T,\ \mu^n,\ \phi^n_N,\ \phi^n_\sigma,\ \phi^n_M, \ \theta^n) \in \begin{multlined}[t] C^1([0,T_n]; V_n) \times C([0,T_n]; V_n) \times \left( C^1([0,T_n]; V_n) \right)^3 \times C([0,T_n];X_\alpha),
\end{multlined}
\end{aligned}
\end{equation*}
for sufficiently short time $T_n \leq T$. The upcoming energy estimate will allow us to extend the existence interval to $[0,T]$.
\subsection{Energy estimates}
Our next goal is to derive an energy estimate for solutions of \eqref{loc:1}-\eqref{loc:6}, \eqref{GalerkinLocal:InitialCon} that is uniform with respect to $n$. To this end, we test the equations \eqref{loc:1}-\eqref{loc:5} with different test functions. \\

\noindent \textbf{Estimates for \mathversion{bold}$\theta^n$.}\\[4mm]
\indent Since the integral and the exponential function are continuous and the function $f_5$ is non-negative by assumption (A7$_\alpha$), we conclude
\begin{equation} \label{estimate:theta:initial}
\begin{aligned}
|\theta^n(t)|_{L^p} \leq |\theta_0|_{L^p}, 
\end{aligned}
\end{equation}
for all $t \leq T_n$. Above, $p\in[1,2]$ for $\theta_0 \in X_\text{nonloc}=L^2$ and $p \in [1,\infty]$ for $\theta_0 \in X_\text{loc}=H^1 \cap L^\infty$.\\
\indent In the nonlocal case, this uniform bound of $\theta^n$ is already enough for the upcoming energy estimates. We recall that the term $J_\alpha(\phi_T^n,\phi_N^n,\theta^n)$ in the equation for $\phi_T^n$ can be expressed as $g(\phi_T^n,\phi_N^n) G(\theta^n)$ on account of assumption (A6). Since the operator $G$ requires an argument in $X_\alpha$, we still have to derive an estimate of $\theta^n$ in $H^1$ when $\alpha=\text{loc}.$  \\

\noindent \underline{$\alpha=\text{loc}$}: By the product rule and the chain rule for the composition of a bounded Lipschitz continuous function with a Sobolev function, see \cite[Theorem 2.1.11]{ziemer2012weakly}, we further infer that
\begin{equation}
\begin{aligned}
\nabla \theta^n(t) =   \left(
\nabla \theta_0 - \theta_0 \int_0^t f_5'(\phi^n_M(s)) \nabla \phi^n_M(s)\, \dd s\right) \cdot \exp\left \{-\int_0^t f_5(\phi^n_M(s))\, \dd s \right\}, 
\end{aligned} \label{Eq:thetagradient}
\end{equation} 
for all $t \in [0,T_n]$. From here, using assumption (A7$_{\textup{loc}}$), we obtain the bound for the gradient of the ECM density
\begin{equation*} \label{est_grad_ecm}
\begin{aligned}
|\nabla \theta^n(t) |_{L^2} \leq  |\nabla \theta_0|_{L^2} +  |\theta_0|_{L^\infty} \sqrt{T_n} \, \bar{f}_\infty \,   \|\nabla \phi^n_M\|_{L^2_t L^2},  
\end{aligned}
\end{equation*}
for $t \in [0, T_n]$, where we have used the abbreviation $L^2_t L^2$ for $L^2(0,t; L^2(\Omega))$. By combining this estimate and the estimate \eqref{estimate:theta:initial} with $p=2$, it follows for all $t \in [0,T_n]$ that
\begin{equation} \label{est_H1_ecm}
\begin{aligned} 
|\theta^n(t)|_{H^1}^2 &\leq 2 |\theta_0|_{H^1}^2 + 2 T \bar{f}^2_\infty |\theta_0|_{L^\infty}^2  \| \nabla \phi^n_M\|_{L^2_tL^2}^2.
\end{aligned}
\end{equation}

\noindent \textbf{Estimates for \mathversion{bold}$\phi_M^n$.}\\[4mm]
\indent Testing equation \eqref{loc:5} with $\varphi^n=\phi_M^n(t) \in V_n$ 
and recalling assumption (A3) as well as the bound \eqref{estimate:theta:initial} for $\theta^n$ yields
\begin{equation} \label{estimate:3}
\begin{aligned}
\frac{1}{2} \frac{\text{d}}{\text{d} t} |\phi_M^n|_{L^2}^2 +  D_0 |\nabla \phi_M^n|^2_{L^2} +  \lambda_M^{\textup{pro}} |\phi_M^n|_{L^2}^2 
 \leq \frac{f_\infty}{2} ( |\theta_0|_{L^2}^2 + |\phi_M^n|_{L^2}^2).
\end{aligned} 
\end{equation}
\noindent After integrating over $(0,t)$, where $t \leq T_n$, we conclude by the Gronwall lemma that
\begin{equation} \label{est_phi_MDE}
\begin{aligned} |\phi_M^n(t)|^2_{L^2} + \|\nabla \phi_M^n\|_{L^2_t L^2}^2 + \|\phi_M^n\|_{L^2_t L^2}^2 
\leq C(T_n) \, \big(|\phi^n_{M,0}|_{L^2}^2 + |\theta_0|_{L^2}^2\big).
\end{aligned}
\end{equation}
Adding to this estimate \eqref{estimate:theta:initial} and \eqref{est_H1_ecm} for $\alpha=\text{loc}$, or  \eqref{estimate:theta:initial} for $\alpha=\text{nonloc}$, we get 
\begin{equation}\begin{aligned} |\theta^n(t)|_{X_\alpha}^2 +|\phi_M^n(t)|^2_{L^2} + \|\phi_M^n\|^2_{L^2_t H^1} 
\leq C(T_n) \, \big(|\theta_0|_{X_\alpha}^2 + |\phi_{M,0}|_{L^2}^2\big).
\end{aligned} \label{loc:estimate2}
\end{equation}
Note that above we have also employed the uniform bound for the approximate initial data $\phi_{M,0}^n$ given in \eqref{bounds_approx_initial}.\\

\noindent \textbf{Estimates for \mathversion{bold}$\phi_T^n,\mu^n,\phi_\sigma^n$.}\\[4mm]

Testing equation \eqref{loc:1} with $\mu^n(t)+\chi_C \phi_\sigma^n(t)$, equation \eqref{loc:2} with $-\partial_t \phi_T^n(t)$, and equation \eqref{loc:4} with $K_1 \phi_\sigma^n(t)$, where $K_1>0$, and adding the resulting equations yields
\begin{equation}
\begin{aligned}
& \frac{\text{d}}{\text{d} t} \bigg[ |\Psi(\phi_T^n)|_{L^1} + \frac{\varepsilon_T^2}{2} |\nabla \phi_T^n|_{L^2}^2+\frac{\delta_T}{2} |\phi_T^n|_{L^2}^2  +\frac{K_1}{2} |\phi_\sigma^n|_{L^2}^2 \bigg] 
\\ &+ \left|\sqrt{m_T(\phi_T^n, \phi_N^n)} \, \nabla \mu^n \right|^2_{L^2} + K_1\delta_\sigma^{-1} \left | \sqrt{D_\sigma(\phi_\sigma^n)} \, \nabla \phi_\sigma^n \right|^2_{L^2} \\
 =&\, -\chi_C (m_T(\phi_T^n, \phi_N^n) \nabla \mu^n, \nabla \phi_\sigma^n)+(J_\alpha(\phi_T^n,\phi_N^n,\theta^n), \nabla (\mu^n+\chi_C \phi_\sigma^n)) \\
 &+ (f_1(\phi_T^n,\phi_N^n)\phi_\sigma^n -\lambda_T^{\textup{apo}} \phi_T^n -\lambda_N^{\textup{dec}}  \phi_N^n, \mu^n + \chi_C \phi_\sigma^n) \\
& +K_1\chi_C (D_\sigma(\phi_\sigma^n)\nabla \phi_T^n,\nabla \phi_\sigma^n) - K_1 ((\phi^n_T-\phi^n_N)f_3(\phi_\sigma^n), \phi_\sigma^n)
=:  \textup{RHS}.
\end{aligned} \label{loc:estimate1}
\end{equation}
We can then estimate the right-hand side of \eqref{loc:estimate1} by using assumptions (A2), (A6), (A7$_\alpha$), and Hölder's inequality as follows
\begin{equation} \label{rhs:1}
\begin{aligned}
\text{RHS} \leq& \, \chi_C m_\infty |\nabla \mu^n|_{L^2} |\nabla \phi_\sigma^n|_{L^2} + C |\theta^n|_{X_\alpha} (|\nabla \mu^n|_{L^2} + \chi_C |\nabla \phi_\sigma^n|_{L^2}) \\
& + (f_\infty |\phi_\sigma^n|_{L^2}+\lambda_T^{\textup{apo}} |\phi_T^n|_{L^2}+|\lambda_N^{\textup{dec}}|\cdot  |\phi_N^n|_{L^2})(|\mu^n|_{L^2}+\chi_C |\phi_\sigma^n|_{L^2}) 
\\
& + K_1 \chi_C D_\infty |\nabla \phi_T^n|_{L^2}|\nabla \phi_\sigma^n|_{L^2} + K_1 f_\infty(|\phi^n_T|_{L^2}+|\phi_N^n|_{L^2}) |\phi_\sigma^n|_{L^2}.
\end{aligned}\end{equation}
We note that we need a bound on $|\mu^n|_{L^2}$ to further estimate \eqref{rhs:1}. Testing \eqref{loc:2} with $1 \in H^1$ and taking into account assumption (A7) on the function $\Psi$ results in 
\begin{equation*}
\begin{aligned}
|\mu^n|_{L^1} \leq&\, \int_{\Omega}|\Psi^\prime (\phi^n_T)| \, \textup{d}x + \chi_C |\phi_\sigma^n|_{L^1} +\delta_T |\phi_T^n|_{L^1} \\
\leq& \, R_3\big(|\Omega|+|\phi^n_T|_{L^1}\big) + \chi_C |\phi_\sigma^n|_{L^1} +\delta_T |\phi_T^n|_{L^1} \\
\leq& \, R_3|\Omega|+(R_3+\delta_T)|\Omega|^{1/2}|\phi^n_T|_{L^2} + \chi_C|\Omega|^{1/2} |\phi_\sigma^n|_{L^2};
\end{aligned}
\end{equation*}
we refer also to~\cite{garcke2016global} where a similar argument is employed. By the Poincar\'e inequality \eqref{Poincare_1}, we then conclude 
\begin{equation} \label{estimate:mu}
\begin{aligned}
|\mu^n|_{L^2} &\leq |\mu^n-\overline{\mu^n}|_{L^2}+|\overline{\mu^n}|_{L^2}
\leq C_{\textup{P}}|\nabla \mu^n|_{L^2}+\frac{1}{|\Omega|}|\mu^n|_{L^1} \\
&\leq C_{\text{P}} |\nabla \mu^n|_{L^2} + R_3 + (R_3+\delta_T) |\Omega|^{-1/2} |\phi^n_T|_{L^2}+\chi_C|\Omega|^{-1/2} |\phi_\sigma^n|_{L^2}.
\end{aligned}
\end{equation}
Therefore, by using (\ref{estimate:mu}), we can further estimate the right-hand side of (\ref{loc:estimate1}) as follows:
\begin{equation*}
\begin{aligned}
\text{RHS} \leq& \, \chi_C m_\infty |\nabla \mu^n|_{L^2} |\nabla \phi_\sigma^n|_{L^2} +C |\theta^n|_{X_\alpha} (|\nabla \mu^n|_{L^2} + \chi_C |\nabla \phi_\sigma^n|_{L^2}) \\
& + (f_\infty |\phi_\sigma^n|_{L^2}+\lambda_T^{\textup{apo}}|\phi_T^n|_{L^2}+|\lambda_N^{\textup{dec}}|\cdot|\phi_N^n|_{L^2})\{C_{\text{P}} |\nabla \mu^n|_{L^2} + R_3 \\
&+ (R_3+\delta_T) |\Omega|^{-1/2} |\phi^n_T|_{L^2}+\chi_C|\Omega|^{-1/2} |\phi_\sigma^n|_{L^2}+\chi_C |\phi_\sigma^n|_{L^2}\} 
\\
& + K_1 \chi_C D_\infty |\nabla \phi_T^n|_{L^2}|\nabla \phi_\sigma^n|_{L^2} + K_1 f_\infty(|\phi^n_T|_{L^2}+|\phi_N^n|_{L^2}) |\phi_\sigma^n|_{L^2}.
\end{aligned}\end{equation*}
By employing Young's inequality, we get
\begin{equation*}
\begin{aligned}
\text{RHS} &\leq \left(\frac{m_0}{2}+4\eps\right) |\nabla \mu^n|_{L^2}^2 + \left(\frac{\chi_C^2 m_\infty^2 }{2m_0}+\eps\right) |\nabla \phi_\sigma^n|_{L^2}^2\\ &\quad + C \big( 1 +|\phi_\sigma^n|_{L^2}^2 + |\phi_T^n|_{L^2}^2 + |\nabla \phi_T^n|_{L^2}^2 + |\phi_N^n|_{L^2}^2 + |\theta^n|_{X_\alpha}^2 \big),
\end{aligned}\end{equation*}
 where $\varepsilon>0$. Introducing this estimate of the right-hand side into (\ref{loc:estimate1}) and recalling assumptions (A2) and (A3) yields
\begin{equation}
\begin{aligned}
& \frac{\text{d}}{\text{d} t} \bigg[ |\Psi(\phi_T^n)|_{L^1} + \frac{\varepsilon_T^2}{2} |\nabla \phi_T^n|_{L^2}^2+\frac{\delta_T}{2}|\phi_T^n|_{L^2}^2  +\frac{K_1}{2} |\phi_\sigma^n|_{L^2}^2 \bigg] \\ &\quad + \left( \frac{m_0}{2} - 4\eps \right) |\nabla \mu^n|_{L^2}^2 + \left(K_1 D_0 \delta_\sigma^{-1}-\frac{\chi_C^2 m_\infty^2 }{2m_0}-\eps\right) |\nabla \phi_\sigma^n|^2_{L^2} \\
 &\leq  C \big( 1+|\theta^n|_{X_\alpha}^2 +|\phi_\sigma^n|_{L^2}^2 + |\Psi(\phi_T^n)|_{L^1} + |\nabla \phi_T^n|_{L^2}^2 + |\phi_N^n|_{L^2}^2 \big),
\end{aligned} \label{estimate:1}
\end{equation}
where we have first picked $\varepsilon \in (0, m_0/8)$ and then chosen $K_1$ sufficiently large so that \begin{equation} K_2:=K_1 D_0 \delta_\sigma^{-1}-\frac{\chi_C^2 m_\infty^2 }{2m_0}-\eps>0.
\label{Eq_ConstantK2}
\end{equation}

\noindent \textbf{Estimates for \mathversion{bold}$\phi_N^n$.}\\[4mm]
\indent Testing equation \eqref{loc:3} with $\phi_N^n(t) \in V_n$ yields, after some standard manipulations,
\begin{equation} \label{estimate:2}
\frac12 \frac{\text{d}}{\text{d}t} |\phi_N^n|_{L^2}^2 +\lambda^{\textup{deg}}_N |\phi_N^n|_{L^2}^2 \leq \frac{f_\infty}{2} \big(|\phi_N|_{L^2}^2+|\Omega|\big).
\end{equation}

\noindent This estimate would be enough to absorb the $\phi_N^n$ term on the right hand side of (\ref{estimate:1}). However, we here also derive an estimate of $\phi_N^n$ in the space $L^\infty(0,T;H^1)$, which will enable us to perform the limit process as $n \rightarrow \infty$ later on. Testing (\ref{loc:3}) with $-\Delta \phi_N^n(t) \in V_n$ and performing integration by parts results in
\begin{equation*}
\begin{aligned}
    \frac12 \frac{\dd}{\dd t} |\nabla \phi_N^n|_{L^2}^2 + \lambda_N^{\text{deg}} |\nabla \phi_N^n|_{L^2}^2 = &\,  \lambda_{VN} (\nabla (\mathscr{S}(\sigma_{VN}-\phi_\sigma^n)  f_2(\phi_T^n,\phi_N^n) ), \nabla \phi_N^n) \\
    = &\,  (-\nabla \phi_\sigma^n \mathscr{S}'(\sigma_{VN}-\phi_\sigma) f_2(\phi_T^n,\phi_N^n),\nabla \phi_N^n) \\
     &+ (\mathscr{S}(\sigma_{VN}-\phi_\sigma^n) \nabla \phi_T^n \partial_1 f_2(\phi_T^n,\phi_N^n),\nabla \phi_N^n)\\
      &   + (\mathscr{S}(\sigma_{VN}-\phi_\sigma^n) \nabla \phi_N^n \partial_2 f_2(\phi_T^n,\phi_N^n),\nabla \phi_N^n),
    \end{aligned}
\end{equation*}
where we have applied the chain rule for the composition of a bounded Lipschitz, piecewise continuously differentiable function and a vector-valued Sobolev function; see~\cite{murat2003chain,leoni2007necessary}. After employing the same type of arguments as before, this estimate implies that
\begin{equation} \label{estimate:2b}
\frac{\dd}{\dd t} |\nabla \phi_N^n|_{L^2}^2 +  |\nabla \phi_N^n|_{L^2}^2 \leq  \frac{K_2}{2} |\nabla \phi_\sigma^n|_{L^2}^2 +C(K_2)\cdot \big(|\nabla \phi_N^n|^2_{L^2}+|\nabla \phi_T^n|_{L^2}^2\big),
\end{equation}
where $K_2$ is the positive constant in (\ref{Eq_ConstantK2}).\\

\noindent \textbf{Final energy estimate.} \\[4mm]
\indent Combining the upper bounds (\ref{estimate:mu}), \eqref{estimate:1}, \eqref{estimate:2}, and (\ref{estimate:2b}) yields
\begin{equation}
\begin{aligned}
& \frac{\text{d}}{\text{d} t} \Big[ |\Psi(\phi_T^n)|_{L^1} + |\nabla \phi_T^n|_{L^2}^2 +|\phi_T^n|_{L^2}^2  + |\phi_\sigma^n|_{L^2}^2 + |\phi_N|_{H^1}^2 \Big] +  |\mu^n|_{H^1}^2 +|\nabla \phi_\sigma^n|^2_{L^2} \\
 &\leq  C \big( 1+|\theta^n|_{X_\alpha}^2 +|\phi_\sigma^n|_{L^2}^2 + |\Psi(\phi_T^n)|_{L^1} + |\nabla \phi_T^n|_{L^2}^2 + |\phi_N^n|_{H^1}^2 \big).
\end{aligned} \label{estimate:4}
\end{equation}
After integrating \eqref{estimate:4} over $(0,t)$, where $t \leq T_n$ and taking into account estimate (\ref{loc:estimate2}), we have 
\begin{equation} \label{final_est_local}
\begin{aligned} &|\Psi(\phi_T^n(t))|_{L^1} + |\nabla \phi_T^n(t)|_{L^2}^2 + |\phi_\sigma^n(t)|_{L^2}^2 + |\phi_N^n(t)|^2_{H^1}+ \|\mu^n\|_{L^2_t H^1}^2 + \|\nabla \phi^n_\sigma\|_{L^2_t L^2}^2 
\\ &\leq \text{IC}^n + C(T_n) \cdot \left( 1+ |\phi_\sigma^n|_{L^2}^2 + |\Psi(\phi_T^n)|_{L^1} + |\nabla \phi_T^n|_{L^2}^2 + |\phi_N^n|_{H^1}^2 \right).
\end{aligned}
\end{equation}
Above, we have introduced the following constant that depends on the approximate initial data to simplify the notation
$$\text{IC}^n=|\phi^n_{T,0}|^2_{H^1}+|\Psi(\phi^n_{T,0})|_{L^1} +|\phi^n_{\sigma,0}|^2_{L^2} + |\phi^n_{N,0}|^2_{L^2} + |\phi^n_{M,0}|^2_{L^2} + |\theta_0|^2_{X_\alpha}.
$$

We can employ the fact that
$$|\Psi(\phi^n_{T,0})|_{L^1} \leq C+C|\phi^n_{T,0}|^2_{L^2} \leq C+C|\phi_{T,0}|^2_{L^2}, $$
and thus, $\text{IC}^n$ can be estimated in terms of the initial data as follows
$$\text{IC}^n \leq  \text{IC}=|\phi_{T,0}|^2_{H^1}+C+C|\phi_{T,0}|^2_{L^2} +|\phi_{\sigma,0}|^2_{L^2} + |\phi_{N,0}|^2_{H^1} + |\phi_{M,0}|^2_{L^2} + |\theta_0|^2_{X_\alpha},
$$
where the constant $C$ does not depend on $n$. By adding \eqref{loc:estimate2} to \eqref{final_est_local}, applying Gronwall's inequality to the resulting estimate, and taking the supremum over $(0,T_n)$, we get
\begin{equation}
\begin{aligned} 
& \begin{multlined}[t]\|\Psi(\phi_T^n)\|_{L^\infty_t L^1} + \|\nabla \phi_T^n\|_{L^\infty_t L^2}^2 + \|\phi_\sigma^n\|_{L^\infty_t L^2}^2 + \|\phi_N^n\|^2_{L^\infty_t H^1}+\|\theta^n\|_{L^\infty_t X_\alpha}^2\\[1mm] +\|\phi_M^n\|^2_{L^\infty_t L^2} + \| \phi_M^n\|_{L^2_t H^1}^2+ \|\mu^n\|_{L^2_t H^1}^2 + \|\phi^n_\sigma\|_{L^2_t H^1}^2 \end{multlined}
\\[2mm] &\leq C(T) \, (1+\text{IC}),
\end{aligned}
\label{loc:final}
\end{equation}
for all $t\in [0,T_n]$. The right-hand side of this estimate is independent of $T_n$, which allows to extend the existence interval to $[0,T]$; see also \cite[I.6.VI]{walter1998ordinary}.\\
\indent We remark that from \eqref{loc:final} we can get a uniform bound for $\phi_T^n$ in $L^\infty(0,T;H^1)$ by noting that
\begin{equation} \label{est:phi_T_H1}
    \begin{aligned}
    |\phi_T^n(t)|^2_{L^2} \leq& \,  2C_\text{P}|\nabla \phi_T^n(t)|^2_{L^2}+2\frac{1}{|\Omega|^2}|\phi_T^n(t)|^2_{L^1}\\
    \leq&\, 2C_\text{P}|\nabla \phi_T^n(t)|^2_{L^2}+2\frac{1}{|\Omega|^2}\frac{1}{R_1}\big(|\Psi(\phi^n_T(t))|_{L^1}+R_2\big),
    \end{aligned}
\end{equation}
for all $t \in [0,T]$. Above, we have made use of the Poincar\' e inequality \eqref{Poincare_1} and assumption (A5) on the potential $\Psi$.\\

\noindent \textbf{Additional estimates of the time derivatives of \mathversion{bold}$\theta^n$, $\phi^n_N$, $\phi^n_T$, and $\phi_\sigma^n$.} \\[2mm]

The derived energy estimate (\ref{loc:final}) implies the boundedness of the Galerkin solution \linebreak $(\phi_T^n, \mu^n, \phi_N^n, \phi_\sigma^n, \phi_M^n)$ and of $\theta^n$ in appropriate Banach spaces, which in turn implies the weak and weak-$*$ convergence of subsequences. We consider taking the limit $n \to \infty$ in the Galerkin system (\ref{GalerkinLocal:System}). Since the equations in our system are nonlinear in $\phi_T^n$, $\phi_N^n$, $\phi_\sigma^n$ and $\phi_M^n$, we want to acquire strong convergence of the respective subsequences. We can obtain strong convergence from compact embeddings (\ref{Eq_AubinLions}), which requires the boundedness of the respective time derivative. We derive these estimates in this section. \\
\indent Testing equation \eqref{loc:3} with $\partial_t \phi^n_N(t) \in V_n$ and employing Young's inequality yields
\begin{equation} \label{est:phi_N_t}
    \begin{aligned}
    (1-\eps) \|\partial_t \phi^n_N\|_{L^2L^2}^2 + \frac{\lambda_N^\text{deg}}{2} \|\phi_N\|_{L^\infty L^2}^2 \leq C(T,\eps)+\frac{\lambda_N^\text{deg}}{2} |\phi_{N,0}|_{L^2}^2,
    \end{aligned}
\end{equation}
where $\varepsilon \in (0,1)$. \noindent Furthermore, from equation \eqref{loc:4} we find that for all $\varphi \in L^2(0,T;H^1)$ it holds that
\begin{align*}
&\int_0^T \int_{\Omega} \partial_t \phi_\sigma^n \varphi \, \textup{d}x \textup{d}t\\
\leq& \,  \big(D_\infty \delta^{-1}_{\sigma} \|\nabla \phi_\sigma^n\|_{L^2L^2}+D_\infty \chi_C\|\nabla \phi_T^n\|_{L^2L^2}+f_\infty \big(\|\phi_N^n\|_{L^2L^2}+\|\phi_T^n\|_{L^2L^2}\big)\big)\|\varphi\|_{L^2H^1},
\end{align*}
from which we also get that
\begin{equation} \label{est:phi_sigma_t}
\|\partial_t \phi_\sigma^n\|_{L^2 (H^1)'}\leq C \big(\|\nabla \phi_{\sigma}^n\|_{L^2L^2}+\|\nabla \phi_T^n\|_{L^2L^2}+\|\phi_N^n\|_{L^2L^2}\big),
\end{equation}
where the constant $C>0$ does not depend on $n$. Similarly, from equation \eqref{loc:1} we have
\begin{align*}\int_0^T \int_{\Omega} \partial_t \phi_T^n \, \varphi\, \textup{d}x \textup{d}t \leq&\,\begin{multlined}[t] \big(m_\infty \|\nabla \mu^n\|_{L^2 L^2}+C \|\theta^n\|_{L^2 L^2}+f_\infty \|\phi_\sigma^n\|_{L^2L^2}\\
+\lambda_T^{\textup{apo}}\|\phi_T^n\|_{L^2L^2}+|\lambda_N^{\textup{dec}}|\cdot\|\phi_N^n\|_{L^2L^2}\big)\|\varphi\|_{L^2H^1},\end{multlined}
\end{align*}
and from equation \eqref{loc:5} 
\begin{align*}
\int_0^T \int_{\Omega} \partial_t \phi_M^n \, \varphi \, \textup{d}x \textup{d}t \leq \big(D_\infty \|\nabla \phi_M^n\|_{L^2L^2}+f_\infty \|\theta^n\|_{L^2L^2}+\lambda_M^{\textup{pro}} \|\phi^n_M\|_{L^2L^2}\big)\|\varphi\|_{L^2H^1},
\end{align*}
for all $\varphi \in L^2(0,T;H^1)$. From the above two estimates it follows that
\begin{equation} \label{est:phi_Tsigma_t}
\begin{aligned}
\|\partial_t \phi_T^n\|_{L^2 (H^1)'}\leq&\, C \big( \|\nabla \mu^n\|_{L^2 L^2}\!+\|\nabla \theta^n\|_{L^2L^2}\!+\|\phi_\sigma^n\|_{L^2L^2}\!+\|\phi_T^n\|_{L^2L^2}\!+\|\phi_N^n\|_{L^2L^2}\big), \\
\|\partial_t \phi_M^n\|_{L^2 (H^1)'}\leq&\,C \big( \|\nabla \phi_M^n\|_{L^2 L^2}+\|\theta^n\|_{L^2L^2}+\|\phi_M^n\|_{L^2L^2}\big).
\end{aligned}
\end{equation}

Lastly, we note that from the integral representation of the ECM density $\theta^n$ we can directly derive a uniform bound of $\partial_t \theta^n$ in $L^\infty(0,T;L^2)$ for $\alpha=\text{nonloc}$ and in $L^\infty(0,T;L^\infty)\cap L^2(0,T;H^1)$ for $\alpha=\text{loc}$.

\subsection{Passing to the limit}
On account of the final estimate \eqref{final_est_local} for Galerkin approximations and estimates \eqref{est:phi_T_H1}--\eqref{est:phi_Tsigma_t}, we can conclude that
\begin{equation}\label{loc:boundedweak} \begin{alignedat}{2}
&\{\phi_T^n\}_{n \in \mathbb{N}} &&\text{ is bounded in } L^\infty(0,T;H^1) \cap H^1(0,T;(H^1)'), \\
&\{\mu^n\}_{n \in \mathbb{N}} &&\text{ is bounded in } L^2(0,T;H^1), \\
&\{\phi_N^n\}_{n \in \mathbb{N}} &&\text{ is bounded in } L^\infty(0,T;H^1)\cap H^1(0,T;L^2), \\
&\{\phi_\sigma^n\}_{n \in \mathbb{N}} &&\text{ is bounded in } L^\infty(0,T;L^2) \cap L^2(0,T;H^1) \cap H^1(0,T;(H^1)'),  \\
&\{\phi_M^n\}_{n \in \mathbb{N}} &&\text{ is bounded in } L^\infty(0,T;L^2) \cap L^2(0,T;H^1) \cap H^1(0,T;(H^1)'),  \\
&\{\theta^n\}_{n \in \mathbb{N}} &&\text{ is bounded in } \begin{cases} W^{1 ,\infty}(0,T;L^2), &\alpha=\text{nonloc}, \\  W^{1,\infty}(0,T;L^\infty) \cap H^1(0,T;H^1), &\alpha=\text{loc}, \end{cases}
\end{alignedat}\end{equation}
uniformly with respect to $n$. This implies the existence of weakly/weakly-$*$ converging subsequences, indexed again by $n$, to some limit functions $(\phi_T,\mu,\phi_N,\phi_\sigma,\phi_M,\theta)$ in the respective spaces and the following strong convergences due to the Aubin--Lions Compactness lemma, see (\ref{Eq_AubinLions}),
\begin{equation} \label{loc:strongconv}
\begin{alignedat}{2}
\phi_T^n &\longrightarrow \phi_T &&\text{ strongly in } C([0,T];L^2), \\
\phi_N^n &\longrightarrow \phi_N &&\text{ strongly in } C([0,T];L^2), \\
\phi_\sigma^n &\longrightarrow \phi_\sigma &&\text{ strongly in } L^2(0,T;L^2), \\
\phi_M^n &\longrightarrow \phi_M &&\text{ strongly in } L^2(0,T;L^2), 
\end{alignedat}
\end{equation}
as $n \to \infty$ and the following weak convergence,
\begin{equation} \label{loc:weakconv:theta}
    \theta^n \longweak \theta \text{ weakly in } L^2(0,T;X_\alpha).
\end{equation}

We next show that the limit functions $(\phi_T, \mu, \phi_N, \phi_\sigma, \phi_M, \theta)$ are a solution of the problem \eqref{mod_problem_loc} in the sense of Definition~\ref{Definition_loc}. In particular, for the ECM density $\theta$ we have to prove that it possesses the integral representation given in (\ref{cont_loc:6}).  Due to the strong convergence of $\phi_M^n$ to $\phi_M$ in $L^2(\Omega \times (0,T))$, there is a subsequence, for notational simplicity indexed again by $n$, such that
$$\phi_M^n(x,t) \longrightarrow \phi_M(x,t) \text{ for a.e. } (x,t) \in \Omega \times (0,T),$$
for $n \to \infty$. On account of the exponential function being continuous, the Lebesgue dominated convergence theorem, and $f_5$ being continuous and bounded, we have
$$\theta^n(x,t)=\theta_0(x) \exp\left\{ -\int_0^t f_5(\phi_M^n(x,s)) \, \dd s \right\} \longrightarrow \theta_0(x) \exp\left\{ -\int_0^t f_5(\phi_M(x,s)) \, \dd s \right\} \text{ a.e. }$$
as $n \to \infty$. Applying the Lebesgue dominated convergence theorem again yields
$$\theta^n  \longrightarrow \left( (x,t) \mapsto \theta_0(x) \exp\left\{-\int_0^t f_5(\phi_M(x,s)) \, \dd s\right\}\right) \text{ in } L^2(\Omega\times (0,T)),$$
as $n \to \infty$. Since strong convergence implies weak convergence and weak limits are unique, we have proven that $\theta$, given in (\ref{loc:weakconv:theta}), is of the required form.

For the other solution functions, we multiply the Galerkin system (\ref{GalerkinLocal:System}) by an arbitrary test function $\eta \in C_c^\infty(0,T)$ and integrate from $0$ to $T$, which gives for all $j \in \{1,\dots,n\}$,
\begin{equation} \label{loc:systemlimit}
\begin{aligned}
&\begin{multlined}[t] \int_0^T \big[ \langle \partial_t \phi^n_T, w_j \rangle  +  (m_T(\phi^n_T, \phi^n_N) \nabla \mu^n , \nabla  w_j ) - (J_\alpha(\phi^n_T, \phi^n_{N},\theta^n),\nabla w_j ) \\[-1mm] - (\phi^n_\sigma f_1(\phi_T^n,\phi_N^n), w_j ) + (\lambda_T^{\textup{apo}}\phi^n_T+ \lambda_N^{\textup{dec}}\phi^n_N,  w_j ) \big] \eta(t) \, \dd t=0, \end{multlined} \\[2mm]
& \int_0^T \big[-(\mu^n , w_j )+(\Psi'(\phi^n_T),  w_j )  +\eps_T^2 (\nabla \phi^n_T, \nabla  w_j ) - \chi_C (\phi^n_\sigma,  w_j ) + \delta_T (\phi_T^n,w_j)\big] \eta(t) \, \dd t=0 ,  \\[2mm]
& \int_0^T \big[ ( \partial_t \phi^n_N, \varphi^n) -   (\mathscr{S}(\sigma_{VN}-\phi^n_\sigma)f_2(\phi^n_T,\phi^n_N), w_j ) + \lambda_N^{\textup{deg}} (\phi_N^n,w_j) \rangle \big] \eta(t) \, \dd t=0, \\[2mm]
&\int_0^T \big[\langle \partial_t \phi^n_\sigma,  w_j  \rangle + ( D_\sigma(\phi_\sigma^n) (\delta_\sigma^{-1} \nabla \phi^n_{\sigma} - \chi_C \nabla \phi^n_T), \nabla w_j )  +((\phi^n_T-\phi^n_N) f_3(\phi_\sigma^n) , w_j  )\big] \eta(t) \, \dd t=0,  \\[2mm]
& \begin{multlined}[t] \int_0^T \big[ \langle \partial_t \phi^n_M,  w_j \rangle + (D_M(\phi_M^n) \nabla \phi^n_M , \nabla w_j )- (\theta^n f_4(\phi_T^n, \phi_N^n, \phi_\sigma^n, \phi_M^n), w_j )  \\[-1mm] +\lambda^{\textup{pro}}_M (\phi^n_M, w_j )\big] \eta(t) \, \dd t =0. \end{multlined}
\end{aligned}\end{equation}
We take the limit $n \to \infty$ in each equation. The convergence of the linear terms follows directly from the definition of weak convergence. For instance, the functional
$$\mu^n \mapsto \int_0^T (\mu^n,w_j) \eta(t) \, \dd t \leq \|\mu^n\|_{L^2L^2} |w_j|_{L^2} |\eta|_{L^2(0,T)}$$
is linear and continuous on $L^2(0,T;L^2)$ and therefore, we conclude that
$$\int_0^T (\mu^n,w_j) \eta(t) \, \dd t \longrightarrow \int_0^T (\mu,w_j) \eta(t) \, \dd t,$$
as $n \to \infty$. It remains to treat the nonlinear terms. We note that a similar limit process is performed in~\cite{fritz2018unsteady} for a tumor growth system which also includes a nonlinear mobility, diffusion, and potential function with the same assumptions as in (A2), (A3), and (A5). The same arguments can be applied to our model; we therefore omit the details here. \\
\indent We focus on the treatment of the adhesion flux $J_\alpha$ and the nonlinear functions $f_1$, ...,  $f_5$. We employ the following three arguments. \vspace*{2mm}

\noindent (i) By assumption (A6), the adhesion flux has the representation $$J_\alpha(\phi_T^n,\phi_N^n,\theta^n)=g(\phi_T^n,\phi_N^n)G(\theta^n),$$ for $g\in C_b(\mathbb{R}^2)$ and $G \in \mathscr{L}(X_\alpha;[L^2]^d)$. On the one hand, we know $\theta^n \rightharpoonup \theta$ weakly in $L^2(0,T;X_\alpha)$ as $n\to \infty$ by (\ref{loc:boundedweak}), which implies by the weak sequential continuity of $G$, $$G\theta^n \rightharpoonup G\theta \quad \text{weakly in} \ L^2(\Omega \times (0,T);\mathbb{R}^d)$$ as $n \to \infty$. On the other hand, we have derived $\phi_T^n \to \phi_T$ and $\phi_N^n \to \phi_N$ strongly in $L^2(\Omega \times (0,T))$ in (\ref{loc:strongconv}). Therefore, applying the Lebesgue dominated convergence theorem yields
$$g(\phi_T^n,\phi_N^n) \nabla w_j \eta \longrightarrow g(\phi_T,\phi_N) \nabla w_j \eta \quad \text{ strongly in } L^2(\Omega \times (0,T);\mathbb{R}^d),$$
as $n \to \infty$. Putting these two results together, we finally have, as $n \to \infty$,
$$J_\alpha(\phi_T^n,\phi_N^n,\theta^n) \nabla w_j \eta \longrightarrow J_\alpha(\phi_T,\phi_N,\theta) \nabla w_j \eta \quad \text{ strongly in } L^1(\Omega \times (0,T)).$$ \vspace{1mm}

\noindent (ii) Since $\mathscr{S}$ and $f_2$ are bounded, continuous functions, we obtain analogously to (i), as $n \to \infty$,
$$\mathscr{S}(\sigma_{VN}-\phi_\sigma^n) f_2(\phi_T^n,\phi_N^n) w_j \eta \longrightarrow \mathscr{S}(\sigma_{VN}-\phi_\sigma^n) f_2(\phi_T,\phi_N) w_j \eta \quad \text{ strongly in } L^2(\Omega \times (0,T)).$$ \vspace{1mm}

\noindent (iii) Similar to (i), we employ that $\theta^n \rightharpoonup \theta$ weakly in $L^2(\Omega \times (0,T))$ and 
$$f_4(\phi_T^n,\phi_N^n,\phi_\sigma^n,\phi_M^n) w_j \eta \longrightarrow  f_4(\phi_T,\phi_N,\phi_\sigma,\phi_M) w_j \eta \quad \text{ strongly in } L^2(\Omega \times (0,T)),$$
as $n \to \infty$, which implies the convergence of their product in $L^1(\Omega \times (0,T))$. Convergence of the terms involving $f_1$, $f_3$, and $f_5$ follows in the same manner. 

\indent Finally, by taking the limit $n \to \infty$ in the system (\ref{loc:systemlimit}), using the density of $\text{span}\{w_1,w_2,\dots\}$ in $H^1$, and the fundamental lemma of calculus of variations, we obtain a solution $(\phi_T,\mu,\phi_N,\phi_\sigma,\phi_M,\theta)$ of the system (\ref{Loc:SystemContinuous}) in the sense of Definition \ref{Definition_loc}.

We note that on account of the standard Sobolev embeddings, we have the following regularity in time of our solution:
\begin{alignat*}{2}
\phi_T, \phi_N &\in  C([0,T];L^2)\cap C_w([0,T];H^1), \\
\phi_\sigma,\phi_M &\in C([0,T];L^2) ,  
\end{alignat*}
and, thus, initial conditions are meaningful and the Galerkin approximations fulfill the initial data. This completes the proof. \qed



\section{Finite Element Approximations} \label{Sec_FEM}
	We select a similar algorithmic framework as in \cite{fritz2018unsteady,lima2014hybrid,lima2015analysis} to solve the deterministic systems of the respective local and nonlocal model with the initial and boundary data (\ref{Initial_bnd_data_loc}). This framework contains a discrete-time local semi-implicit scheme with an energy convex-nonconvex splitting; that means the stable contractive part is treated implicitly and the expansive part explicitly. In particular, recalling the Ginzburg--Landau energy $\mathcal{E}$ in (\ref{Ginzburg}), we split its contractive part $\mathcal{E}_c$ and expansive part $\mathcal{E}_e$ via $\mathcal{E}_e=\mathcal{E}-\mathcal{E}_c$, see also \cite{lima2015analysis,hawkins2012numerical}. 
	
	Let the time domain be divided into the steps $\Delta t_n = t_{n+1} - t_n$ for $n \in \{0,1,\dots\}$. To simplify exposition, we assume $\Delta t_n = \Delta t$ for all $n$. We write $\phi_{T_n}$ for the approximation of $\phi_T^h(t_n)$ and likewise for the other variables. The backward Euler method applied to the system (\ref{mod_problem_loc}) reads
	\begin{equation}
	\begin{aligned}
	\frac{\phi_{T_{n+1}} - \phi_{T_n}}{\Delta t} &= \begin{multlined}[t]  \div  \big(m_T(\phi_{T_{n+1}},\phi_{N_{n+1}}) \nabla \mu_{n+1} \big) -  \div(J_\alpha(\phi_{T_{n+1}},\phi_{N_{n+1}},\theta_{n+1})) \\  + \phi_{\sigma_{n+1}} f_{1,n+1} - \lambda_T^{\text{apo}} \phi_{T_{n+1}} - \lambda_N^{\text{dec}} \phi_{N_{n+1}} \end{multlined} \\[0.1cm]
	\mu_{n+1} &= D_{\phi_T} \mathcal{E}_c(\phi_{T_{n+1}},\phi_{\sigma_{n+1}})-D_{\phi_T} \mathcal{E}_e (\phi_{T_{n}},\phi_{\sigma_{n}}), \\
	\frac{\phi_{N_{n+1}} - \phi_{N_n}}{\Delta t} &=  \mathscr{S}(\sigma_{VN}-\phi_{\sigma_{n+1}}) f_{2,n+1} - \lambda_N^{\textup{deg}} \phi_{N_{n+1}}, \\
	\frac{\phi_{\sigma_{n+1}} - \phi_{\sigma_n}}{\Delta t} &= \div \big(D_\sigma(\theta_{n+1}) (\delta_\sigma^{-1} \nabla \phi_{\sigma_{n+1}} - \chi_C \nabla \phi_{T_{n+1}}) \big) + (\phi_{T_{n+1}}-\phi_{N_{n+1}}) f_{3,n+1}, \\   
    \frac{\phi_{M_{n+1}} - \phi_{M_n}}{\Delta t} &=\div(D_M(\theta_{n+1}) \nabla \phi_{M_{n+1}}) +  \theta_{n+1} f_{4,n+1} - \lambda_M^{\textup{dec}} \phi_{M_{n+1}}  , \\
    \frac{\theta_{{n+1}} - \theta_{n}}{\Delta t} &=- \theta f_{5,n+1}.
	\end{aligned}
	\label{Eq_Numerics}
	\end{equation}
    The functions  $f_{i,n+1}$, $i \in \{1,\dots,5\}$, are given by
    \begin{equation} \label{Eq_FunctionsRHS}
    \begin{aligned}
        f_{1,n+1} &= \lambda_T^{\text{pro}} (\mathcal{C}(\phi_{T_{n+1}})- \mathcal{C}(\phi_{N_{n+1}})) \cdot (1-\mathcal{C}(\phi_{T_{n+1}})), \\
        f_{2,n+1} &= \lambda_{VN} (\mathcal{C}(\phi_{T_{n+1}})- \mathcal{C}(\phi_{N_{n+1}})), \\
        f_{3,n+1} &= \lambda_T^{\text{pro}} (\mathcal{C}(\phi_{T_{n+1}})- \mathcal{C}(\phi_{N_{n+1}})) \frac{\mathcal{C}(\phi_{\sigma_{n+1}})}{\mathcal{C}(\phi_{\sigma_{n+1}})+\lambda_\sigma^{\text{sat}}}, \\
        f_{4,n+1} &= \lambda_M^{\text{pro}} (\mathcal{C}(\phi_{T_{n+1}})-\mathcal{C}(\phi_{N_{n+1}})) \frac{\sigma_H}{\sigma_H+\mathcal{C}(\phi_{\sigma_{n+1}})} (1-\mathcal{C}(\phi_{M_{n+1}})) - \lambda_\theta^{\text{dec}} \mathcal{C}(\phi_{M_{n+1}}),   \\
        f_{5,n+1} &= \lambda_{\theta}^{\text{deg}} \mathcal{C}(\phi_{M_{n+1}}),
    \end{aligned}
    \end{equation}    where $\mathcal{C}$ denotes the cut-off operator,
    $$\mathcal{C}(\sigma) = \max\!\big(0,\min(1,\sigma)\big).$$  The functions $f_{i,n+1}$, $i \in \{1,\dots,5\}$, are selected so that the model given in (\ref{1}), (\ref{N})--(\ref{ECM}) is replicated besides the cut-off operator and the Sigmoid function $\mathscr{S}$ approximating the Heaviside step function $\mathscr{H}$. Furthermore, the functions satisfy the assumptions given in $(\text{A7}_\text{loc})$.
	
	We solve the highly nonlinear coupled system (\ref{Eq_Numerics}) by decoupling the equations and using an iterative Gau\ss--Seidel method. In Algorithm \ref{Alg_Alg} below, the subscript $0$ stands for the initial solution, $k$ the iteration index, $n_{\text{iter}}$ the maximum number of iterations at each time step and TOL the tolerance for the iteration process.
	In each iterative loop, three linear systems are solved and the convergence of the nonlinear solution is achieved at each time if $\max |\phi_{T_{n+1}}^{k+1}-\phi_{T_{n+1}}^{k}| < \text{TOL}$. 
	
	We obtain the algebraic systems using a Galerkin finite element approach. Let $\mathcal{T}^h$ be a quasiuniform family of triangulations of $\Omega$ and let the piecewise linear finite element space be given by
	$$\mathcal{V}^h = \{ v \in C(\overline{\Omega}) : v|_{T} \in P_1(K) \text{ for all } K \in \mathcal{T}^h \} \subset H^1(\Omega),$$
	where $P_1(T)$ denotes the set of all affine linear functions on $T$.
	
	We formulate the discrete problem as follows: for each $k$, find $$\big(\phi_{T_{n+1}}^{k+1},\  \mu_{n+1}^{k+1},\ \phi_{N_{n+1}}^{k+1},\ \phi_{\sigma_{n+1}}^{k+1},  \ \phi_{M_{n+1}}^{k+1},\ \theta_{n+1}^{k+1}\big) \in (\mathcal{V}^h)^6,$$ 
	for all 
	$$\big( \varphi_T,\ \varphi_\mu, \ \varphi_N, \ \varphi_\sigma, \ \varphi_M, \ \varphi_\theta\big) \in (\mathcal{V}^h)^6,$$ such that:
	\begin{equation} \begin{aligned}
	(\phi_{\sigma_{n+1}}^{k+1} - \phi_{\sigma_n},\varphi_\sigma) &+\Delta t\big(D_M(\theta_{n+1}^k) \cdot (\delta_\sigma^{-1} \nabla \phi_{\sigma_{n+1}}^{k+1} - \chi_C \nabla \phi^k_{T_{n+1}}),\nabla \varphi_\sigma\big) 
	\\&- \Delta t\lambda_T^{\text{pro}} \! \left(\!(\phi_{T_{n+1}}^k\!-\phi_{N_{n+1}}^k)  (\mathcal{C}(\phi^k_{T_{n+1}})- \mathcal{C}(\phi^k_{N_{n+1}})) \frac{\mathcal{C}(\phi^{k+1}_{\sigma_{n+1}})}{\mathcal{C}(\phi^{k+1}_{\sigma_{n+1}})+\lambda_\sigma^{\text{sat}}},\varphi_\sigma\!\!\right) = 0;
	\label{Alg_sigma}
	\end{aligned}
	\end{equation}
		\begin{equation} \begin{aligned}
	(\phi_{T_{n+1}}^{k+1} - \phi_{T_n},\varphi_T)
	&+\Delta t  \big(m_T(\phi_{T_{n+1}}^{k+1},\phi^{k+1}_{N_{n+1}}) \nabla \mu_{n+1}^{k+1},\nabla \varphi_T\big) \\&-\Delta t  ( J_\alpha(\phi^{k+1}_{T_{n+1}},\phi^k_{N_{n+1}},\theta_{n+1}^k),\nabla \varphi_T) 
	\\ &-\Delta t \lambda_T^{\text{pro}} ((\mathcal{C}(\phi^{k+1}_{T_{n+1}})- \mathcal{C}(\phi^k_{N_{n+1}})) \cdot (1-\mathcal{C}(\phi^{k+1}_{T_{n+1}})),\phi_{\sigma_{n+1}}^{k+1} \varphi_T) 
	\\ &+\Delta t  (\lambda_T^{\text{apo}} \phi_{T_{n+1}}^{k+1} + \lambda_N^{\text{dec}} \phi_{N_{n+1}}^k,\varphi_T)=0;
	\label{Alg_T}
	\end{aligned}
	\end{equation}
	\begin{equation} \begin{aligned}
	(\mu_{n+1}^{k+1},\varphi_\mu) - \big(D_{\phi_T} \mathcal{E}_c(\phi^{k+1}_{T_{n+1}},\phi^{k+1}_{\sigma_{n+1}}),\varphi_\mu\big)=\big(D_{\phi_T} \mathcal{E}_e (\phi_{T_{n}},\phi_{\sigma_{n}}),\varphi_\mu\big);
	\label{Alg_mu}
	\end{aligned}
	\end{equation}
	\begin{equation} \begin{aligned}
	(\phi_{N_{n+1}}^{k+1} - \phi_{N_n},\varphi_N)  &-\Delta t  \big(\mathcal{C}(\phi_{T_{n+1}}^{k+1}-\phi_{N_{n+1}}^{k+1}),\mathscr{S}(\sigma_{VN}-\phi_{\sigma_{n+1}}^{k+1}) \varphi_N\big) 
	\\&+ \Delta t \lambda_N^{\text{deg}} (\phi_{N_{n+1}}^{k+1},\varphi_N) = 0;
	\label{Alg_N}
	\end{aligned}
	\end{equation}
		\begin{equation} \begin{aligned}
	(\phi_{M_{n+1}}^{k+1} - \phi_{M_n},\varphi_M)
	&+\Delta t (D_M(\theta_{n+1}^k) \nabla \phi_{M_{n+1}}^{k+1},\nabla \varphi_M)
	\\&-\Delta t \lambda_M^{\text{pro}} \!\!\left(\!\theta_{n+1}^{k} (\mathcal{C}(\phi^{k+1}_{T_{n+1}})-\mathcal{C}(\phi^{k+1}_{N_{n+1}})) \frac{\sigma_H}{\sigma_H+\mathcal{C}(\phi^{k+1}_{\sigma_{n+1}})} (1-\mathcal{C}(\phi^{k+1}_{M_{n+1}})) ,\varphi_M\!\!\right)\!\!\!
	\\&+\Delta t \lambda_\theta^{\text{dec}} (\theta_{n+1}^{k}  \mathcal{C}(\phi^{k+1}_{M_{n+1}}) ,\varphi_M)
	\\&+\Delta t \lambda_M^{\text{dec}} (\phi_{M_{n+1}}^{k+1},\varphi_M) = 0;
	\label{Alg_MDE}
	\end{aligned}
	\end{equation}
		\begin{equation} \begin{aligned}
	(\theta_{{n+1}}^{k+1} - \theta_{n},\varphi_\theta)+\Delta t \lambda_\theta^{\text{deg}} (\theta_{n+1}^{k+1} \mathcal{C}(\phi_{M_{n+1}}^{k+1}),\varphi_\theta) = 0.
	\label{Alg_ECM}
	\end{aligned}
	\end{equation}
	
	\begin{algorithm}[H] \small
		\SetAlgoLined
		\caption{Semi-implicit scheme for (\ref{Eq_Numerics})} \label{Alg_Alg} \vspace{2mm}
		\textbf{Input}: $\phi_{T_0}, \phi_{N_0}, \phi_{\sigma_0}, \phi_{M_0}, \theta_0, \Delta t, T, \text{TOL}$\\[2mm]
		\textbf{Output}: $\phi_{T_n}, \mu_n, \phi_{N_n}, \phi_{\sigma_n}, \phi_{M_n}, \theta_n$ for all $n$ \\[2mm]
		$n=0$ \\[2mm]
		$t=0$\vspace{2mm}\\
		\While{$t\leq T$ \vspace{1mm}}{ 
			$\phi^0_{T_{n+1}} = \phi_{T_n}$, $\phi^0_{N_{n+1}} = \phi_{N_n}$, $\phi^0_{\sigma_{n+1}} = \phi_{\sigma_n}$, $\phi^0_{M_{n+1}} = \phi_{M_n}$, $\theta^0_{{n+1}} = \theta_{n}$ \\[2mm]
			\While{$\displaystyle \textup{max} \| \phi_{T_{n+1}}^{k+1} - \phi_{T_{n+1}}^k\| > \textup{TOL}$ \vspace{1mm}}{
				$\phi_{T_{n+1}}^k = \phi_{T_{n+1}}^{k-1}$, $\phi_{N_{n+1}}^k = \phi_{N_{n+1}}^{k-1}$, $\phi_{\sigma_{n+1}}^k = \phi_{\sigma_{n+1}}^{k-1}$, $\phi_{M_{n+1}}^k = \phi_{M_{n+1}}^{k-1}$, $\theta_{{n+1}}^k = \theta_{{n+1}}^{k-1}$
				\\[2mm]
				\textbf{solve \bm{$\phi_{\sigma_{n+1}}^{k+1}$} using \eqref{Alg_sigma}, given} $\phi_{\sigma_n}, \phi_{T_{n+1}}^k, \phi_{N_{n+1}}^k$\\[2mm]
				\textbf{solve \bm{$\phi_{T_{n+1}}^{k+1}$}, \bm{$\mu_{n+1}^{k+1}$} using \eqref{Alg_T} and \eqref{Alg_mu}, given} $\phi_{T_n}, \phi_{T_{n+1}}^k, \phi_{N_{n+1}}^k, \phi_{\sigma_{n+1}}^{k+1}, \theta_{n+1}^k$ \\[2mm]
				\textbf{solve \bm{$\phi_{N_{n+1}}^{k+1}$} using \eqref{Alg_N}, given} $\phi_{N_n}, \phi_{T_{n+1}}^{k+1}, \phi_{\sigma_{n+1}}^{k+1}$\\[2mm]
				\textbf{solve \bm{$\phi_{M_{n+1}}^{k+1}$} using \eqref{Alg_MDE}, given} $\phi_{M_n}, \phi_{T_{n+1}}^{k+1}, \phi_{\sigma_{n+1}}^{k+1}, \theta_{n+1}^k$\\[2mm]
				\textbf{solve \bm{$\theta_{{n+1}}^{k+1}$} using \eqref{Alg_ECM}, given} $\theta_n, \phi_{M_{n+1}}^{k+1}$\\[2mm]
				$k \mapsto k+1$\\} 
			$\phi_{T_{n+1}}= \phi_{T_{n+1}}^{k+1}$, $\mu_{n+1} = \mu_{n+1}^{k+1}$, $\phi_{N_{n+1}} = \phi_{N_{n+1}}^{k+1}$\\[2mm] $\phi_{\sigma_{n+1}} = \phi_{\sigma_{n+1}}^{k+1}$, $\phi_{M_{n+1}} = \phi_{M_{n+1}}^{k+1}$, $\theta_{n+1} = \theta_{n+1}^{k+1}$ \\[2mm]
			$n \mapsto n+1$\\[2mm]
			$t \mapsto t+\Delta t$} 
	\end{algorithm}
	~\\
	
	We implemented Algorithm \ref{Alg_Alg} in libMesh \cite{libMeshPaper}, an open-source computing platform for solving partial differential equations using finite element methods. We use this implementation to obtain the numerical results below. 
\section{Numerical Simulations} \label{Sec_Simulations}

In this section, numerical approximations of the growth of the tumor volume fractions $\phi_T$ and the simulation of the other variables in the local and nonlocal model (\ref{mod_problem_loc}) obtained by implementing Algorithm \ref{Alg_Alg} are presented. We present a numerical experiment of the local model both in two and three dimension in the domain $\Omega=(-1,1)^d$, $d \in \{2,3\}$. Afterwards, we compare the growth of the tumor volume fraction in the local and nonlocal model in two dimensions. 

We impose for the nutrient concentrations an inhomogeneous Dirichlet boundary condition at $x_1 = 1$, namely $\phi_\sigma=1$. This is a slight modification to the analyzed model in Section \ref{Sec_Analysis}, but the existence proof can be adapted in a straightforward way, see \cite{fritz2018unsteady}.\\

We choose for the parameters in our system (\ref{mod_problem_loc}) the dimensionless values
	\begin{equation*} \begin{alignedat}{9} 
	\varepsilon_T &=0.005, \quad &\chi_C &=0, \quad &\chi_H &=0.001, \quad &\delta_\sigma &=0.01, \quad &\delta_T &=0, \\
	\lambda_T^\text{pro} &=2, \quad &\lambda_T^\text{apo} &=0.005, \quad &\lambda_N^\text{deg} &=0, \quad &\lambda_{VN} &=1, \quad &\lambda_\sigma^\text{sat} &=0, \\
	\lambda_M^\text{dec} &=1, \quad &\lambda_M^\text{pro} &=1, \quad &\lambda_\theta^\text{dec} &=0.1, \quad &\lambda_{\theta}^{\text{deg}} &=1, \quad &\overline{E} &=0.045, \\
		\sigma_H &=0.6, \quad &\sigma_{VN} &=0.44, \quad &M_T &=2, \quad &D_\sigma &=0.001, \quad &D_M &=0.1. \\
	\end{alignedat}
	\label{Eq_ParametersSimulation}
	\end{equation*}
	
\pagebreak

\subsection{Local model in two dimensions}

In Figure \ref{FigureLocalPhi2D}, the computed simulations of the volume fractions of tumor cells ($\phi_T$), necrotic cells ($\phi_N$) and viable cells ($\phi_V$) for a local model in a two-dimensional domain are shown at four different time points $t \in \{0,5,10,15\}$. For the initial conditions, we start off from a small circular concentration of tumor cells without a necrotic part, that means $\phi_T=\phi_V$ at $t=0$. 

In the first row of Figure \ref{FigureLocalPhi2D}, one observes that the tumor volume fraction $\phi_T$ evolves towards the nutrient-rich part of the domain, see also Figure \ref{FigureLocalNut2D} below for the simulation of $\phi_\sigma$. As the transition between tumor phenotypes is guided by the nutrient concentration, the necrotic concentration increases in the nutrient-poor region, see the second row in Figure \ref{FigureLocalPhi2D}. Moreover, in the third row in Figure \ref{FigureLocalPhi2D}, the viable tumor cells, responsible for the tumor growth, are concentrated closer to the right side of the domain, which is the region with higher nutrient concentration. 


	\begin{figure}[H]
	\centering
\begin{tikzpicture}
\draw (0, 0) node[inner sep=0] {\includegraphics[width=.225\textwidth]{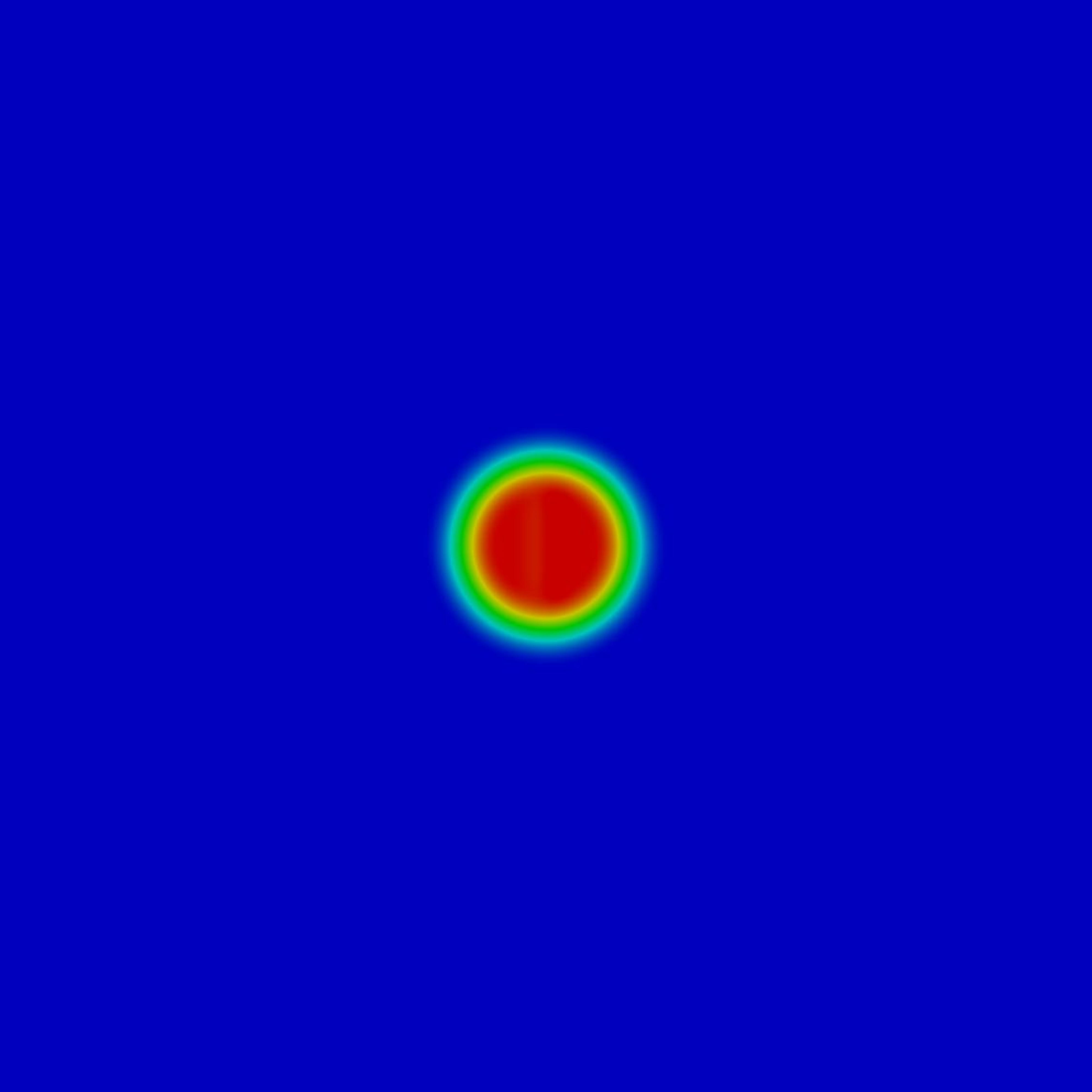}};
\draw (.1, 1.9) node {$t=0$};
\draw (-2,0)  node {$\phi_T$};
\end{tikzpicture}
\begin{tikzpicture}
\draw (0, 0) node[inner sep=0] {\includegraphics[width=.225\textwidth]{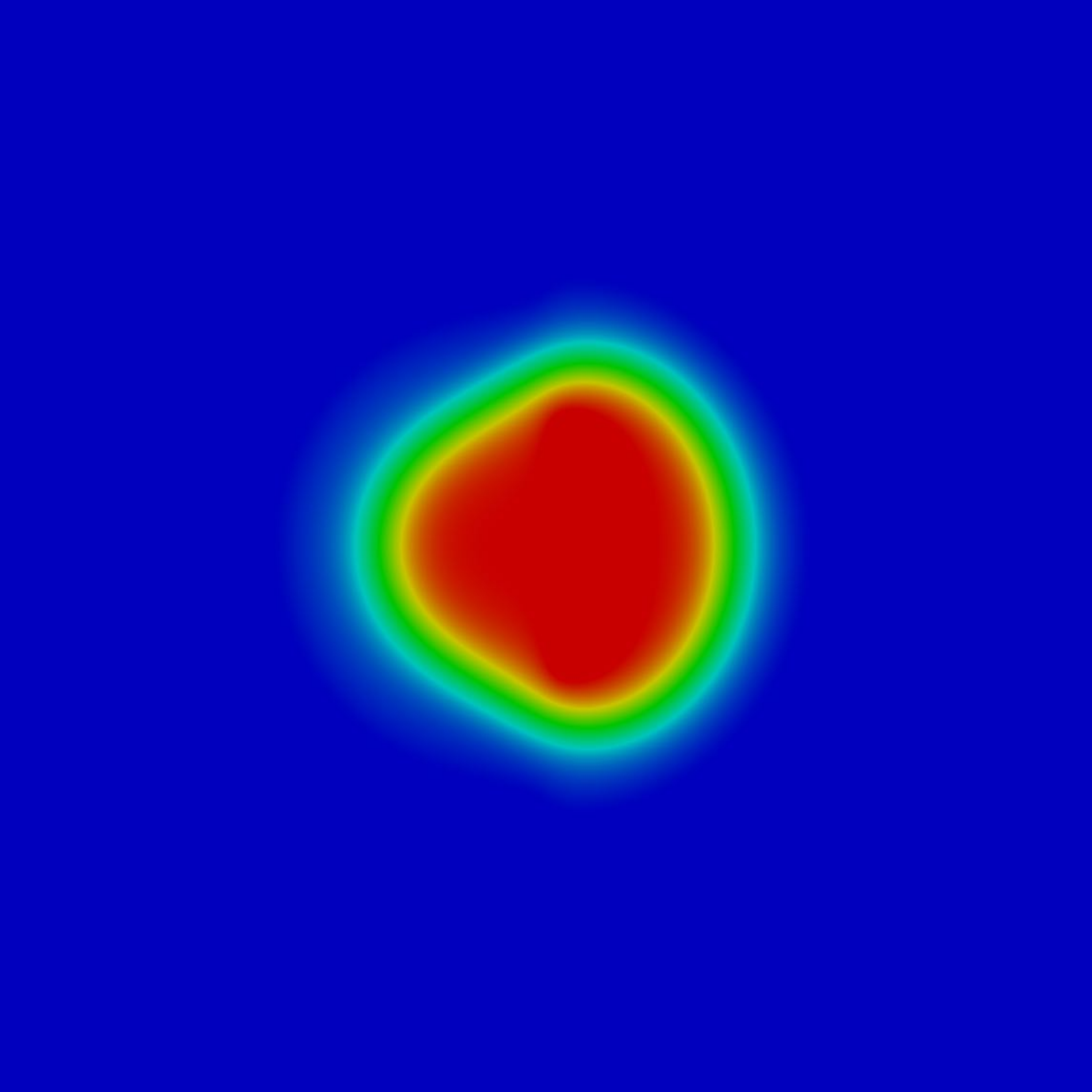}};
\draw (.1, 1.9) node {$t=5$};
\end{tikzpicture}
\begin{tikzpicture}
\draw (0, 0) node[inner sep=0] {\includegraphics[width=.225\textwidth]{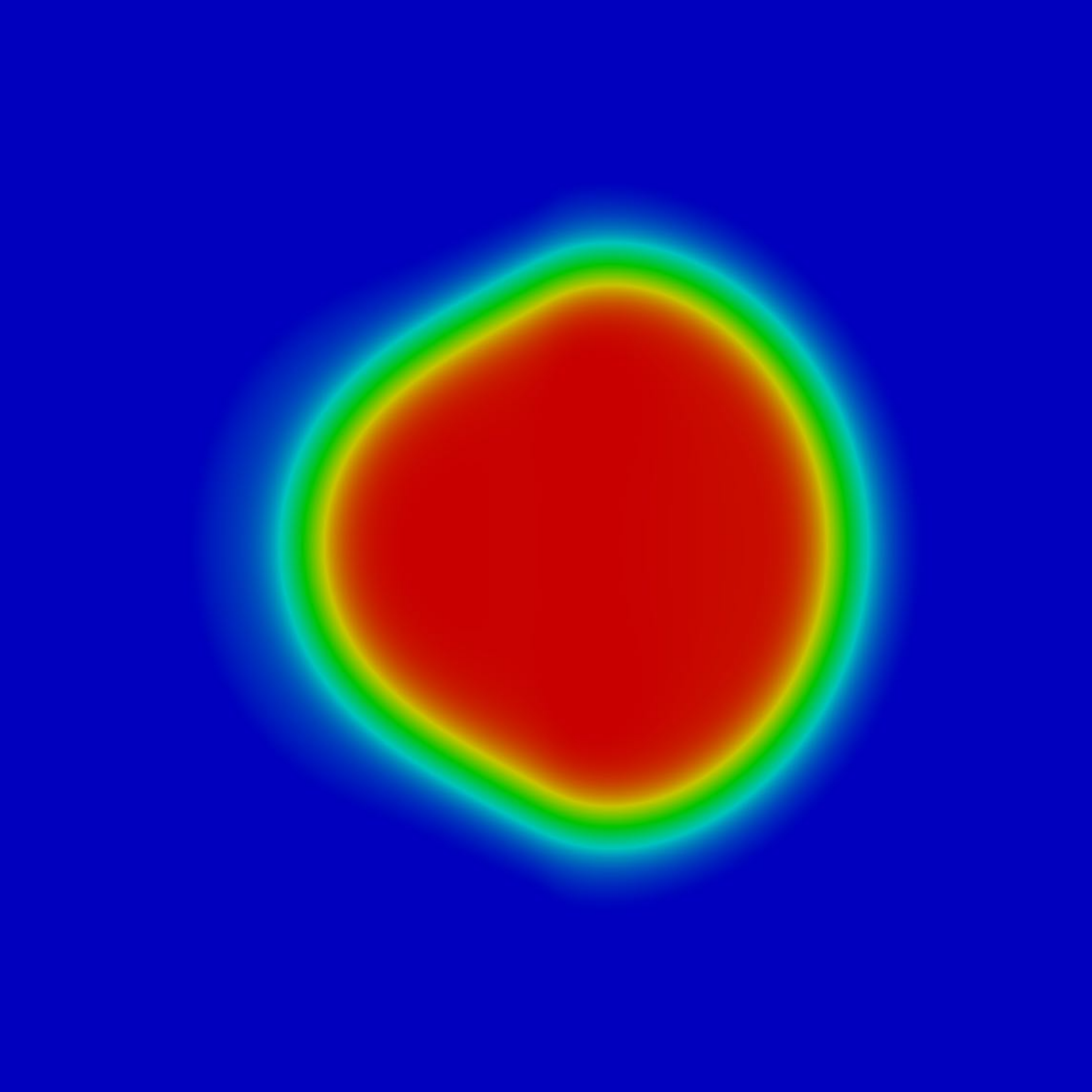}};
\draw (.1, 1.9) node {$t=10$};
\end{tikzpicture}
\begin{tikzpicture}
\draw (0, 0) node[inner sep=0] {\includegraphics[width=.225\textwidth]{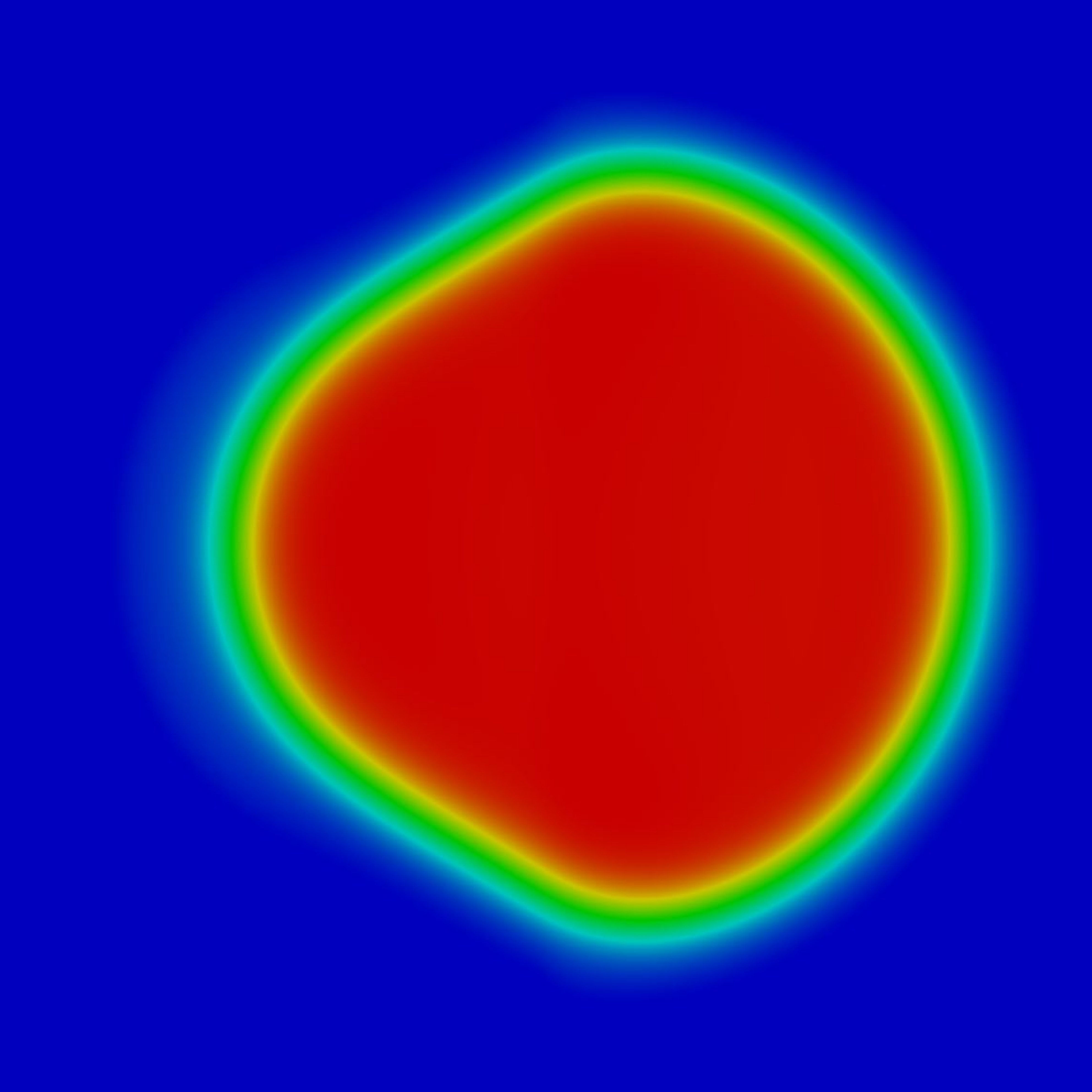}};
\draw (.1, 1.9) node {$t=15$};
\end{tikzpicture} \\[0.1cm]
\begin{tikzpicture}
\draw (0, 0) node[inner sep=0] {\includegraphics[width=.225\textwidth]{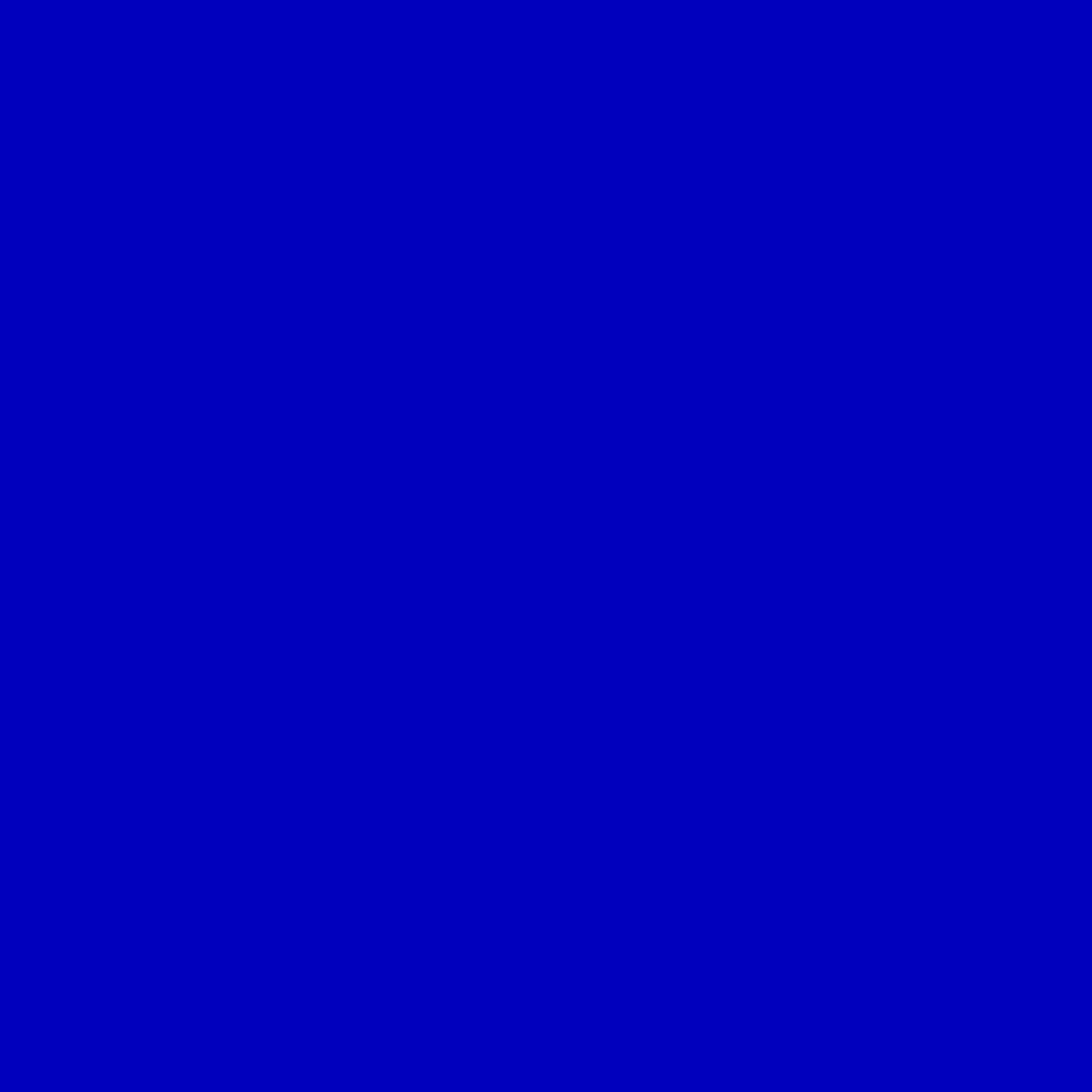}};
\draw (-2,0)  node {$\phi_N$};
\end{tikzpicture}
	\includegraphics[width=.225\textwidth]{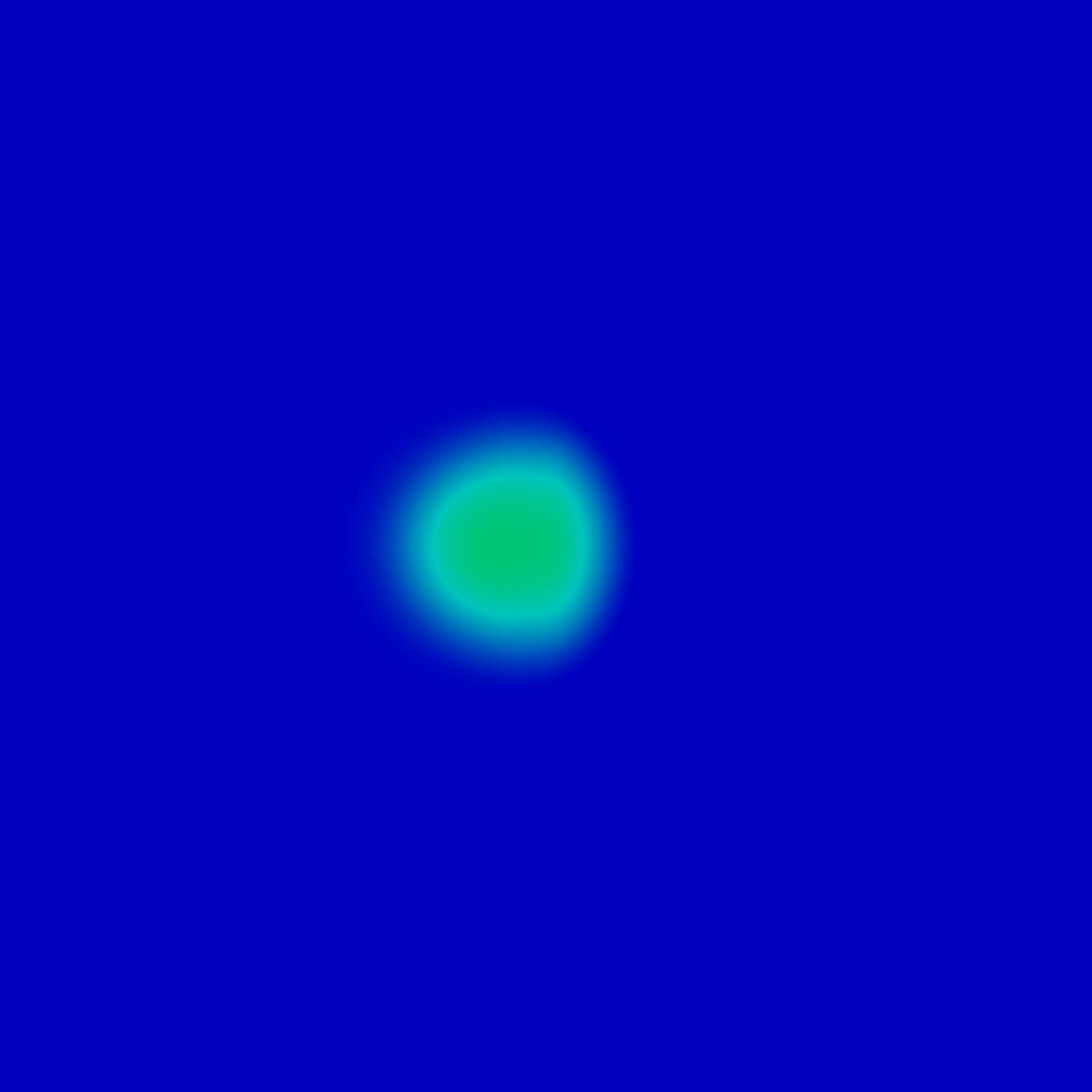}
	\includegraphics[width=.225\textwidth]{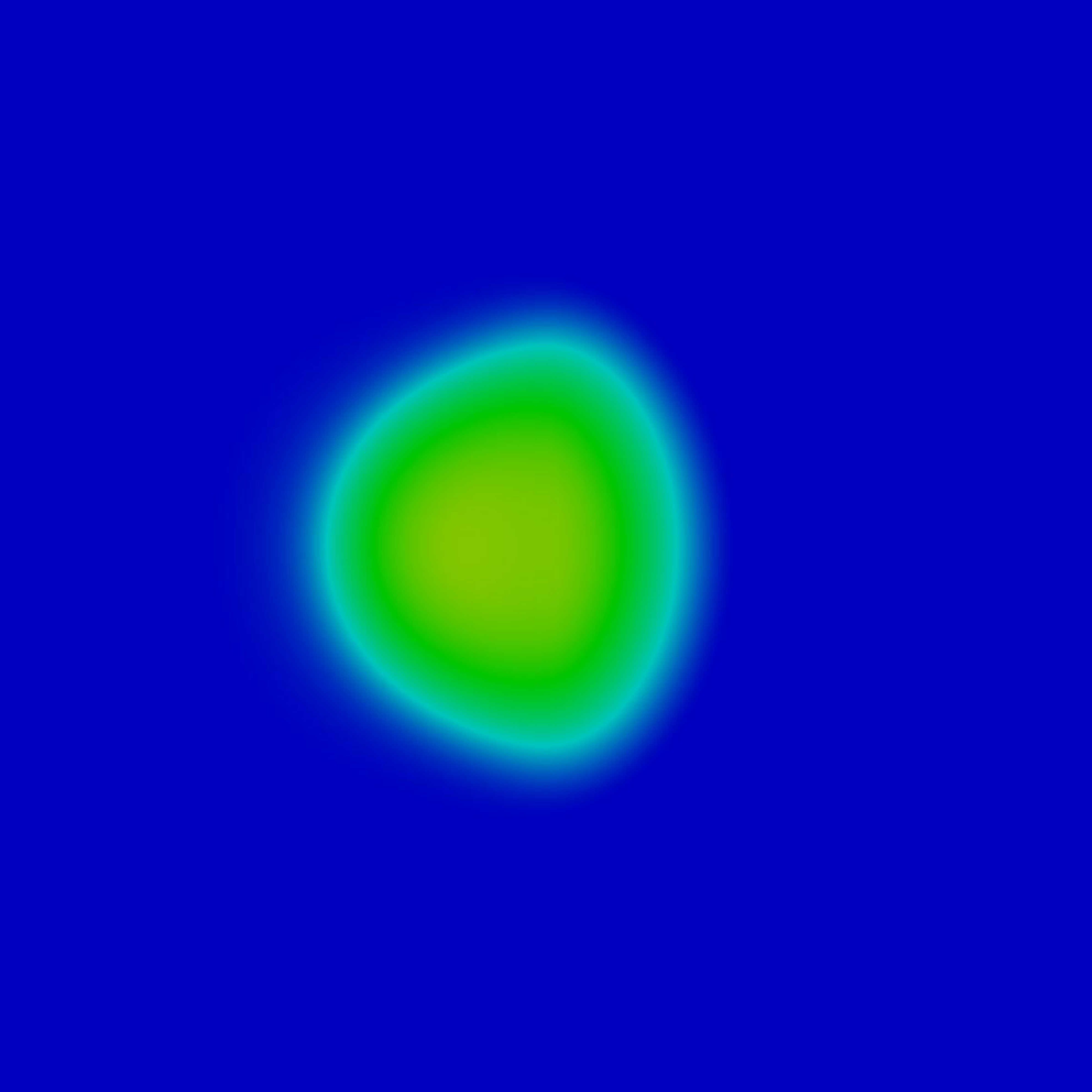}
	\includegraphics[width=.225\textwidth]{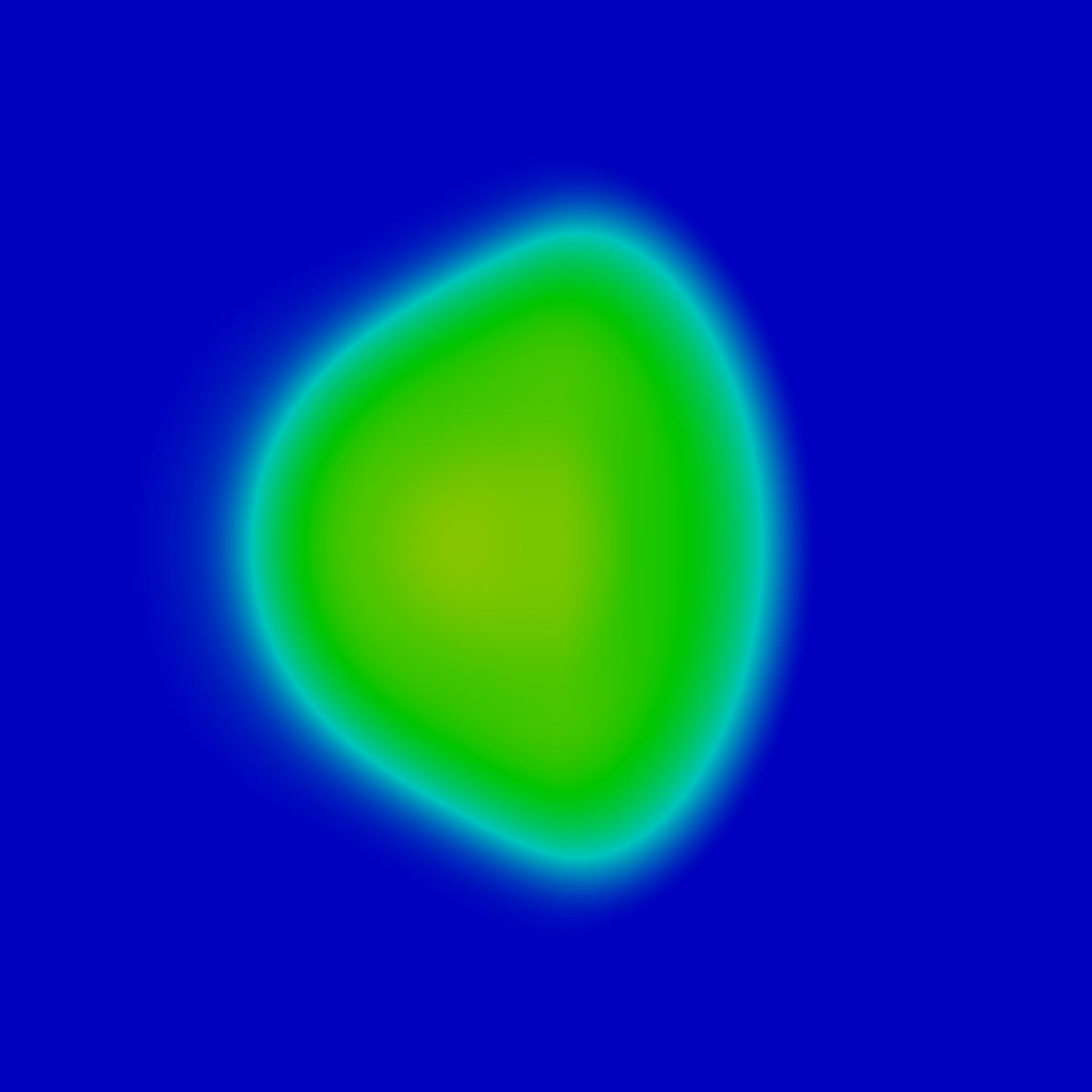} \\[0.1cm]
\begin{tikzpicture}
\draw (0, 0) node[inner sep=0] {\includegraphics[width=.225\textwidth]{FiguresLocal2D/tum0}};
\draw (-2,0)  node {$\phi_V$};
\end{tikzpicture}
	\includegraphics[width=.225\textwidth]{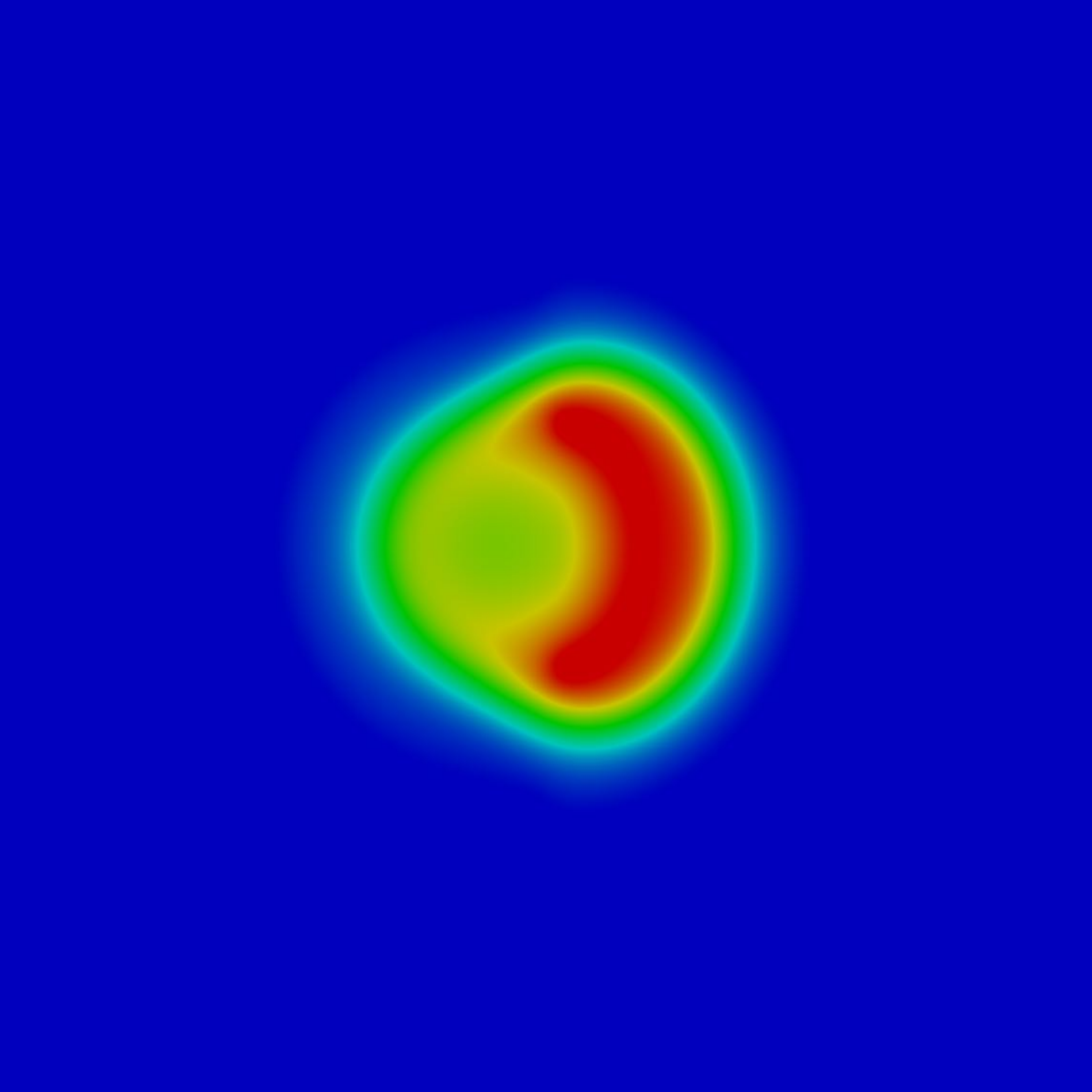}
	\includegraphics[width=.225\textwidth]{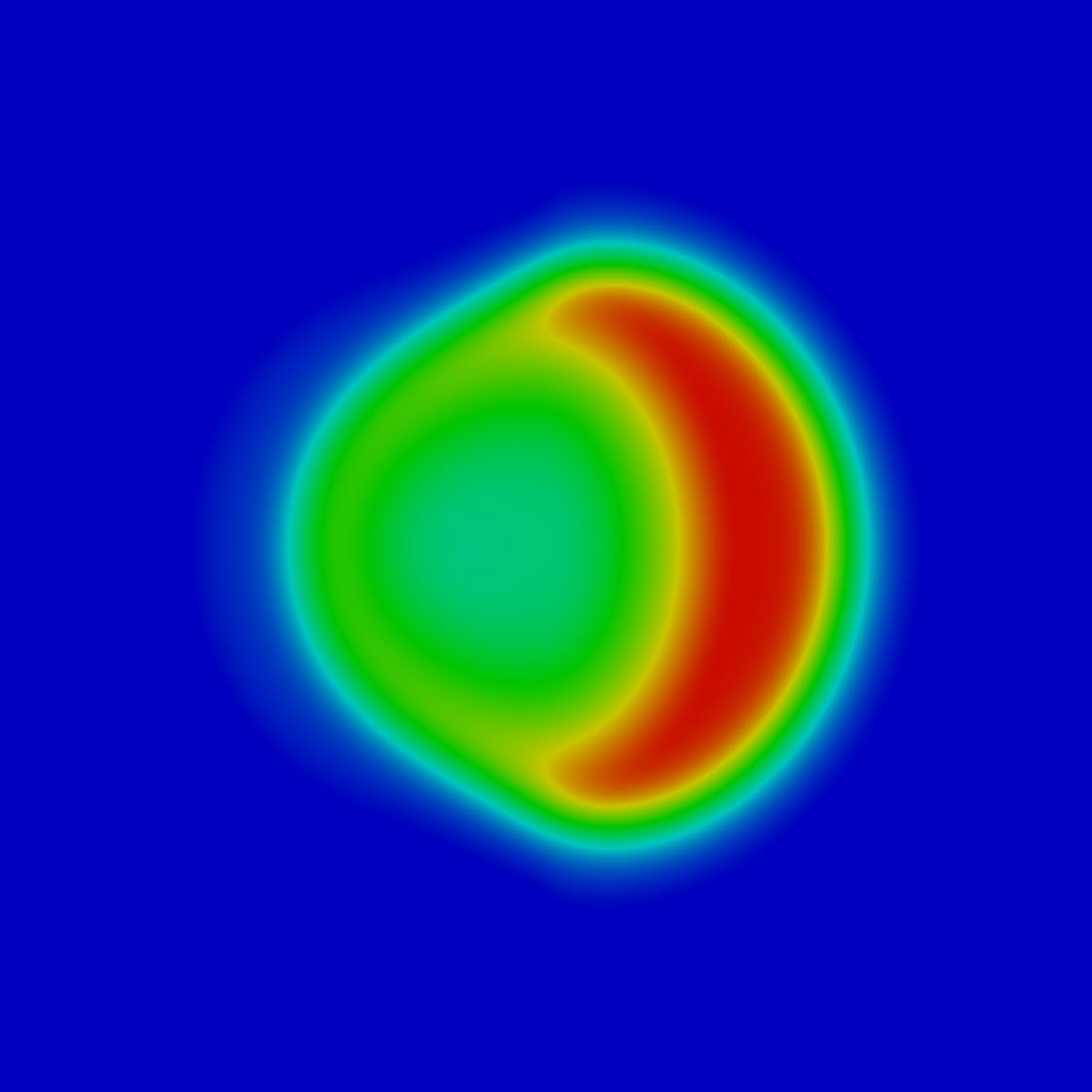}
	\includegraphics[width=.225\textwidth]{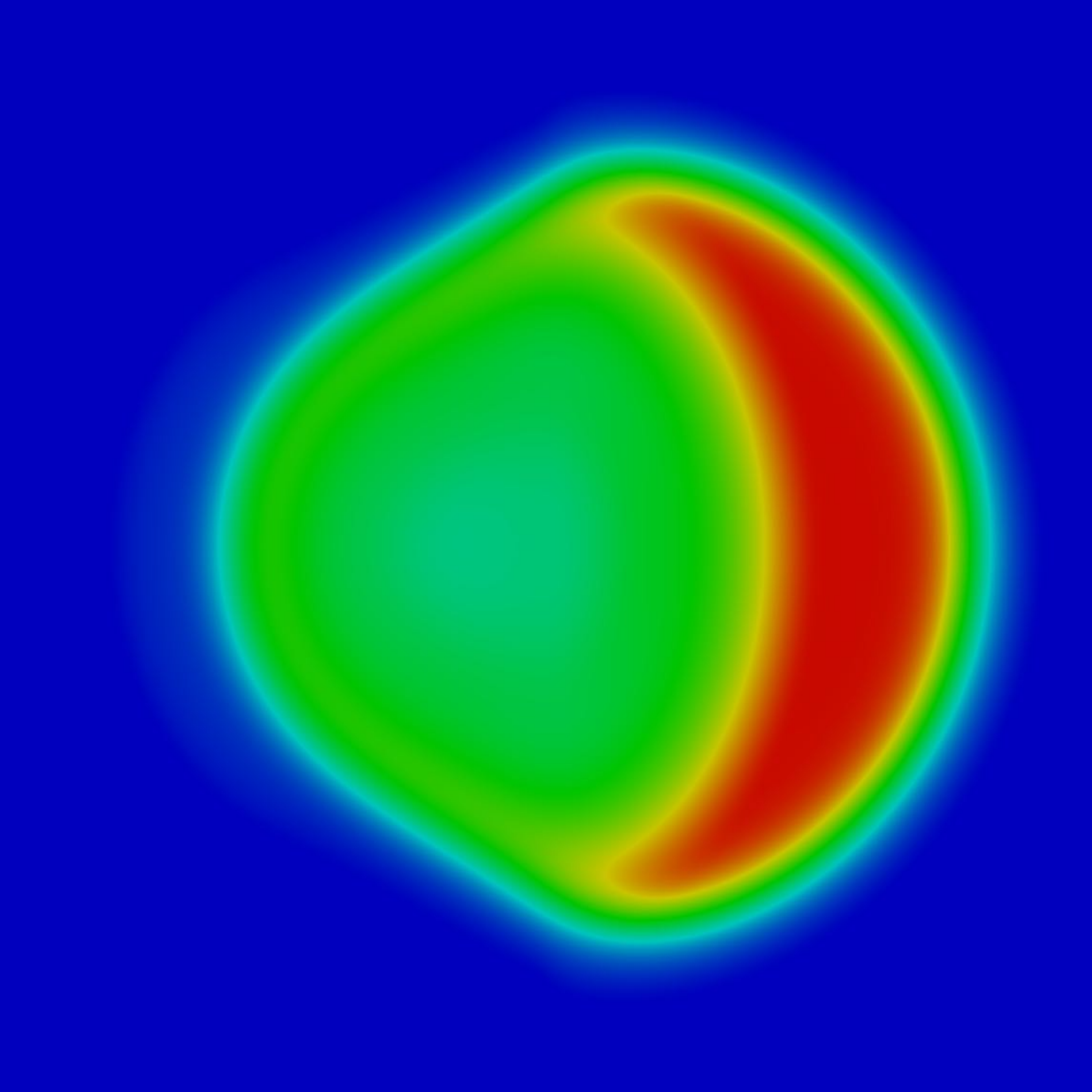} \\
	\hspace{.59cm} \begin{tikzpicture}
	\begin{axis}[
	hide axis,
	scale only axis,
	height=0pt,
	width=0pt,
	colormap name=rainbow,
	colorbar horizontal,
	point meta min=0,
	point meta max=1,
	colorbar style={
		samples=100,
		height=.5cm,
		xtick={0,0.5,1},
		width=10cm
	},
	]
	\addplot [draw=none] coordinates {(0,0)};
	\end{axis}
	\end{tikzpicture} 
	\caption{Simulation of the volume fractions $\phi_T$, $\phi_N$, $\phi_V$ in the local model in the 2D domain $\Omega=(-1,1)^2$; the evolution of the tumor, necrotic and viable cells is shown at the times $t\in \{0,5,10,15\}$}
	\label{FigureLocalPhi2D}
\end{figure}

In the first row of Figure \ref{FigureLocalNut2D}, the extracellular matrix density ($\theta$) is degraded over time by the matrix degrading enzymes ($\phi_{M}$). These enzymes are released by the tumor cells, mainly at regions with low nutrient and high ECM density. The nutrient concentration decreases as the tumor grows, with a higher value of $\phi_\sigma$ towards the boundary on the right hand side of the domain $\Omega=(-1,1)^2$, due to the imposed Dirichlet boundary condition $\phi_\sigma=1$ at $x_1=1$.

	\begin{figure}[H]
	\centering
\begin{tikzpicture}
\draw (0, 0) node[inner sep=0] {\includegraphics[width=.225\textwidth]{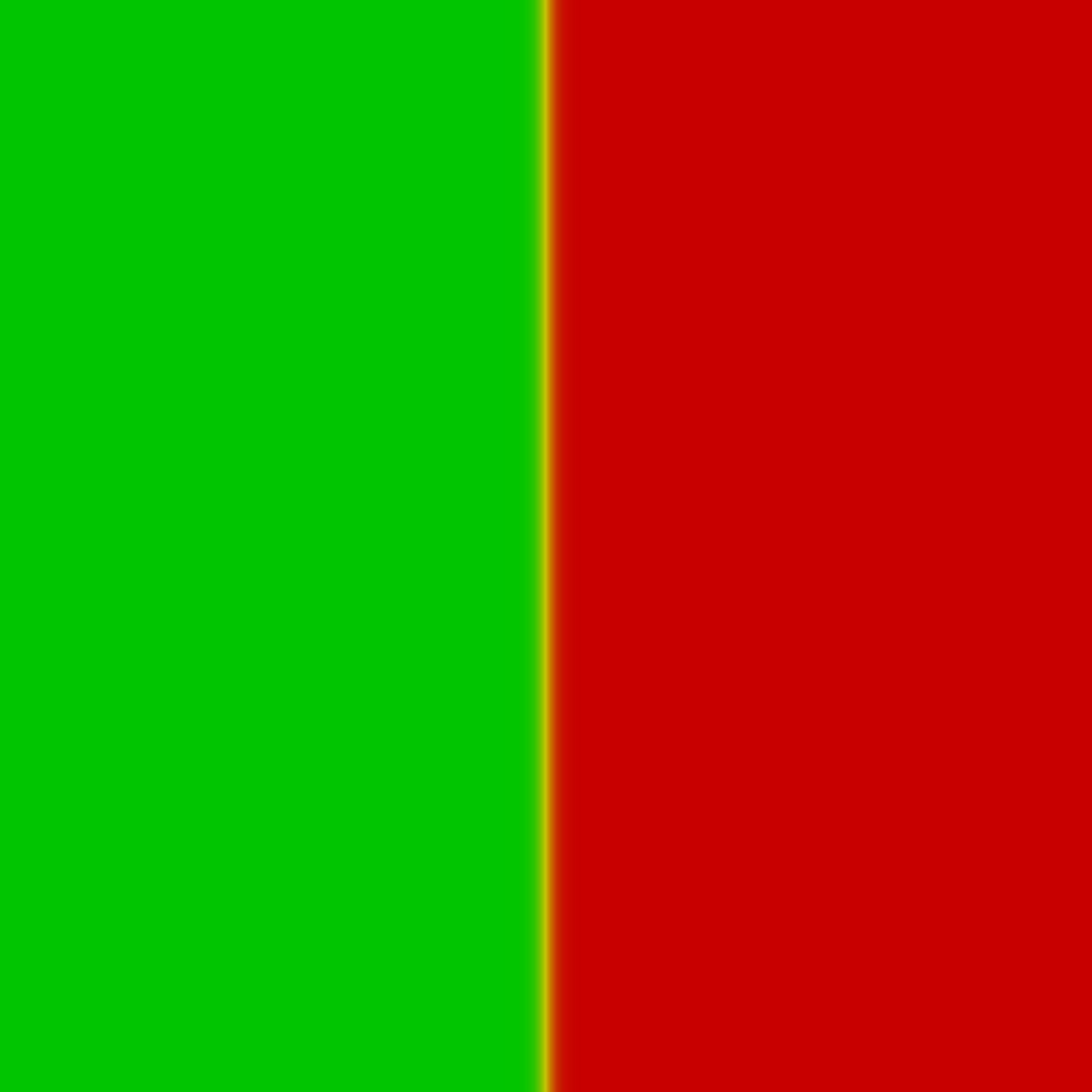}};
\draw (.1, 1.9) node {$t=0$};
\draw (-2,0)  node {$\theta_{\phantom{T}}$};
\end{tikzpicture}
\begin{tikzpicture}
\draw (0, 0) node[inner sep=0] {\includegraphics[width=.225\textwidth]{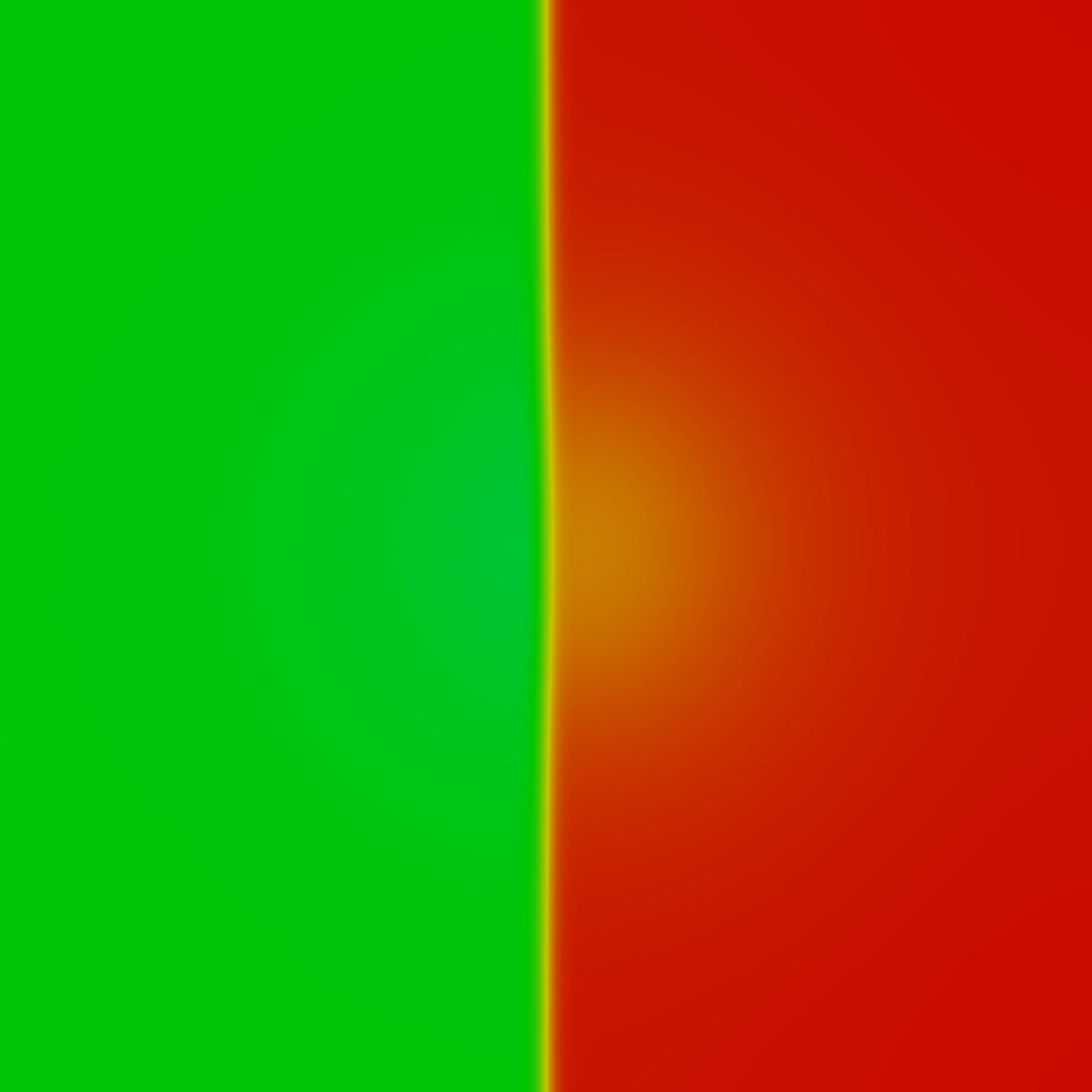}};
\draw (.1, 1.9) node {$t=5$};
\end{tikzpicture}
\begin{tikzpicture}
\draw (0, 0) node[inner sep=0] {\includegraphics[width=.225\textwidth]{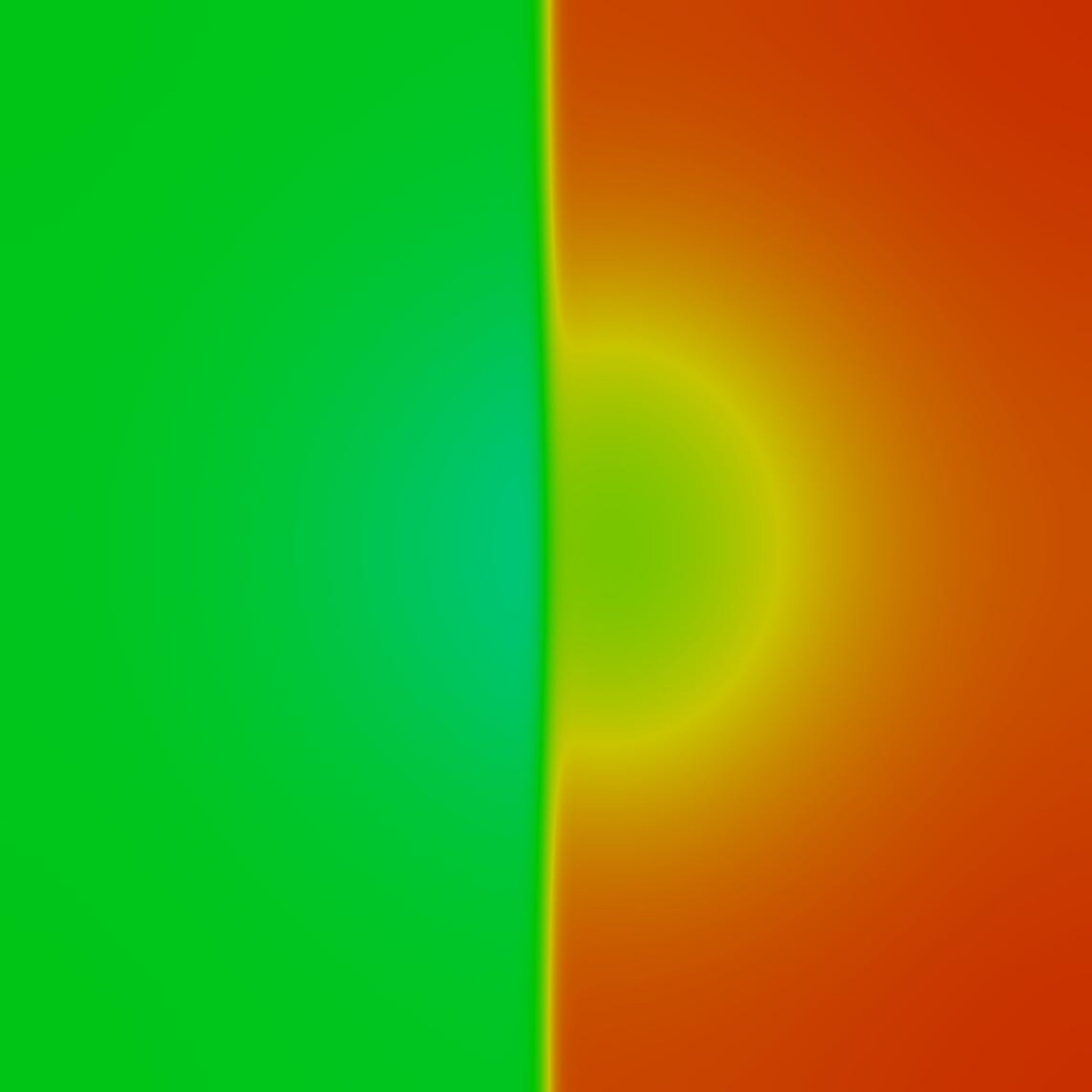}};
\draw (.1, 1.9) node {$t=10$};
\end{tikzpicture}
\begin{tikzpicture}
\draw (0, 0) node[inner sep=0] {\includegraphics[width=.225\textwidth]{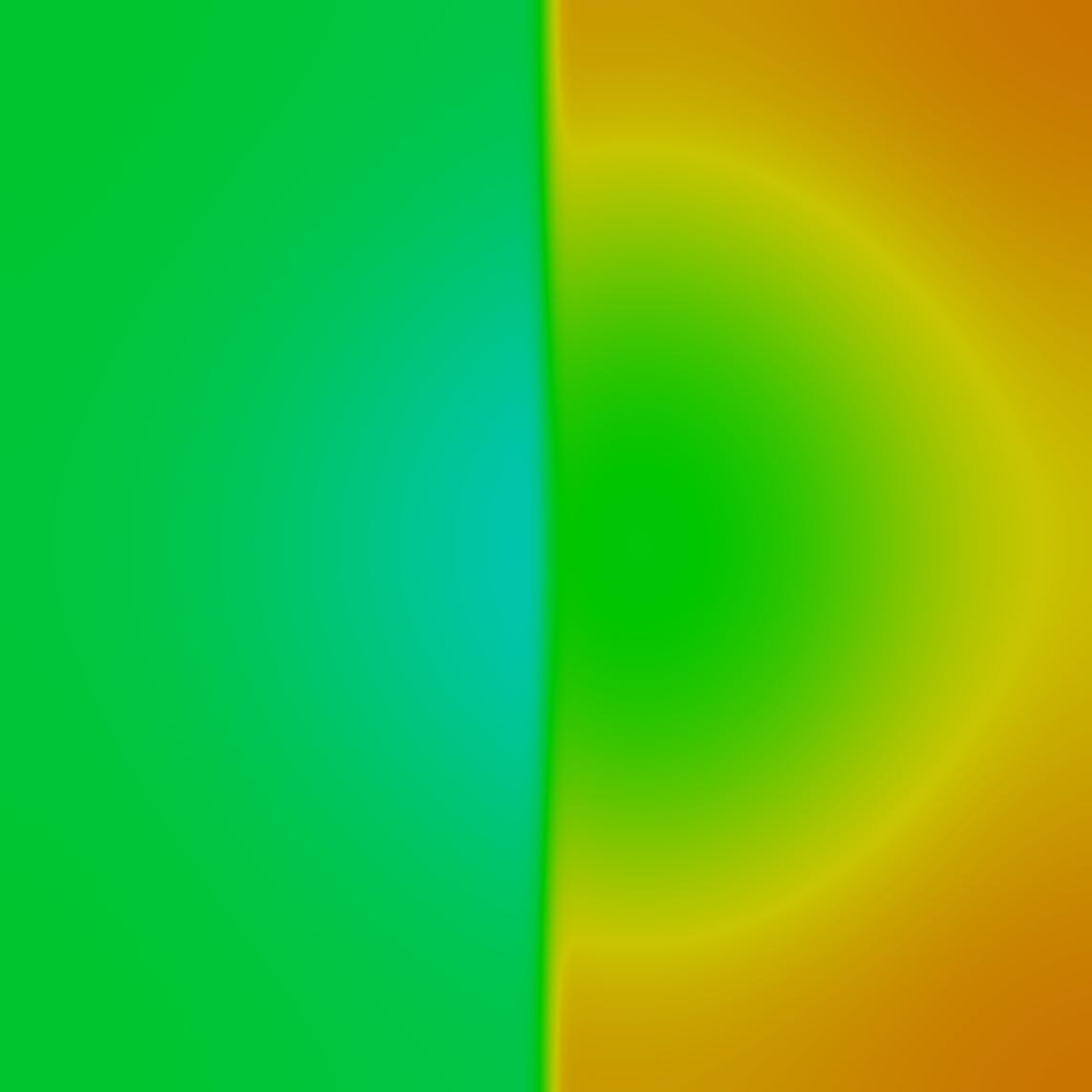}};
\draw (.1, 1.9) node {$t=15$};
\end{tikzpicture} \\[0.1cm]
\begin{tikzpicture}
\draw (0, 0) node[inner sep=0] {\includegraphics[width=.225\textwidth]{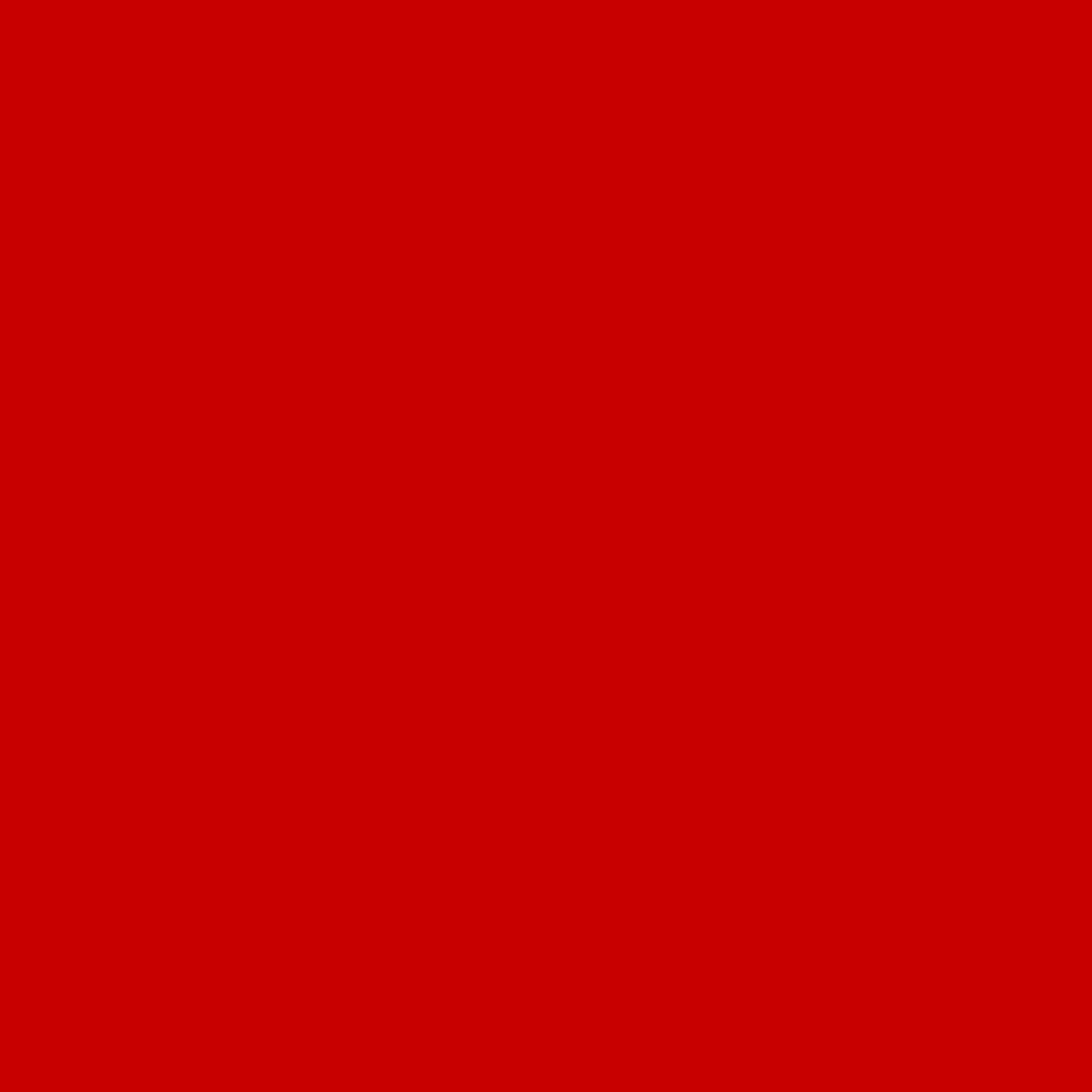}};
\draw (-2,0)  node {$\phi_\sigma$};
\end{tikzpicture}
	\includegraphics[width=.225\textwidth]{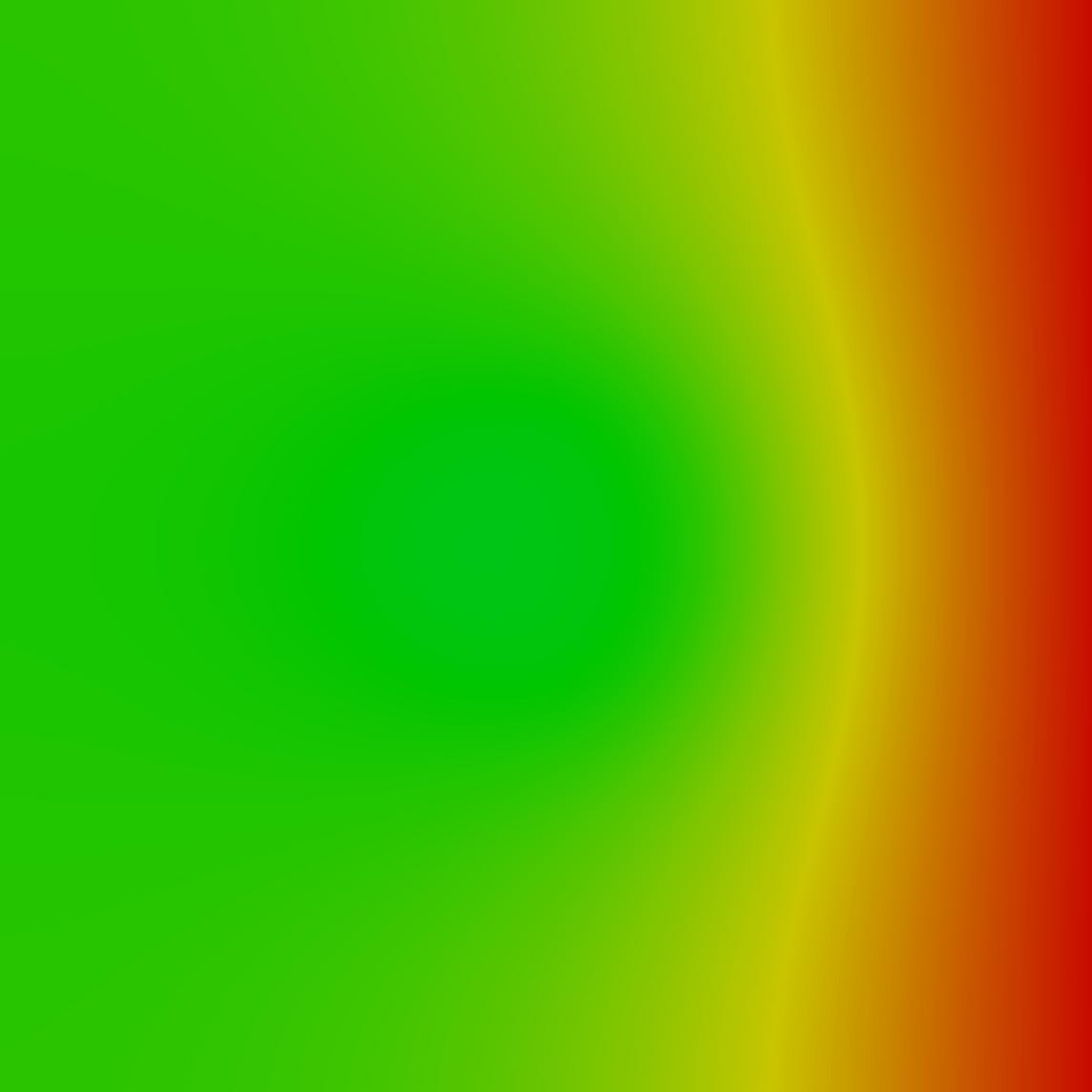}
	\includegraphics[width=.225\textwidth]{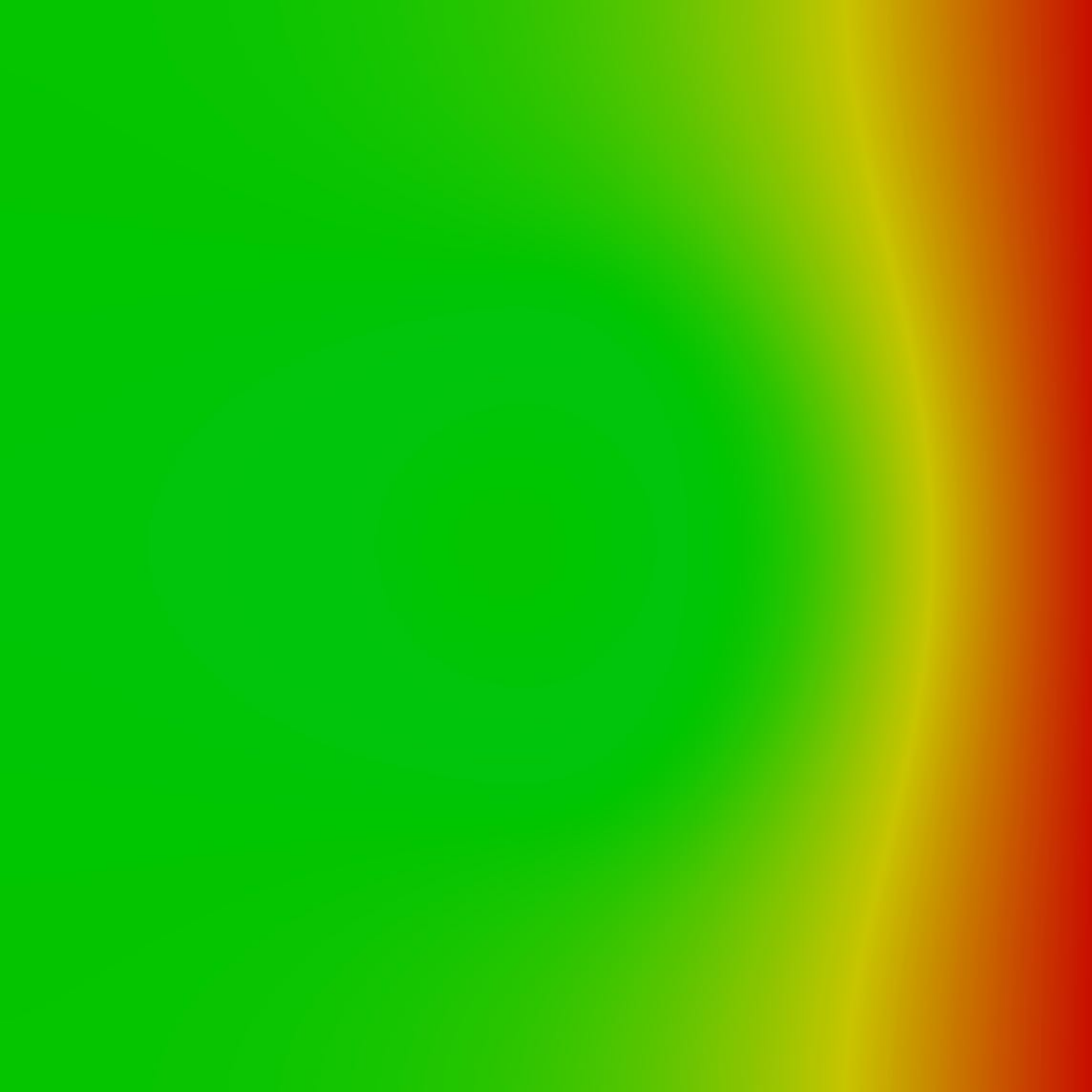}
	\includegraphics[width=.225\textwidth]{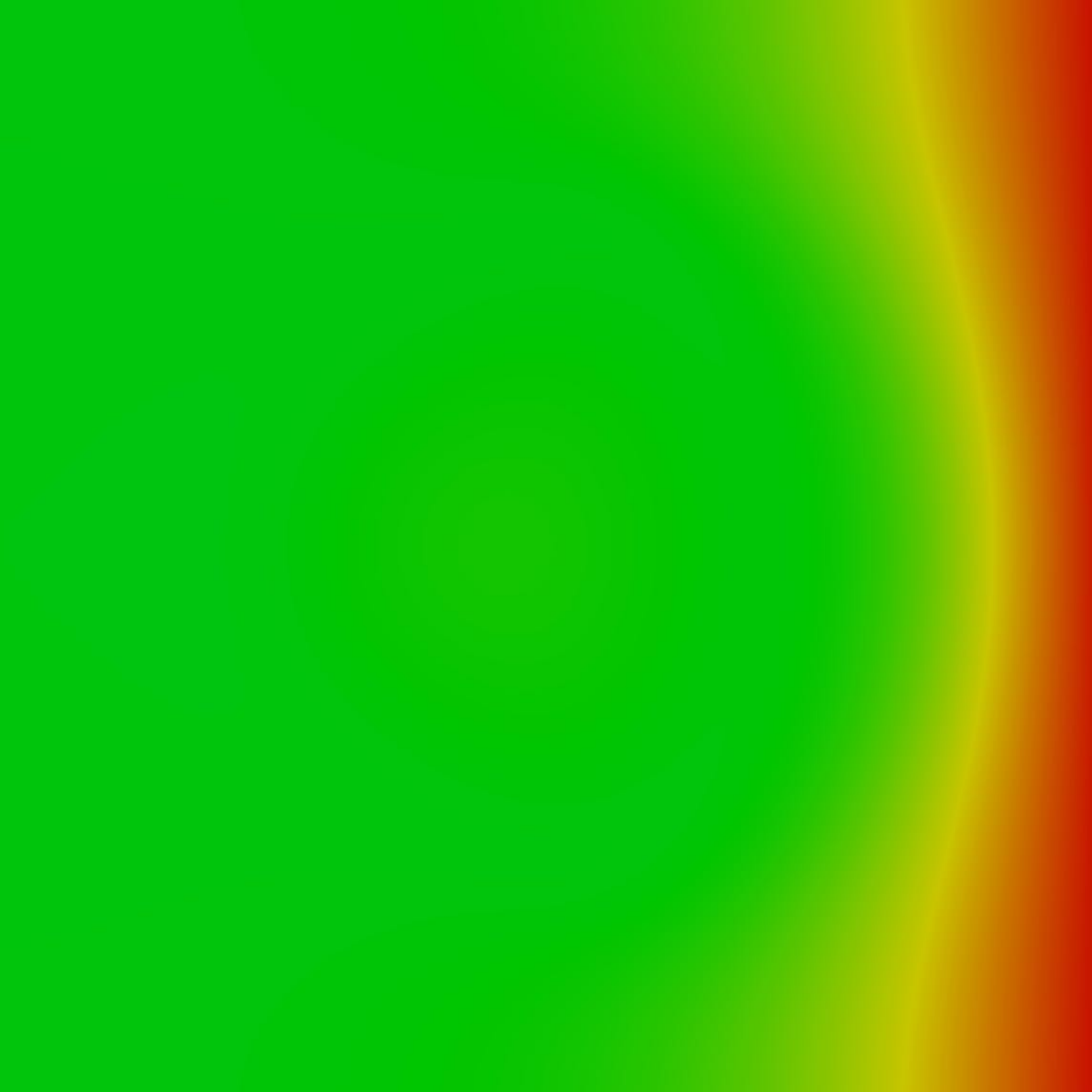}\\
	\hspace{.59cm} \begin{tikzpicture}
	\begin{axis}[
	hide axis,
	scale only axis,
	height=0pt,
	width=0pt,
	colormap name=rainbow,
	colorbar horizontal,
	point meta min=0,
	point meta max=1,
	colorbar style={
		samples=100,
		height=.5cm,
		xtick={0,0.5,1},
		width=10cm
	},
	]
	\addplot [draw=none] coordinates {(0,0)};
	\end{axis}
	\end{tikzpicture} 
\vspace{-0.1cm}
	\caption{Simulation of the ECM density $\theta$ and the nutrient concentration $\phi_\sigma$ in the local model in two dimensions; their evolution is shown at the times $t\in \{0,5,10,15\}$}
	\label{FigureLocalNut2D}
\end{figure}
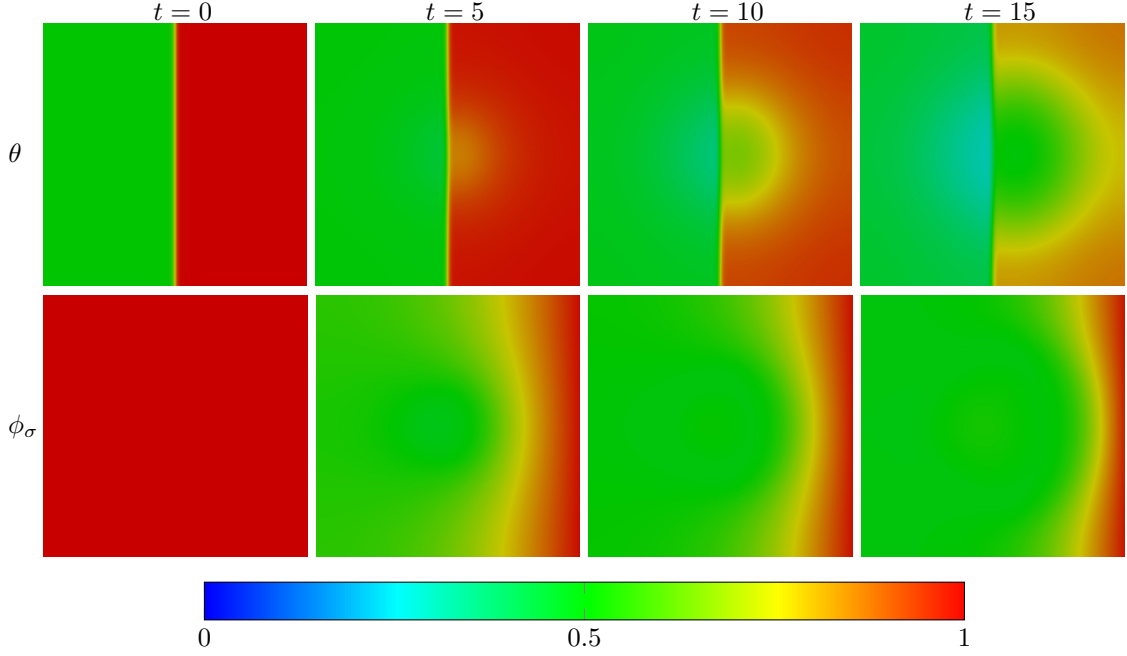

\subsection{Local model in three dimensions}

The simulation of the extracellular matrix density in the three dimensional domain $\Omega=(-1,1)^3$ at the times $t \in \{8,11\}$ is illustrated in Figure \ref{FigureLocalECM3D}. Additionally, an isosurface of the tumor volume fraction $\phi_T$ at $0.7$ is shown in the same plots.

At time $t=0$, the top part of the domain has a higher ECM density $\theta=1$ than the lower part with $\theta=0.5$, similarly to the initial data in the two-dimensional case, see Figure \ref{FigureLocalNut2D}. At $t=8$ and $t=11$, one observes in Figure \ref{FigureLocalECM3D} that the ECM density has degraded around the tumor volume, similar to the two-dimensional case.

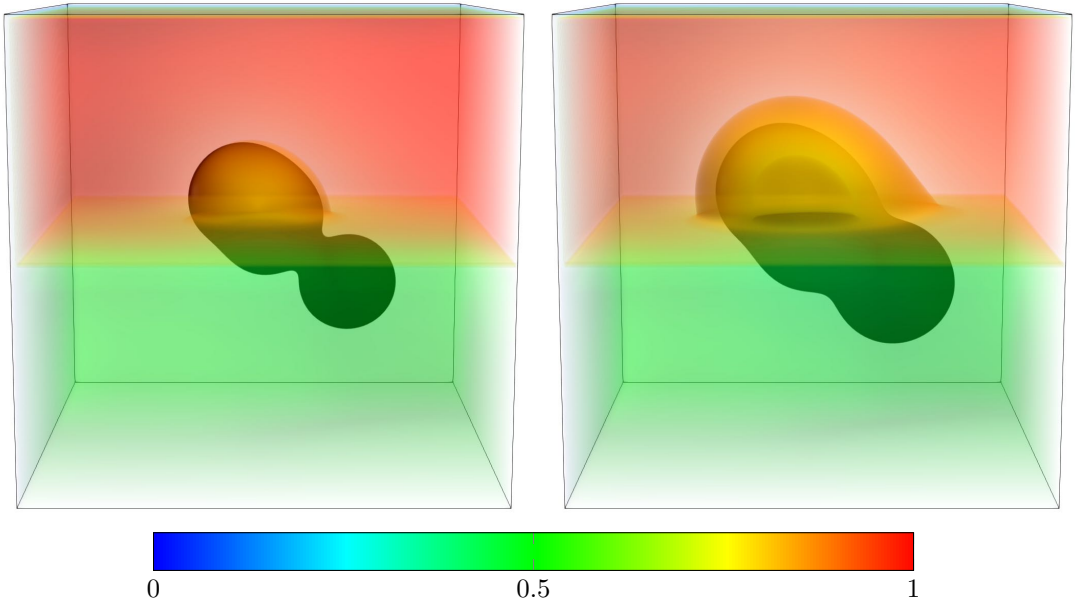
\begin{figure}[H]
\centering
	\includegraphics[trim={0cm 0cm 0cm 4cm}, clip,width=.46\textwidth]{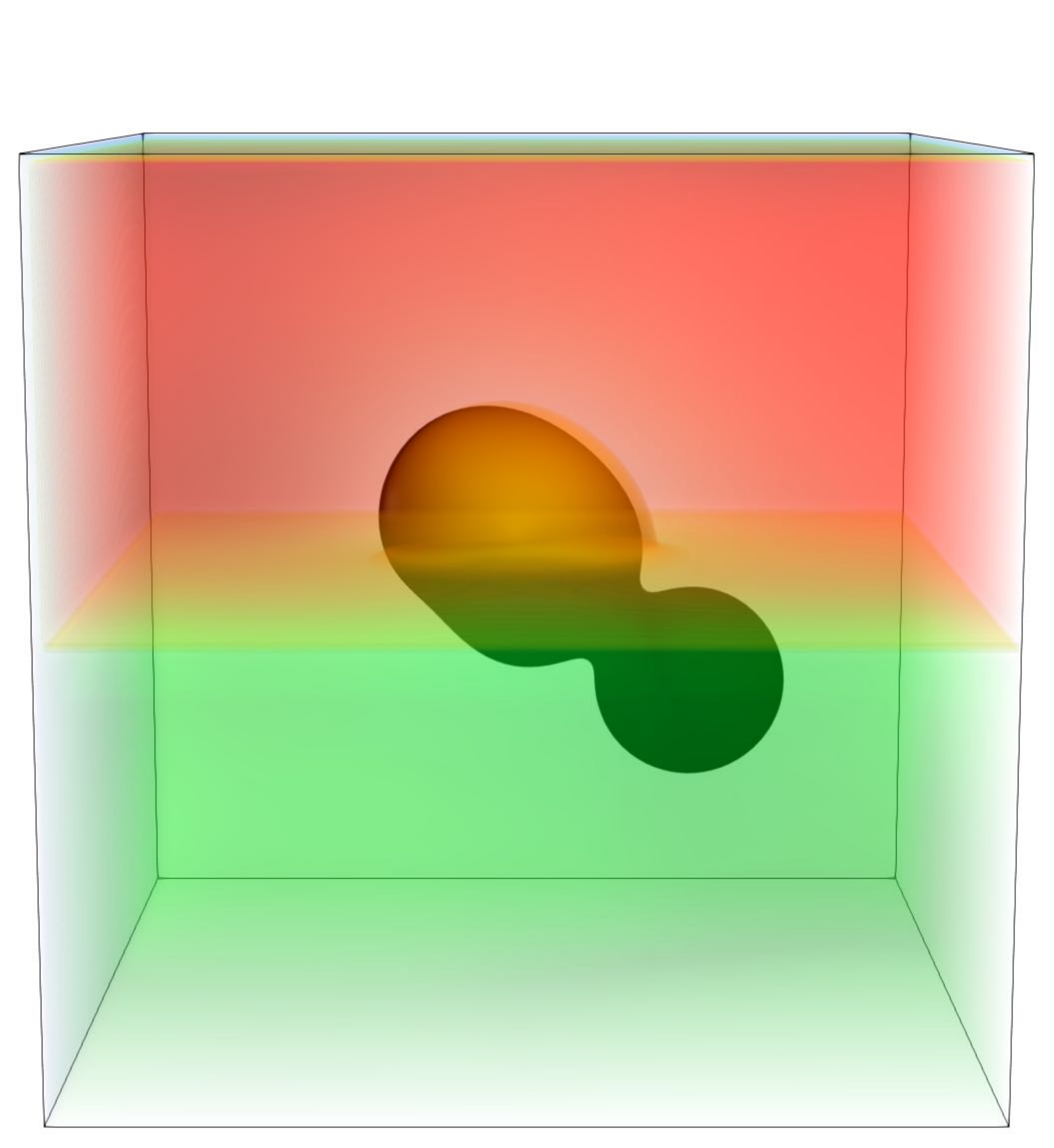}
    \includegraphics[trim={0cm 0cm 0cm 4cm}, clip,width=.46\textwidth]{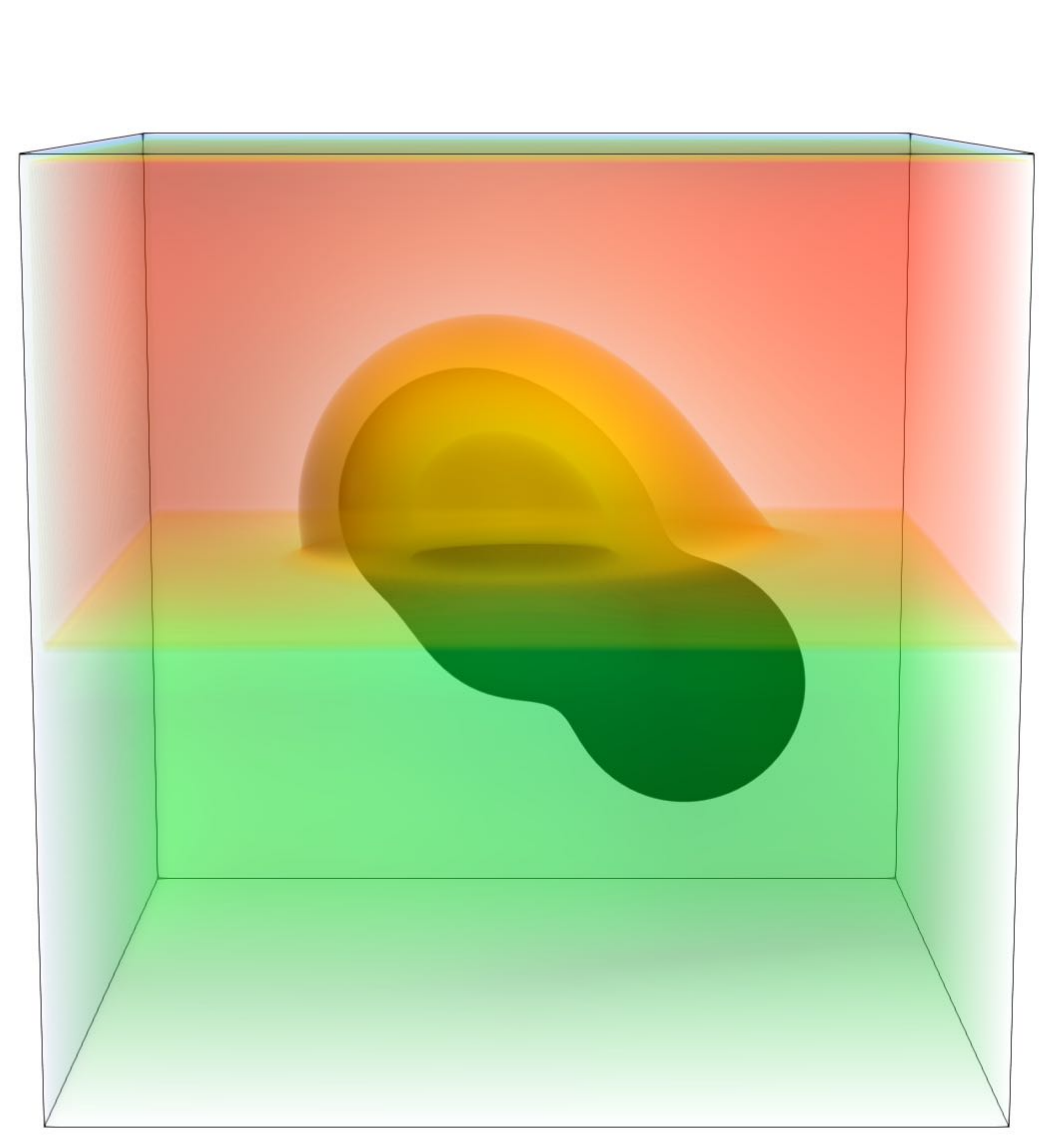}\vspace{-0.15cm}
	\begin{tikzpicture}
	\begin{axis}[
	hide axis,
	scale only axis,
	height=0pt,
	width=0pt,
	colormap name=rainbow,
	colorbar horizontal,
	point meta min=0,
	point meta max=1,
	colorbar style={
		samples=100,
		height=.5cm,
		xtick={0,0.5,1},
		width=10cm
	},
	]
	\addplot [draw=none] coordinates {(0,0)};
	\end{axis}
	\end{tikzpicture} \vspace{-0.1cm}
	\caption{Simulation of the extracellular matrix density $\theta$ in the three-dimensional domain $\Omega=(-1,1)^3$ together with the isosurface of the tumor volume fraction $\phi_T$ at $0.7$, at the times $t=8$ (left plot) and $t=11$ (right plot)}
		\label{FigureLocalECM3D}
\end{figure}

The evolution of the volume fractions of tumor cells $\phi_T$ and necrotic cell $\phi_N$ in the three dimensional case is depicted in Figure \ref{FigureLocalTum3D} below. As initial data we take two separated elliptic-shaped tumor volume fractions, which start to connect at $t=5$. At the initial time there are no necrotic cells. They begin to form at $t=6.5$ and already inhabit a large portion of the tumor volume fraction at $t=11$, as seen in Figure \ref{FigureLocalTum3D}.

		\begin{figure}[H]
	\centering
\begin{tikzpicture}
\draw (0, 0) node[inner sep=0] {{\includegraphics[clip, trim=4cm 3cm 3cm 3cm,width=.46\textwidth]{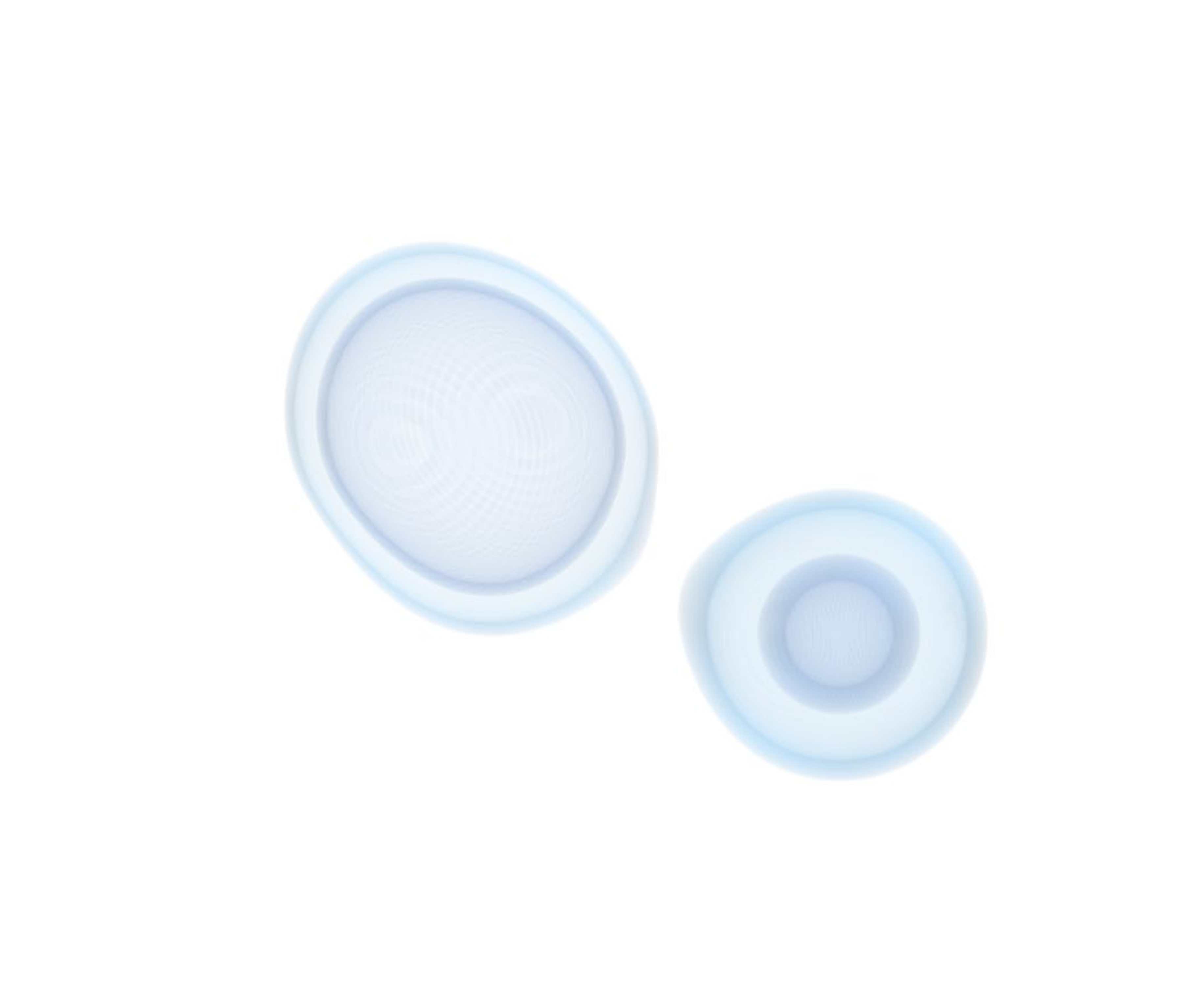}}};
\draw[line width=.7pt, ->]  (xyz cs:x=-3.1,y=-2.5,z=0) -- (xyz cs:x=-2.1,y=-2.5,z=0) node[pos=1.25] {$x_2$};
\draw[line width=.7pt,->] (xyz cs:x=-3.1, y=-2.5, z=0) -- (xyz cs:x=-3.1, y=-1.5, z=0) node[pos=1.15] {$x_1$};
\draw[line width=.7pt,->] (xyz cs:x=-3.1,y=-2.5,z=0) -- (xyz cs:x=-3.1,y=-2.5,z=-1) node[pos=1.35] {$x_3$};
\end{tikzpicture}
\begin{tikzpicture}
\draw (0, 0) node[inner sep=0] {{\includegraphics[clip, trim=4cm 3cm 3cm 3cm,width=.46\textwidth]{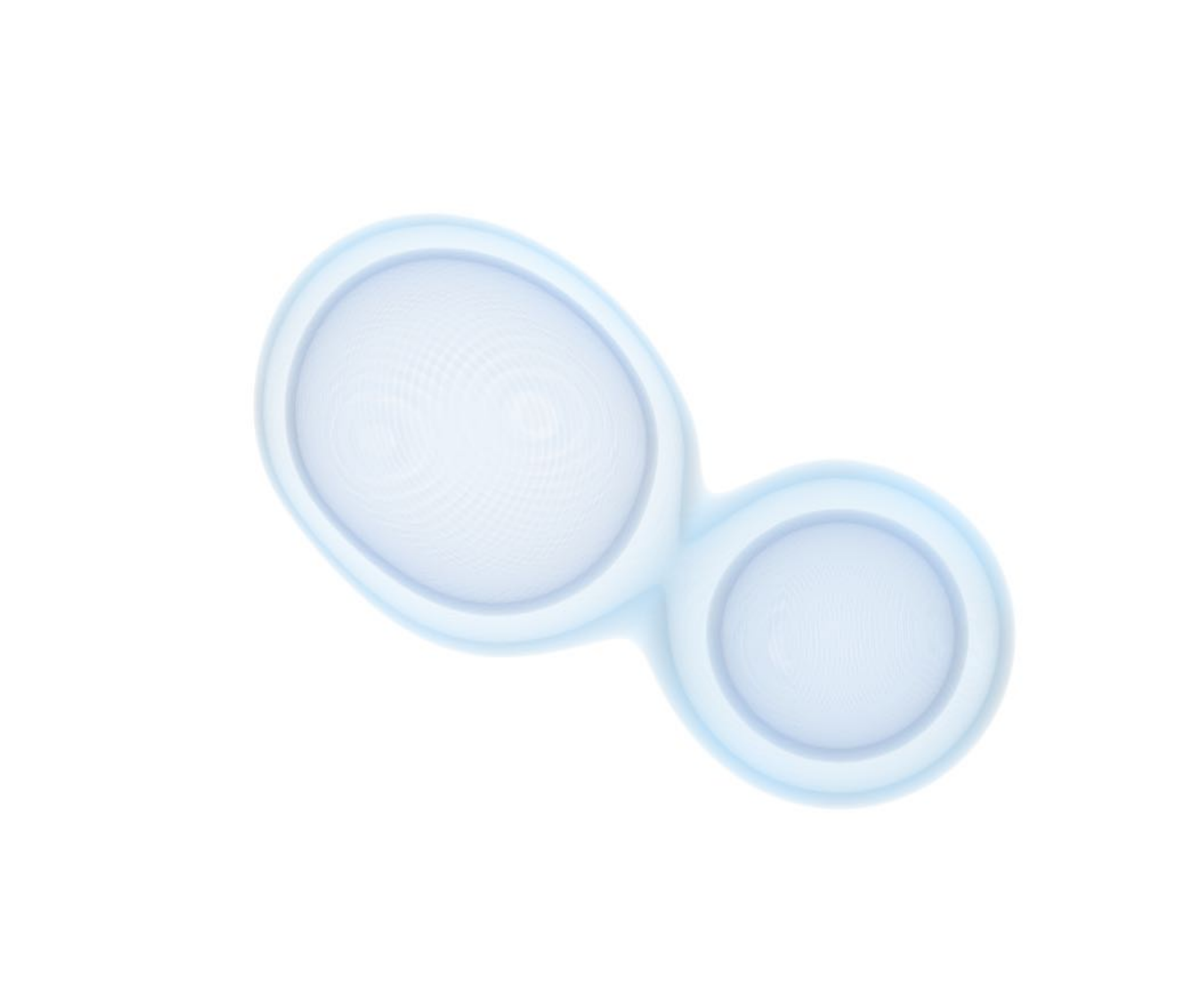}}};
\draw[line width=.7pt, ->]  (xyz cs:x=-3.1,y=-2.5,z=0) -- (xyz cs:x=-2.1,y=-2.5,z=0) node[pos=1.25] {$x_2$};
\draw[line width=.7pt,->] (xyz cs:x=-3.1, y=-2.5, z=0) -- (xyz cs:x=-3.1, y=-1.5, z=0) node[pos=1.15] {$x_1$};
\draw[line width=.7pt,->] (xyz cs:x=-3.1,y=-2.5,z=0) -- (xyz cs:x=-3.1,y=-2.5,z=-1) node[pos=1.35] {$x_3$};
\end{tikzpicture}  \\[-0.1cm]
\begin{tikzpicture}
\draw (0, 0) node[inner sep=0] {{\includegraphics[clip, trim=4cm 3cm 3cm 3cm,width=.46\textwidth]{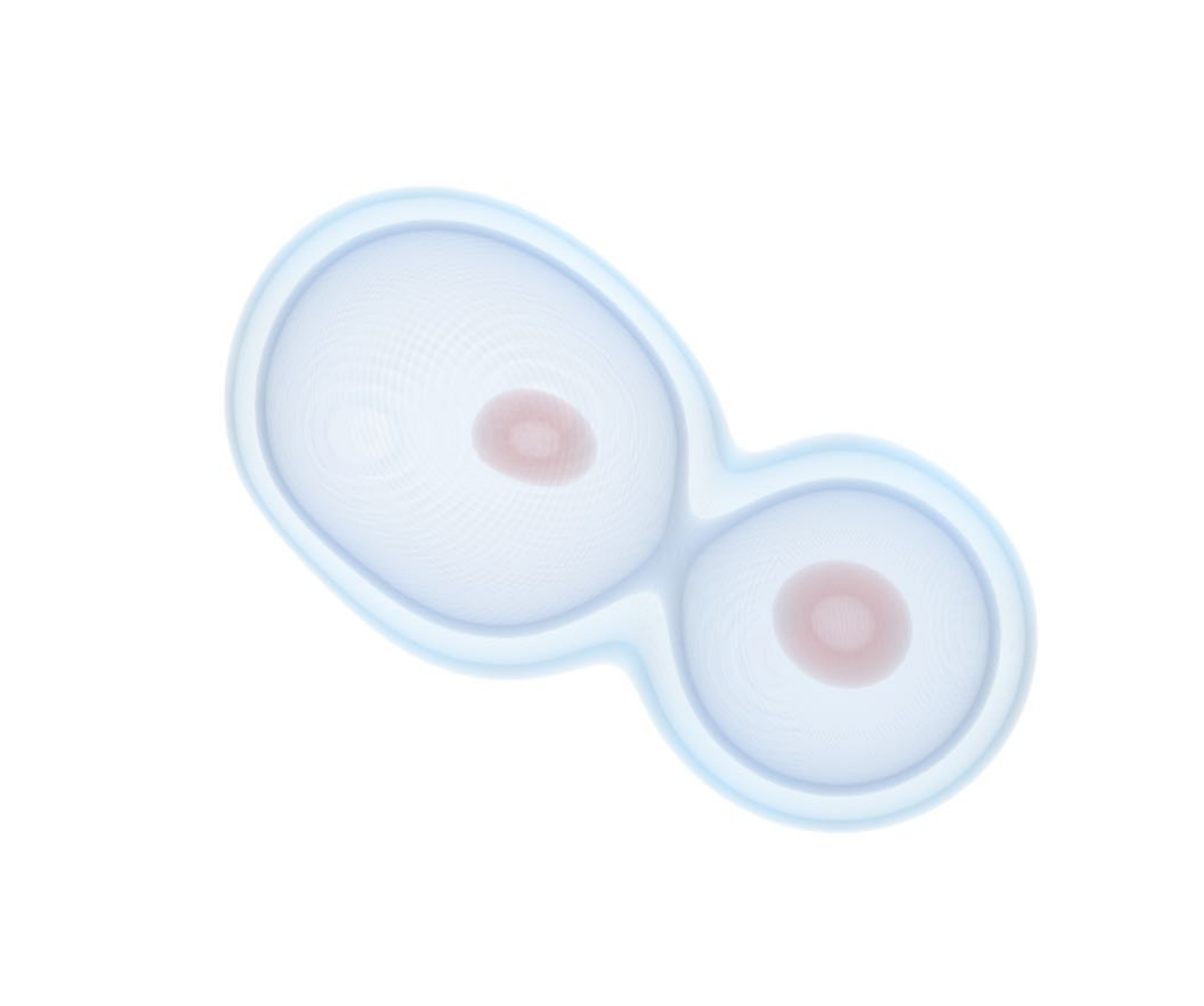}}};
\draw[line width=.7pt, ->]  (xyz cs:x=-3.1,y=-2.5,z=0) -- (xyz cs:x=-2.1,y=-2.5,z=0) node[pos=1.25] {$x_2$};
\draw[line width=.7pt,->] (xyz cs:x=-3.1, y=-2.5, z=0) -- (xyz cs:x=-3.1, y=-1.5, z=0) node[pos=1.15] {$x_1$};
\draw[line width=.7pt,->] (xyz cs:x=-3.1,y=-2.5,z=0) -- (xyz cs:x=-3.1,y=-2.5,z=-1) node[pos=1.35] {$x_3$};
\end{tikzpicture}
\begin{tikzpicture}
\draw (0, 0) node[inner sep=0] {{\includegraphics[clip, trim=4cm 3cm 3cm 3cm,width=.46\textwidth]{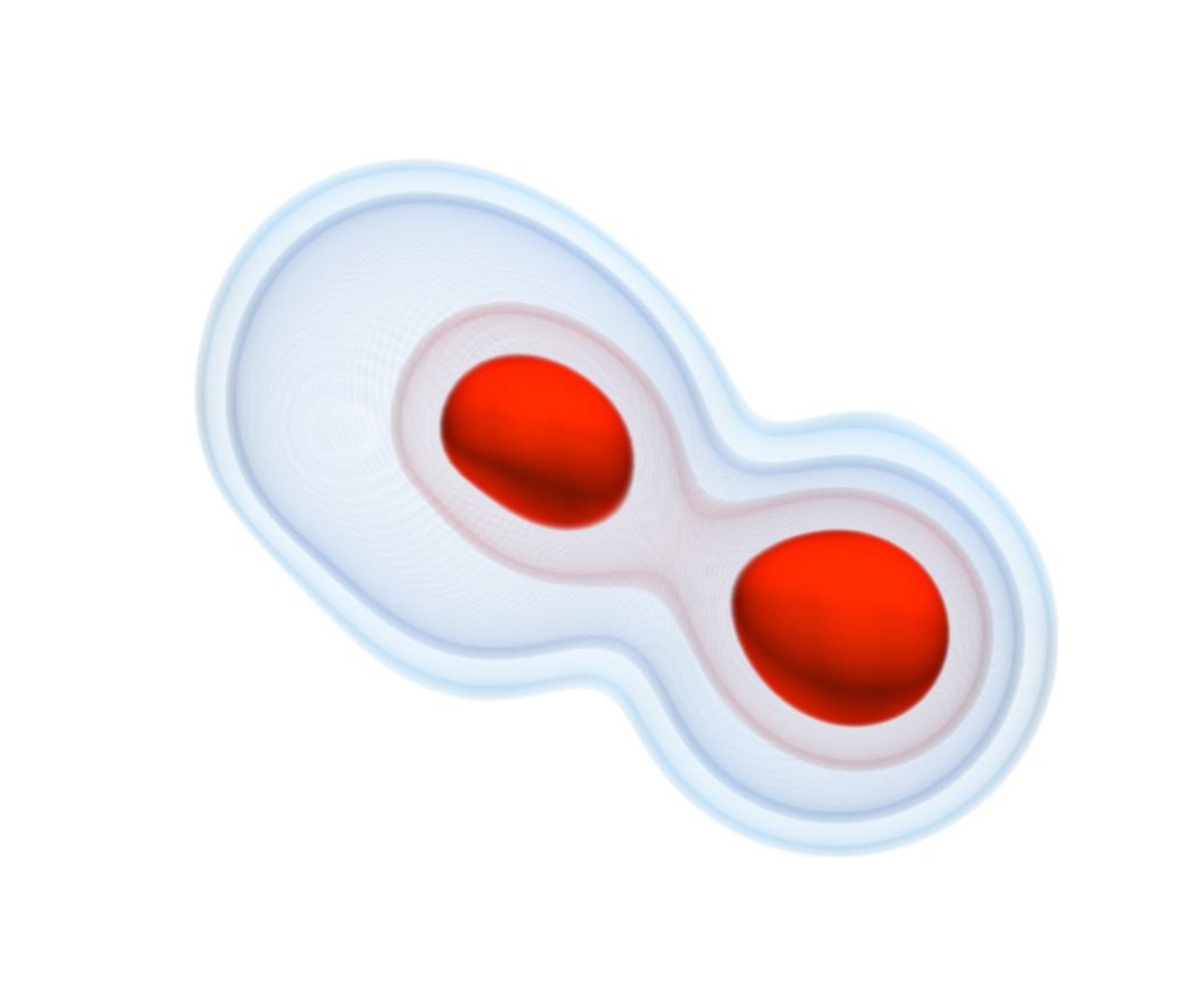}}};
\draw[line width=.7pt, ->]  (xyz cs:x=-3.1,y=-2.5,z=0) -- (xyz cs:x=-2.1,y=-2.5,z=0) node[pos=1.25] {$x_2$};
\draw[line width=.7pt,->] (xyz cs:x=-3.1, y=-2.5, z=0) -- (xyz cs:x=-3.1, y=-1.5, z=0) node[pos=1.15] {$x_1$};
\draw[line width=.7pt,->] (xyz cs:x=-3.1,y=-2.5,z=0) -- (xyz cs:x=-3.1,y=-2.5,z=-1) node[pos=1.35] {$x_3$};
\end{tikzpicture}  \\[-0.1cm]
\begin{tikzpicture}
\draw (0, 0) node[inner sep=0] {{\includegraphics[clip, trim=4cm 3cm 3cm 3cm,width=.46\textwidth]{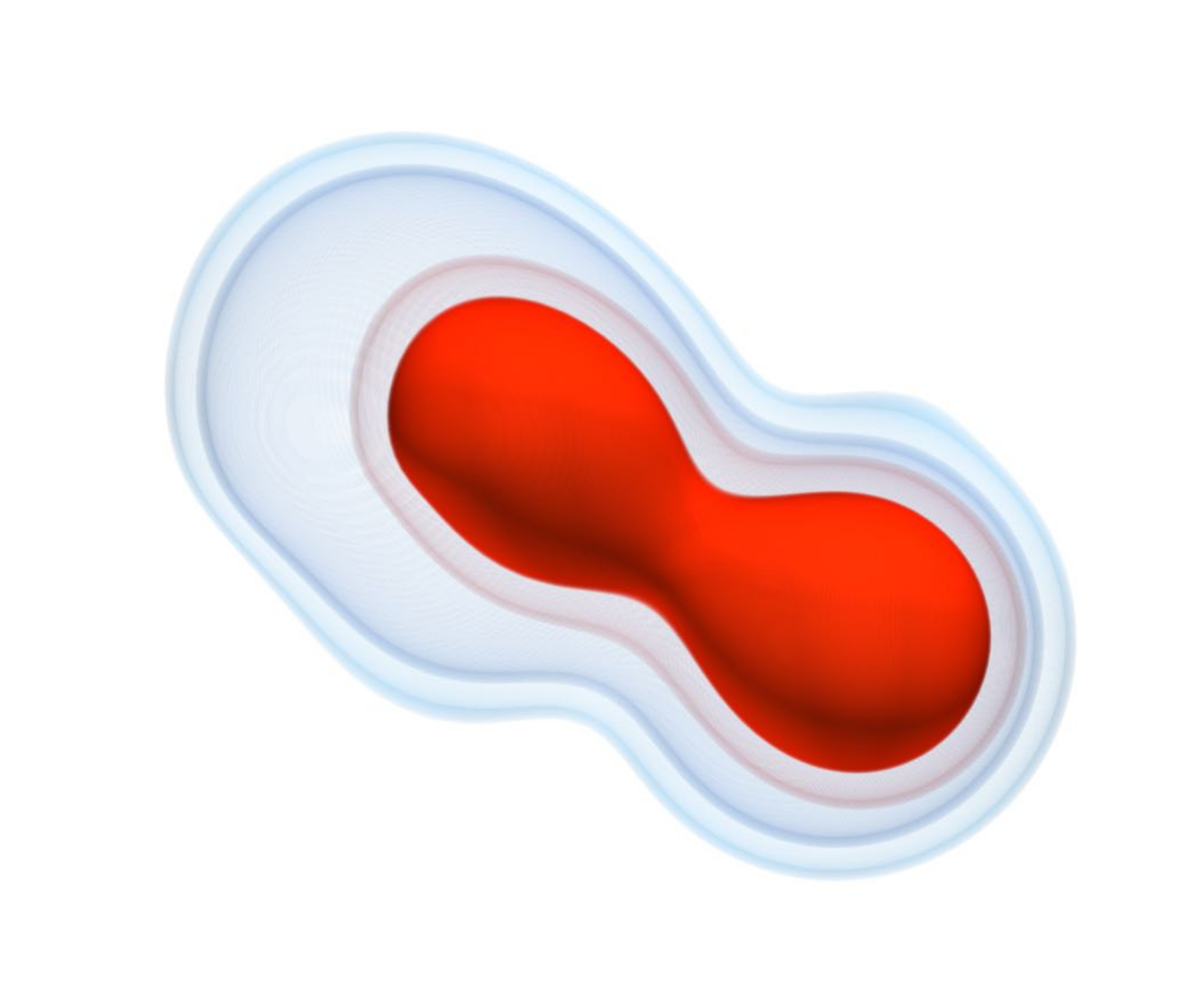}}};
\draw[line width=.7pt, ->]  (xyz cs:x=-3.1,y=-2.5,z=0) -- (xyz cs:x=-2.1,y=-2.5,z=0) node[pos=1.25] {$x_2$};
\draw[line width=.7pt,->] (xyz cs:x=-3.1, y=-2.5, z=0) -- (xyz cs:x=-3.1, y=-1.5, z=0) node[pos=1.15] {$x_1$};
\draw[line width=.7pt,->] (xyz cs:x=-3.1,y=-2.5,z=0) -- (xyz cs:x=-3.1,y=-2.5,z=-1) node[pos=1.35] {$x_3$};
\end{tikzpicture}
\begin{tikzpicture}
\draw (0, 0) node[inner sep=0] {{\includegraphics[clip, trim=4cm 3cm 3cm 3cm,width=.46\textwidth]{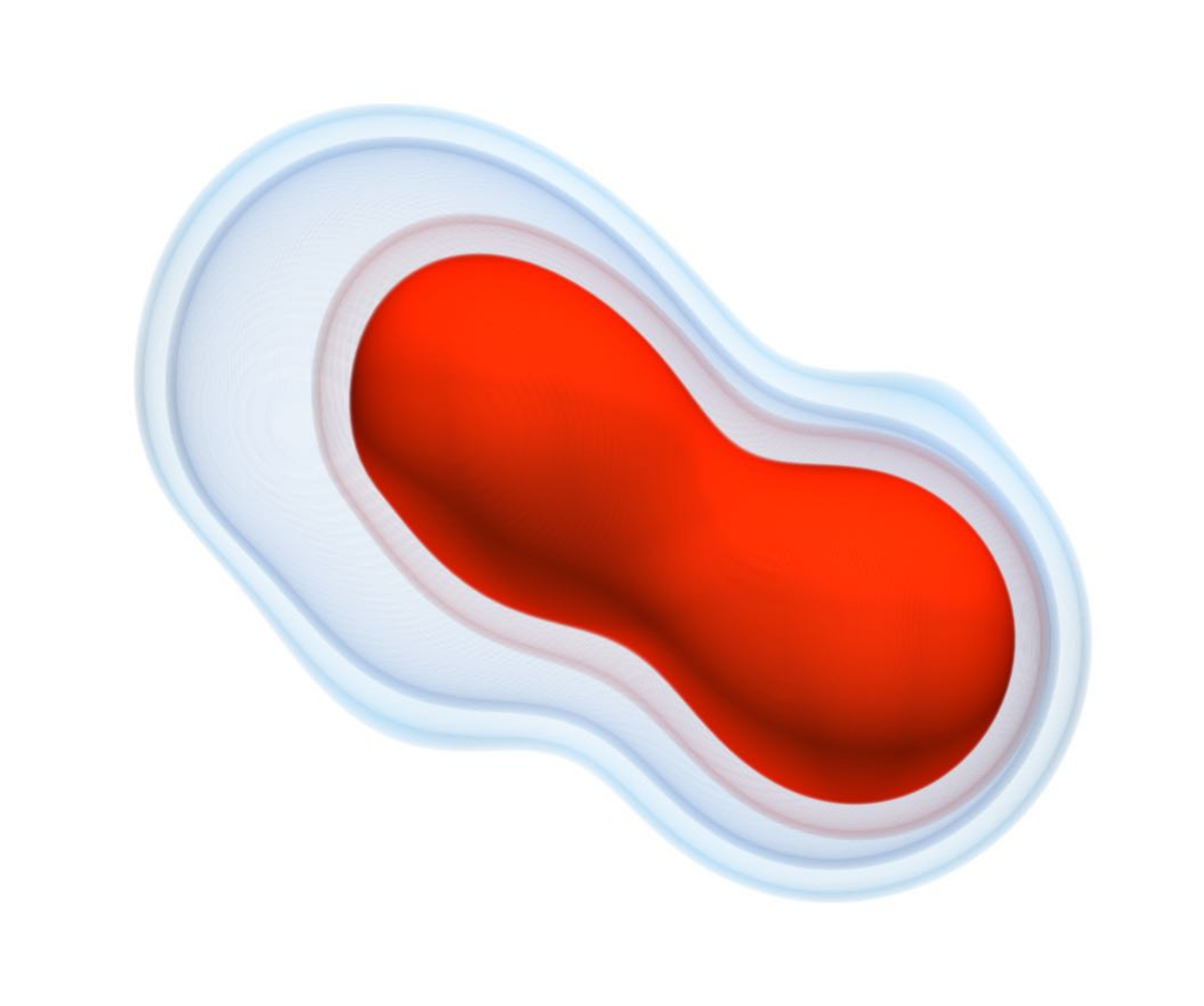}}};
\draw[line width=.7pt, ->]  (xyz cs:x=-3.1,y=-2.5,z=0) -- (xyz cs:x=-2.1,y=-2.5,z=0) node[pos=1.25] {$x_2$};
\draw[line width=.7pt,->] (xyz cs:x=-3.1, y=-2.5, z=0) -- (xyz cs:x=-3.1, y=-1.5, z=0) node[pos=1.15] {$x_1$};
\draw[line width=.7pt,->] (xyz cs:x=-3.1,y=-2.5,z=0) -- (xyz cs:x=-3.1,y=-2.5,z=-1) node[pos=1.35] {$x_3$};
\end{tikzpicture}  
	\begin{tikzpicture}
	\begin{axis}[
	hide axis,
	scale only axis,
	height=0pt,
	width=0pt,
	colormap name=blue,
	colorbar horizontal,
	point meta min=0,
	point meta max=1,
	colorbar style={
	    ylabel=$\phi_T$,
	    y label style={rotate=270},
		samples=100,
		height=.5cm,
		xtick={0,0.5,1},
		width=6cm
	},
	]
	\addplot [draw=none] coordinates {(0,0)};
	\end{axis}
	\end{tikzpicture}
		\begin{tikzpicture}
	\begin{axis}[
	hide axis,
	scale only axis,
	height=0pt,
	width=0pt,
	colormap name=red,
	colorbar horizontal,
	point meta min=0,
	point meta max=1,
	colorbar style={
	    ylabel=$\phi_N$,
	    y label style={rotate=270},
		samples=100,
		height=.5cm,
		xtick={0,0.5,1},
		width=6cm
	},
	]
	\addplot [draw=none] coordinates {(0,0)};
	\end{axis}
	\end{tikzpicture}
	\caption{Simulation of volume fractions of tumor cells $\phi_T$ and necrotic cells $\phi_N$ in a 3D domain, isosurfaces of $0.2$ and $0.4$ of each volume fraction at times $t\in \{3.5, 5, 6.5, 8, 9.5, 11\}$ are shown}
		\label{FigureLocalTum3D}
\end{figure}
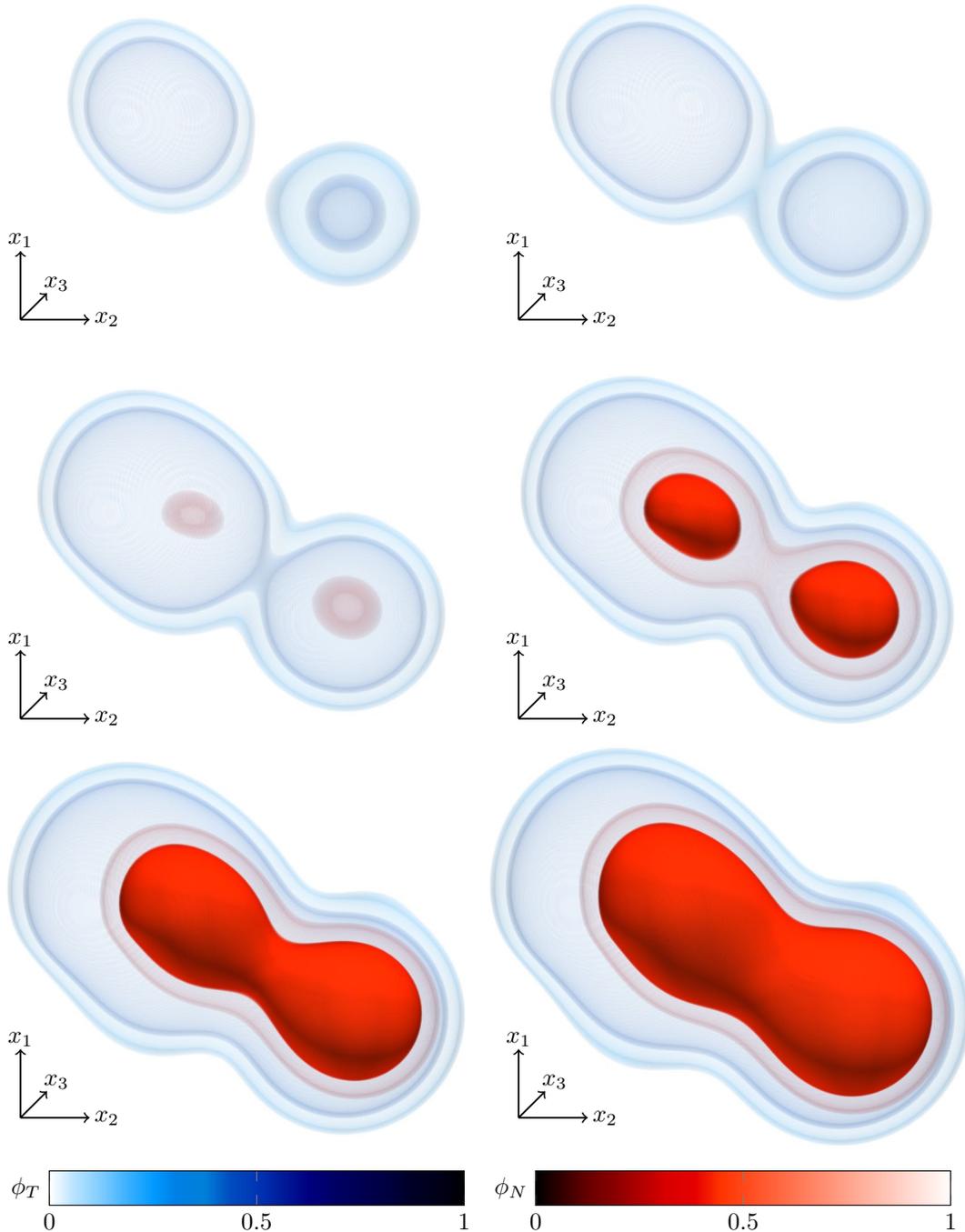

\newpage

\subsection{Comparison to the nonlocal model}
In this section, we compare the simulation of the tumor volume fraction $\phi_T$ in the local and nonlocal model (\ref{mod_problem_loc}), that means in the local model we choose for the adhesion flux $J_\text{loc}=\chi_H \phi_V \nabla \theta$ and for the nonlocal model $J_\text{nonloc}=\chi_H \phi_V k*\theta$, as introduced in (\ref{adhesion_flux}). In the case of the nonlocal adhesion-based haptotaxis effect, we have to select an appropriate vector-valued kernel function $k$. In the existence proof of the nonlocal model we only had to assume $k \in L^1(\mathbb{R}^d)$ and no additional requirements on its representation. Following \cite{gerisch2008mathematical,gerisch2010approximation,chaplain2011mathematical}, we choose a kernel function $k_\eps$, $\eps>0$ indicating some parameter, such that it approximates the gradient-based haptotaxis effect as $\eps \to 0$. See also \cite{du2013nonlocal,mengesha2015localization} for different choices for nonlocal gradient operators.

In tumor growth models involving nonlocal cell-to-cell adhesion effects, it is a standard procedure to replace the term $\frac12 \eps_T^2 |\nabla \phi_T(x)|^2$ in the Ginzburg--Landau free energy functional (\ref{Ginzburg}) by 
\begin{equation} \frac14 \int_\Omega J(x-y)(\phi_T(x)-\phi_T(y))^2\, \dd y. \label{Eq_NonlocalEnergy} \end{equation}
As shown in \cite{frigeri2015a}, choosing $J(x-y)=j^{d+2} \chi_{[0,1]}(|j(x-y)|^2)$ and letting $j \to \infty$, one returns to the Ginzburg--Landau free energy functional, where the interfacial parameter is expressed by $\eps_T^2=\frac{2}{d} \int_{\mathbb{R}^d} J(|z|^2)|z|^2 \dd z$. Therefore, one can interpret the classical Cahn--Hilliard equation as an approximation of its nonlocal version. 

Taking the Gateaux derivative of the nonlocal energy functional results in the chemical potential $\mu$. In particular, the term in (\ref{Eq_NonlocalEnergy}) becomes $$\phi_T\cdot J*1 - J*\phi_T$$ instead of $-\eps_T^2 \Delta \phi_T$ in the case of the local Ginzburg--Landau free energy functional. This suggests for the gradient operator the following approximation:
$$\begin{aligned} k \circledast \theta (x) \, &\!\!:= (k*\theta)(x) - \theta(x) \cdot (k*1)(x) \\ &= \int_{\mathbb{R}^d} k(x-y) (\theta(y)-\theta(x)) \dd y
\\&\approx \int_{\mathbb{R}^d} k(x-y) (\nabla \theta(x) \cdot (y-x)  ) \dd y
\\&=\nabla \theta(x) \int_{\mathbb{R}^d} (y-x) \cdot k(x-y) \dd y  
\\&= \nabla \theta(x),
\end{aligned}$$
where we chose $k$ such that $x k(-x)$ is a Dirac sequence with the typical property $\int_{\mathbb{R}^d} x k(-x) \, \dd x=1$. We impose the representation 
\begin{equation} \label{Eq_Kernel}
k(x)=- \omega(\eps) x \chi_{[0,\eps]} (|x|_\infty),
\end{equation} which gives in the two-dimensional case
$$\int_{\mathbb{R}^2} x k(-x)\, \dd x = \omega(\eps) \int_{-\eps}^{\eps} \int_{-\eps}^\eps (x_1^2+x_2^2) \, \dd x_1 \dd x_2 = \omega(\eps) \frac{8}{3} \eps^4,$$
and defining $\omega(\eps)=\frac{3}{8} \eps^{-4}$ yields the desired normalization property.

Note that $k(x)=-\omega(\eps) x \chi_{[0,\eps]}(|x|_\infty)$ is an odd function and therefore,
$$(k*1)(x)=\int_{\mathbb{R}^d} k(x-y) \dd y =0,$$
and we can write $k \circledast \theta=k * \theta$.

In the following, we numerically investigate the effects of the different haptotaxis parameters $\chi_H$ on the growth of the tumor volume fraction. We distinguish between three different values for $\chi_H$, $\chi_H \in \{5 \cdot 10^{-4}, 10^{-3}, 2 \cdot 10^{-3}\}$. We can observe in Figure \ref{FigureLocalNonlocal} that a lower haptotaxis parameter results in a more circular shape than for a higher $\chi_H$, e.g., for $\chi_H=10^{-3}$, we see that the tumor shape forms a bump at the vertical axis. 
Moreover, we compare the local gradient-based ($\eps=0$) and the nonlocal adhesion-based haptotaxis effect, for which we select $\eps \in \{2.75\cdot 10^{-2}, 5.25 \cdot 10^{-2}\}$ in the definition of the kernel function (\ref{Eq_Kernel}).

\pagebreak

The larger $\eps$, the less sensitive the results are on the three considered values for $\chi_H$. However as we can see in the last column in Figure \ref{FigureLocalNonlocal}, different $(\eps,\chi_H)$ pairings can also yield quite similar results. A larger $\eps$ requires a larger $\chi_H$ to show similar effects as a pairing with smaller $\eps$ and $\chi_H$ values.

The larger $\chi_H$, the more the local and nonlocal model differ from each other. This results from the fact that for $\eps>0$ in the nonlocal model terms involving $\eps^2$ play a more significant role.

	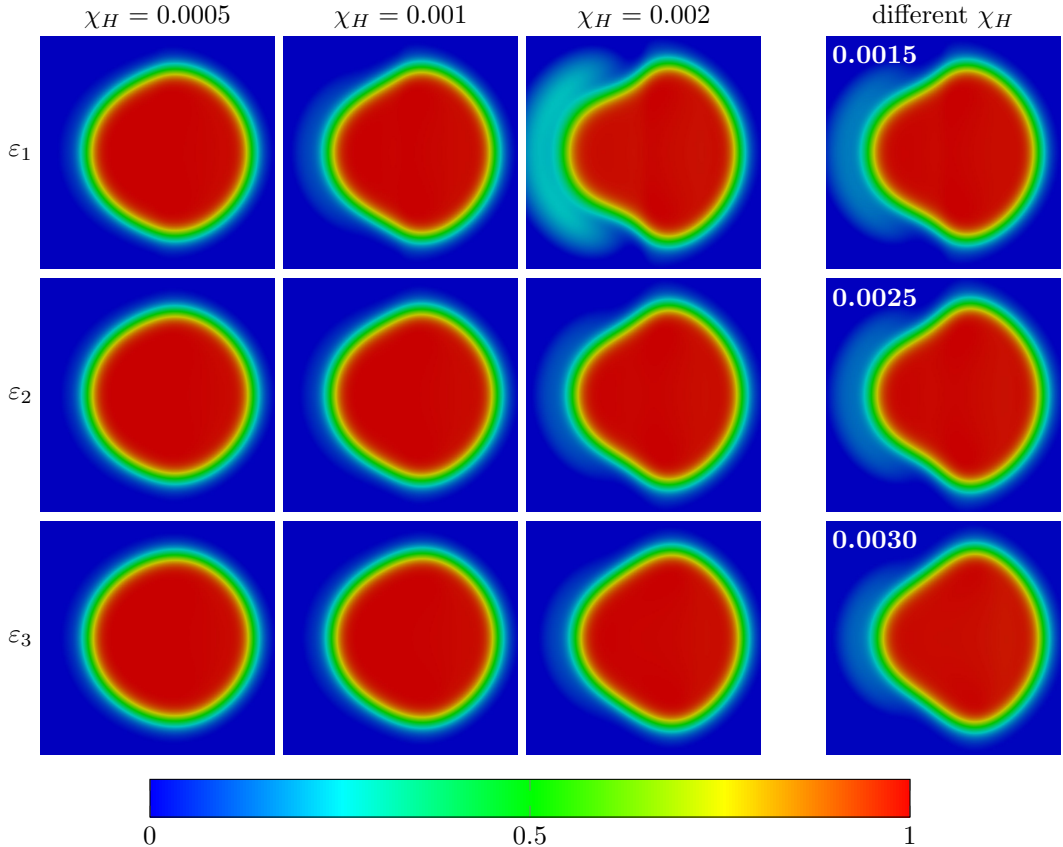
\begin{figure}[H]
	\centering
	\begin{tikzpicture}
\draw (0, 0) node[inner sep=0] {\includegraphics[width=.2\textwidth]{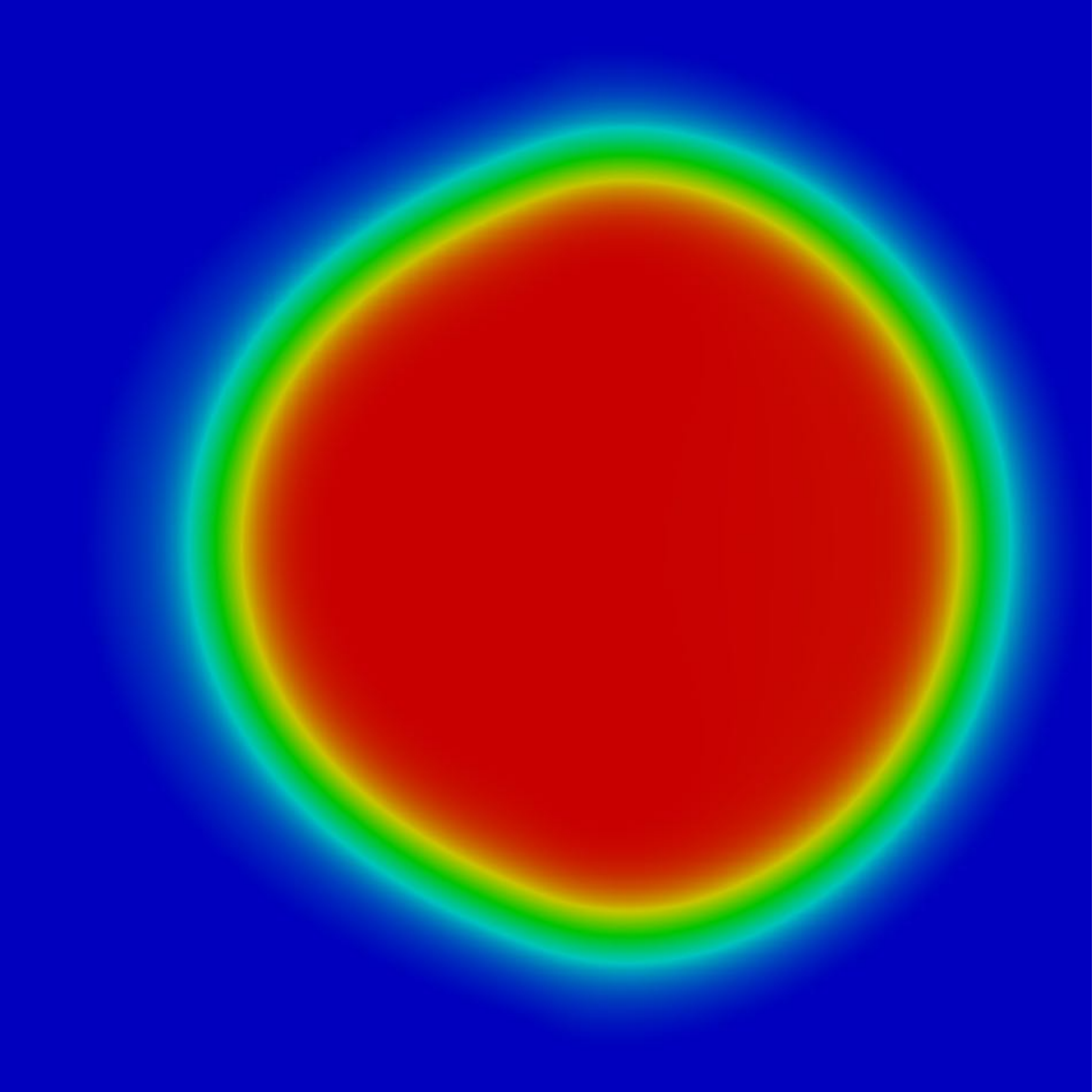}};
\draw (0, 1.8) node {$\chi_{H}=0.0005$};
\draw (-1.8,0)  node {$\varepsilon_1$};
\end{tikzpicture}
\begin{tikzpicture}
\draw (0, 0) node[inner sep=0] {\includegraphics[width=.2\textwidth]{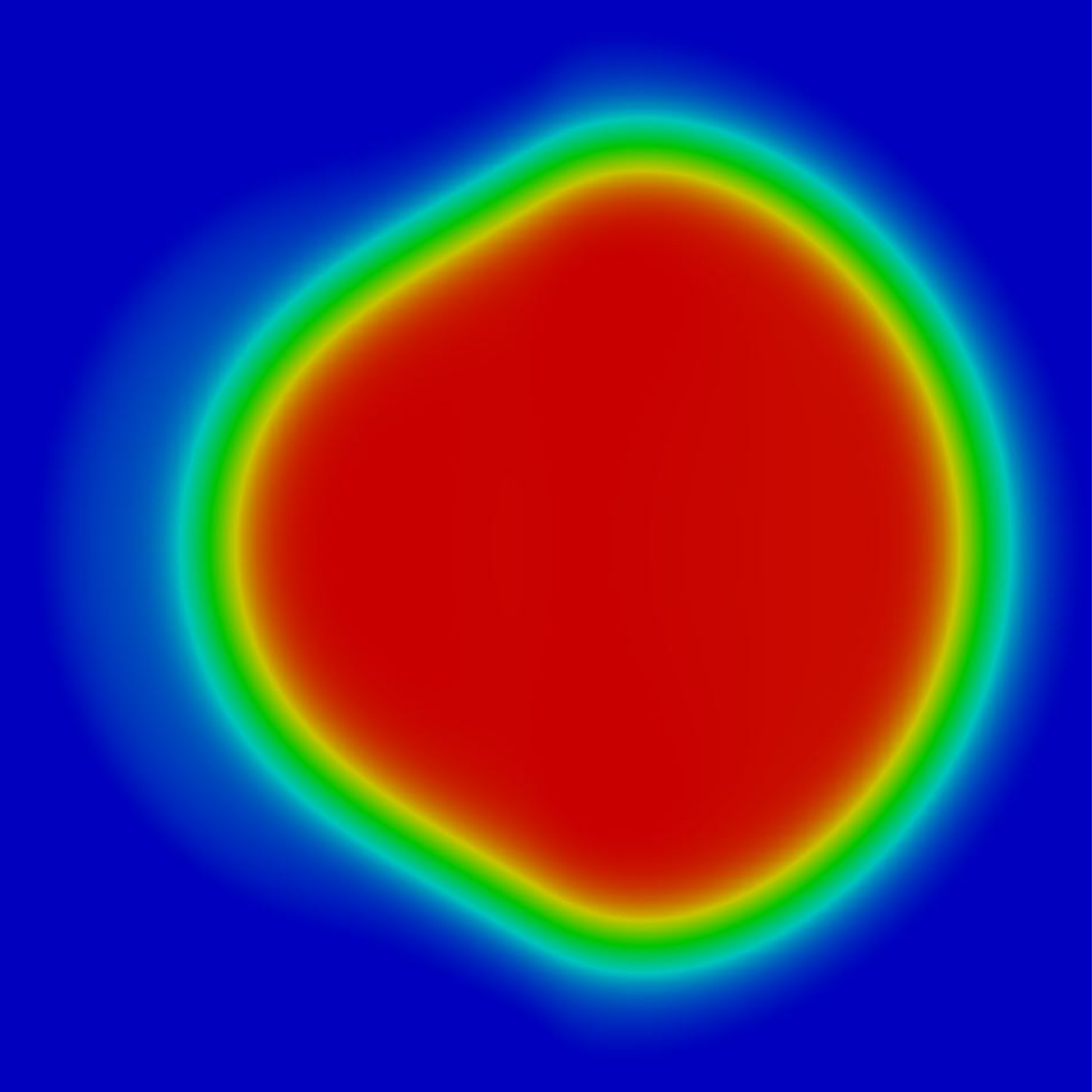}};
\draw (0, 1.8) node {$\chi_{H}=0.001$};
\end{tikzpicture} 
\begin{tikzpicture}
\draw (0, 0) node[inner sep=0] {\includegraphics[width=.2\textwidth]{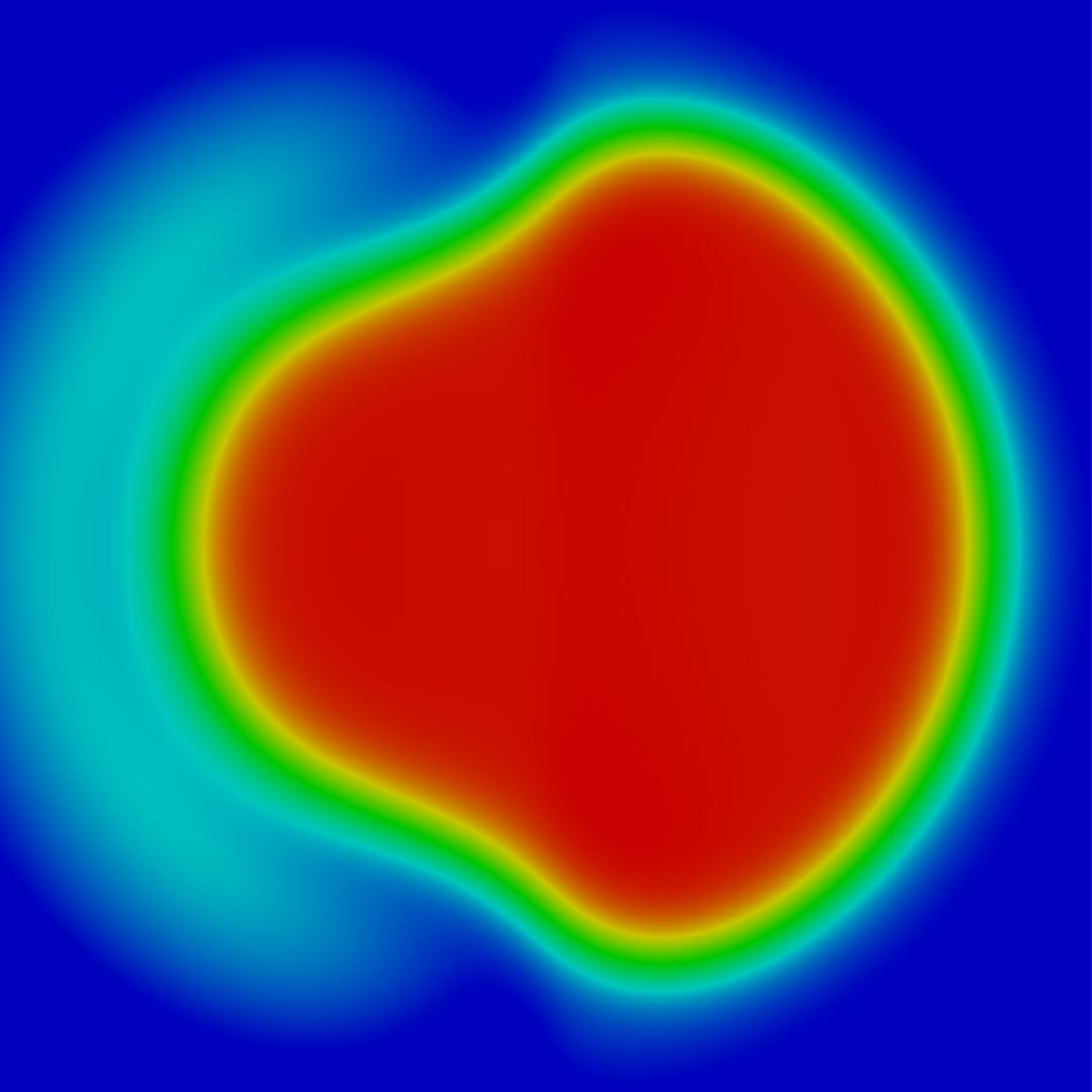}};
\draw (0, 1.8) node {$\chi_{H}=0.002$};
\end{tikzpicture} \qquad
\begin{tikzpicture}
\draw (0, 0) node[inner sep=0] {\includegraphics[width=.2\textwidth]{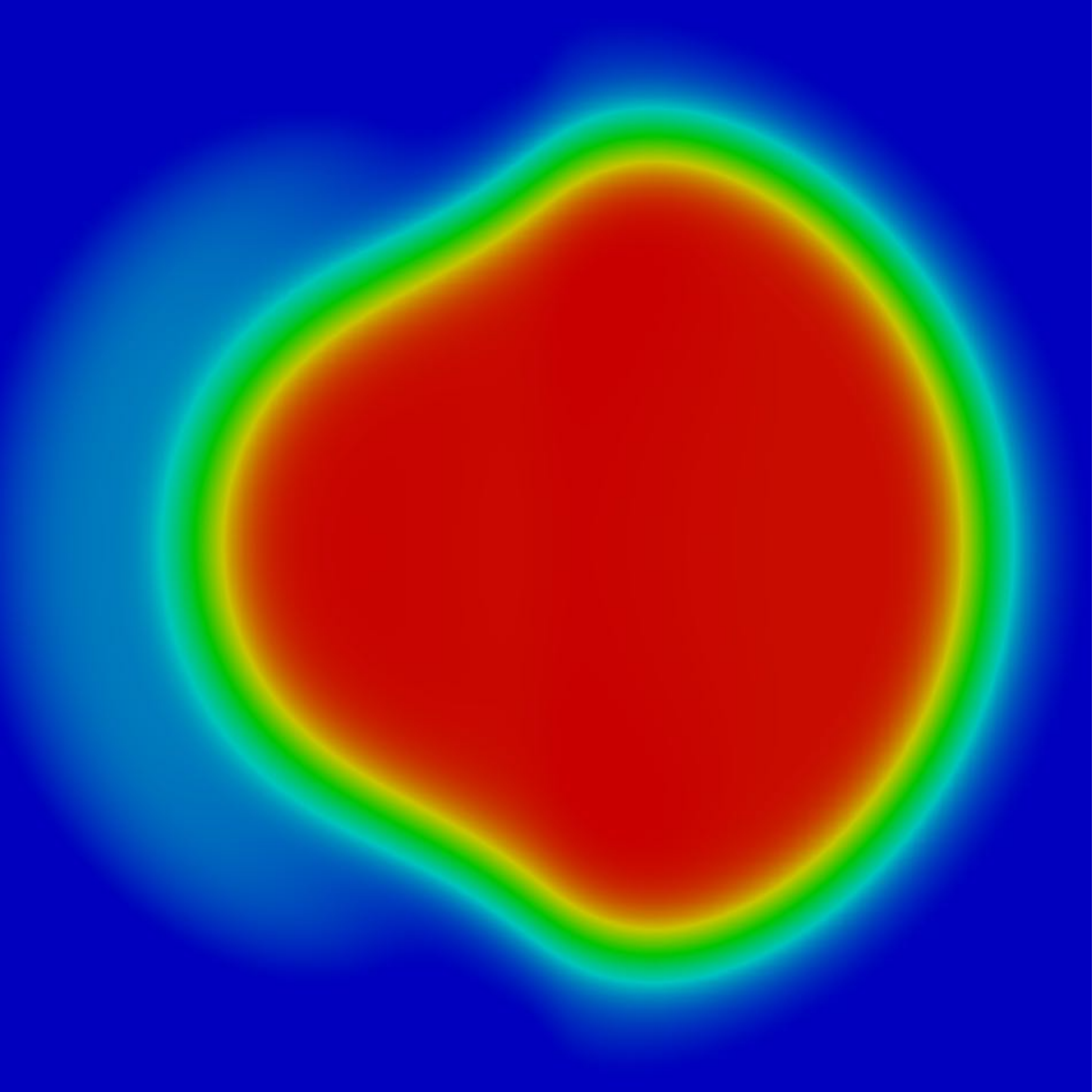}};
\draw (-.9, 1.3) node {$\textbf{\textcolor{white}{0.0015}}$};
\draw (0, 1.8) node {different $\chi_{H}$};
\end{tikzpicture} \\[0.1cm]
\begin{tikzpicture}
\draw (0, 0) node[inner sep=0] {\includegraphics[width=.2\textwidth]{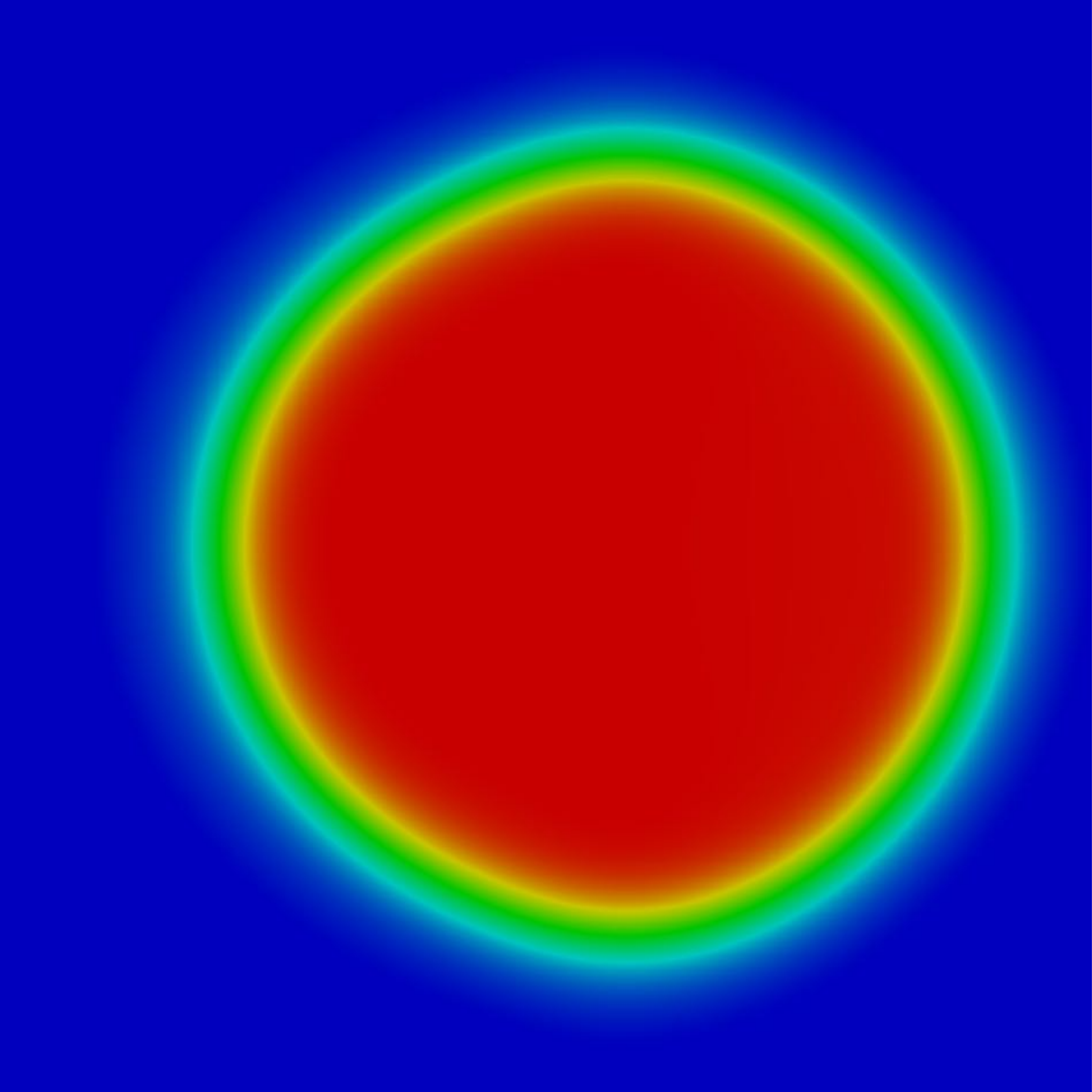}};
\draw (-1.8,0)  node {$\varepsilon_2$};
\end{tikzpicture}
	\includegraphics[width=.2\textwidth]{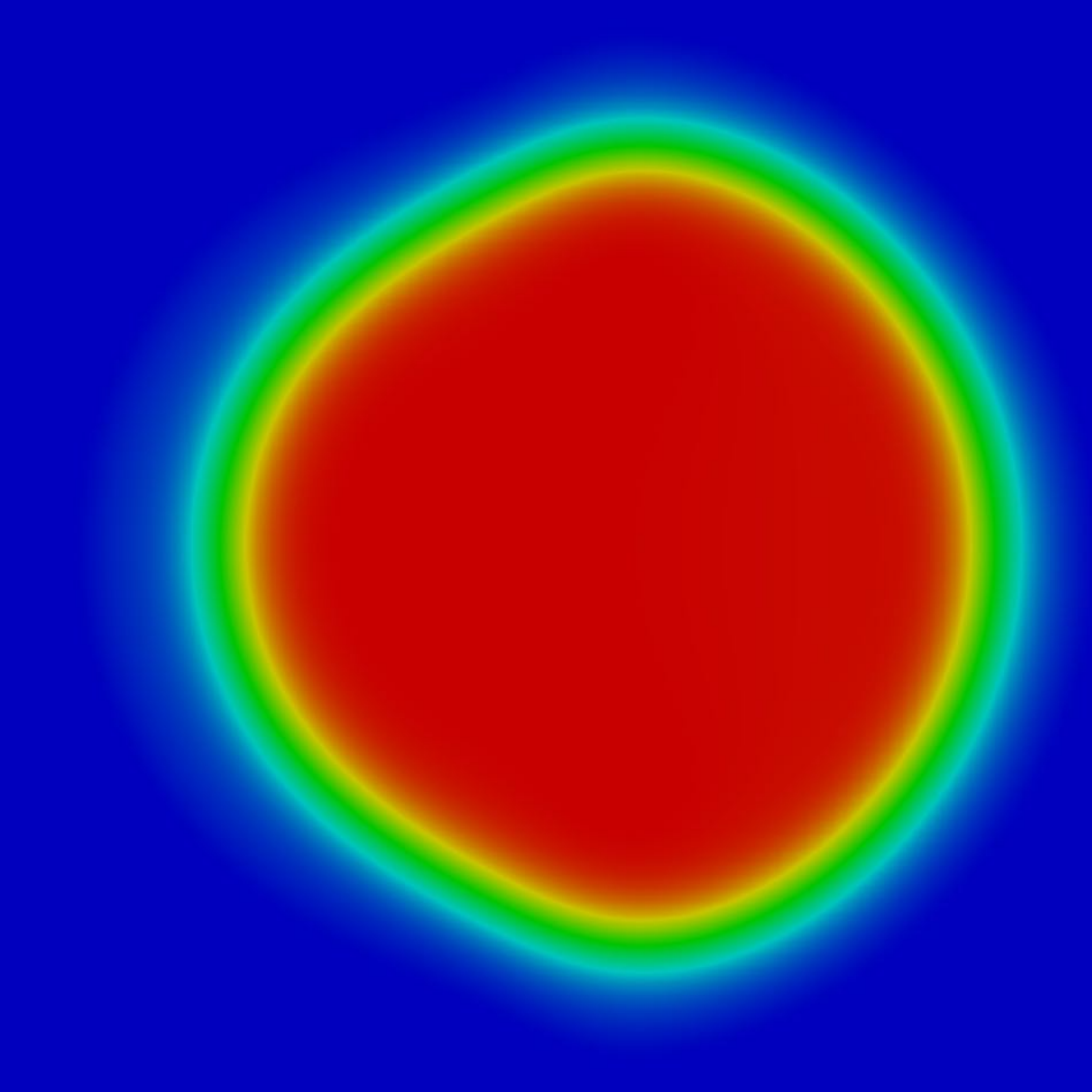}
	\includegraphics[width=.2\textwidth]{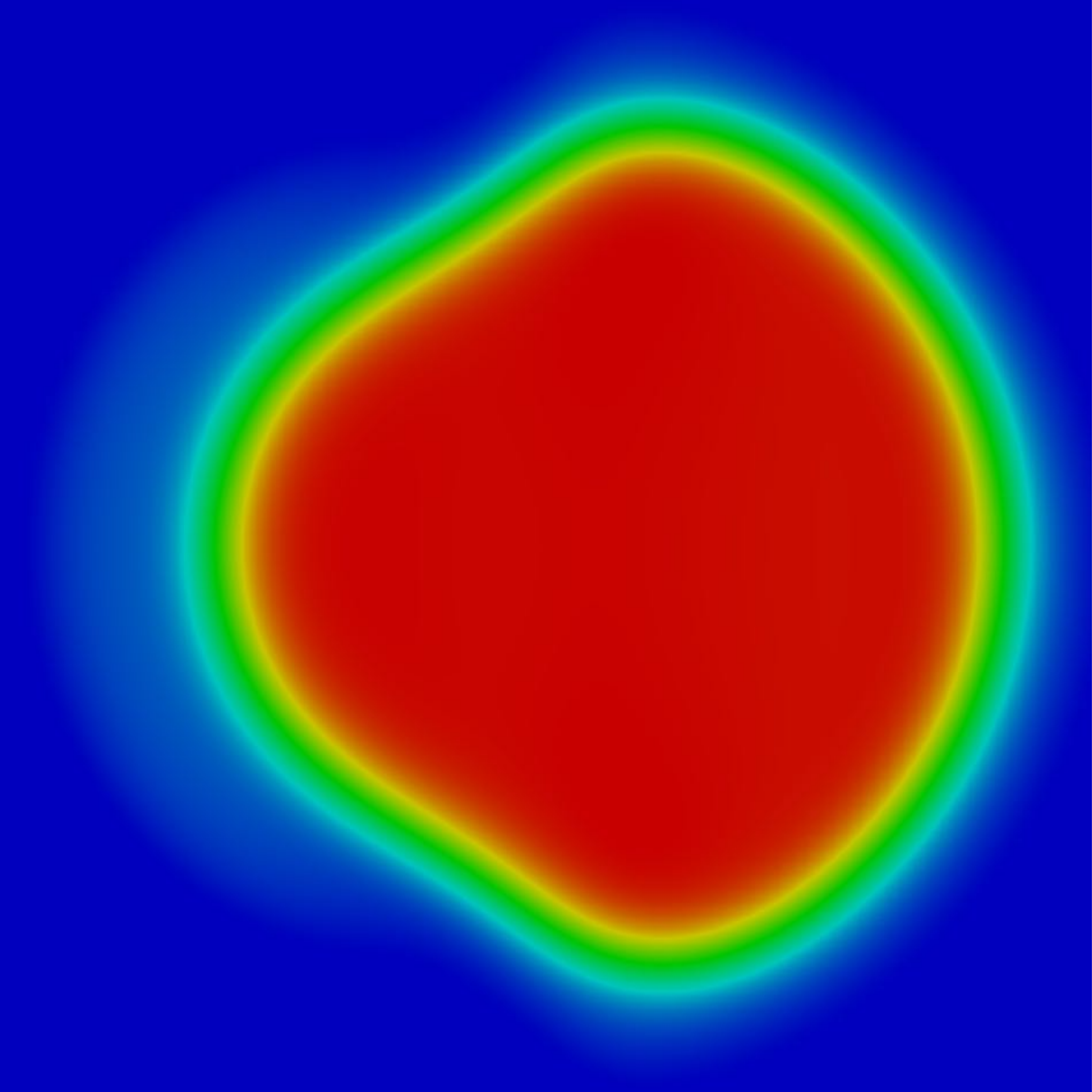} \qquad
	\begin{tikzpicture}
\draw (0, 0) node[inner sep=0] {\includegraphics[width=.2\textwidth]{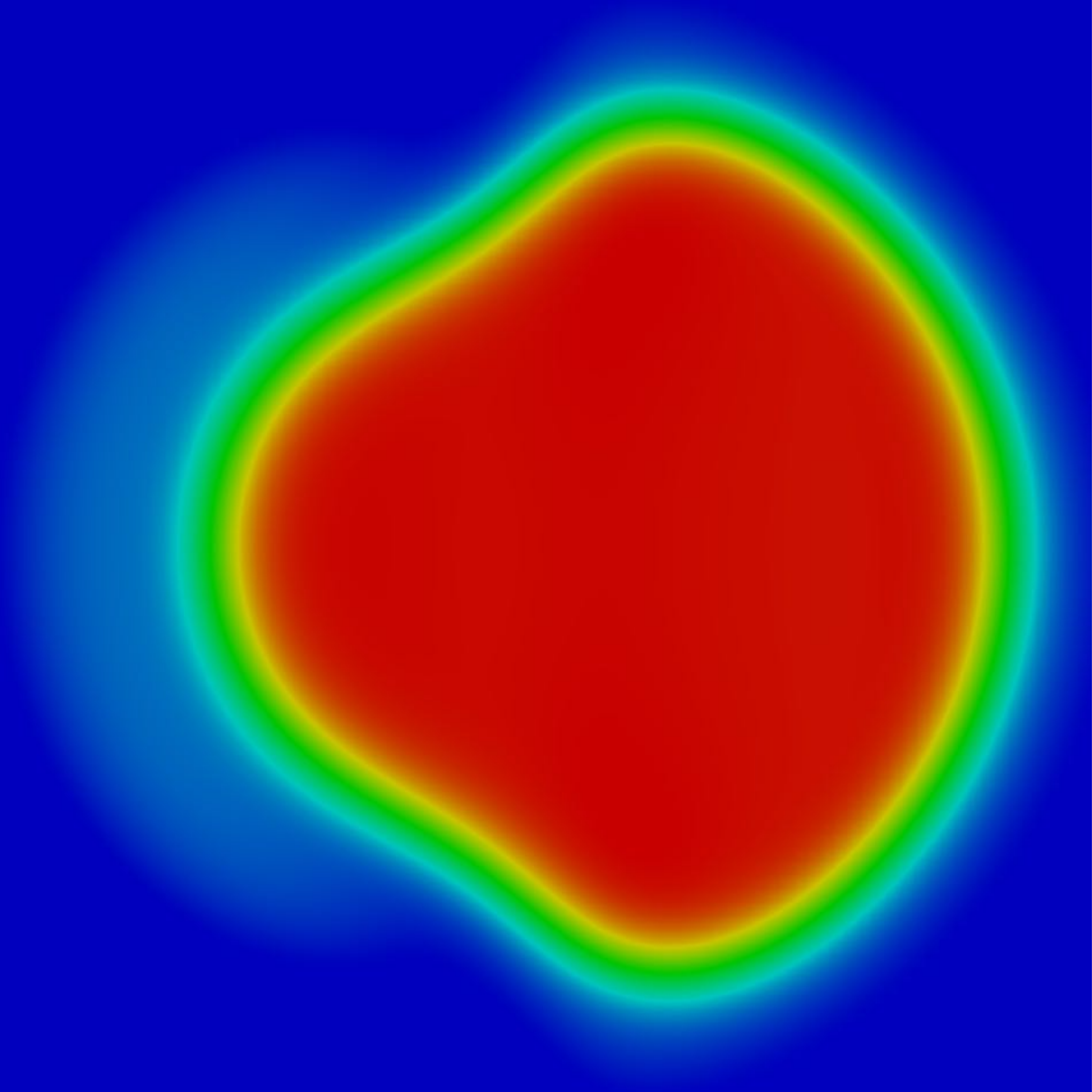}};
\draw (-.9, 1.3) node {$\textbf{\textcolor{white}{0.0025}}$};
\end{tikzpicture} \\[0.1cm]
\begin{tikzpicture}
\draw (0, 0) node[inner sep=0] {\includegraphics[width=.2\textwidth]{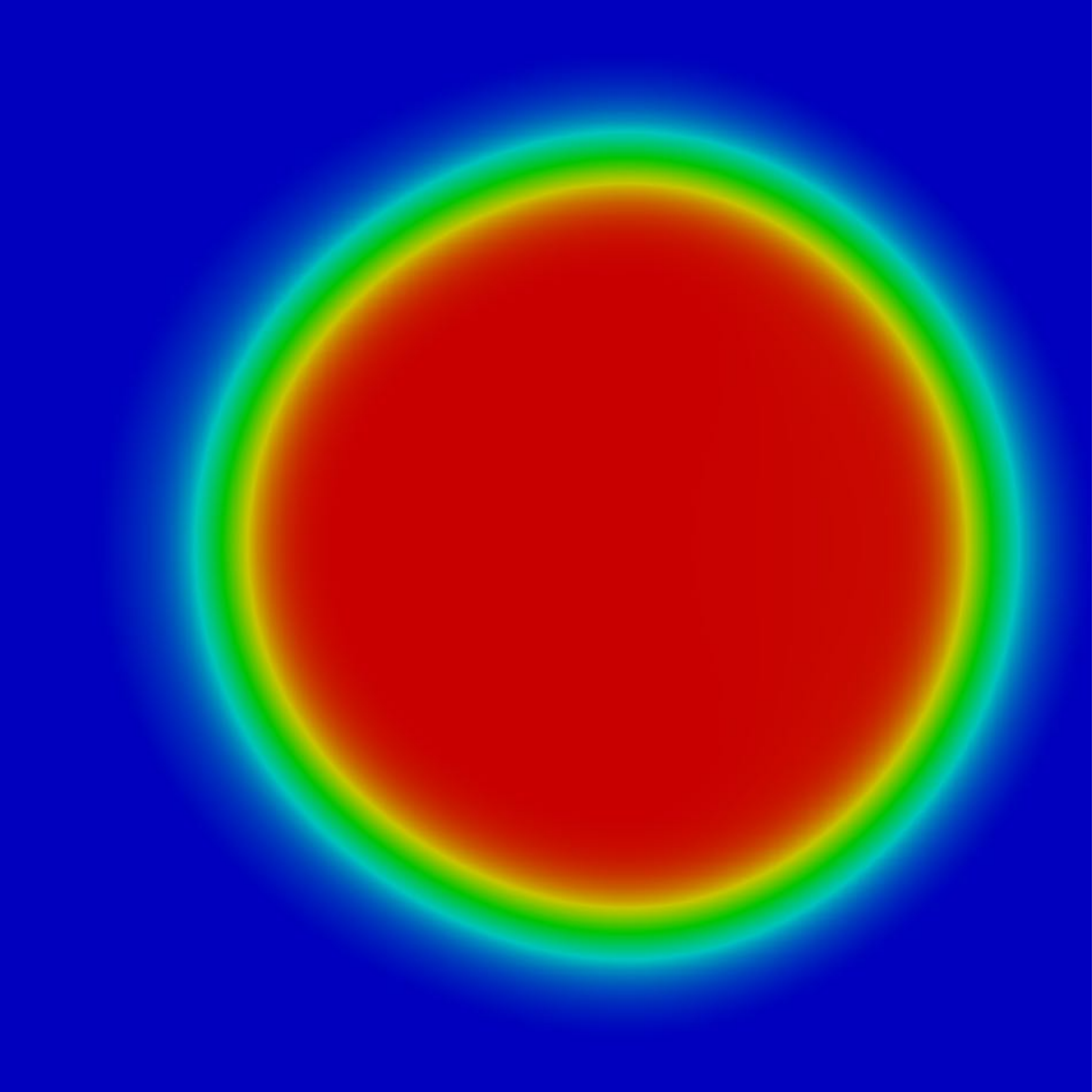}};
\draw (-1.8,0)  node {$\varepsilon_3$};
\end{tikzpicture}
	\includegraphics[width=.2\textwidth]{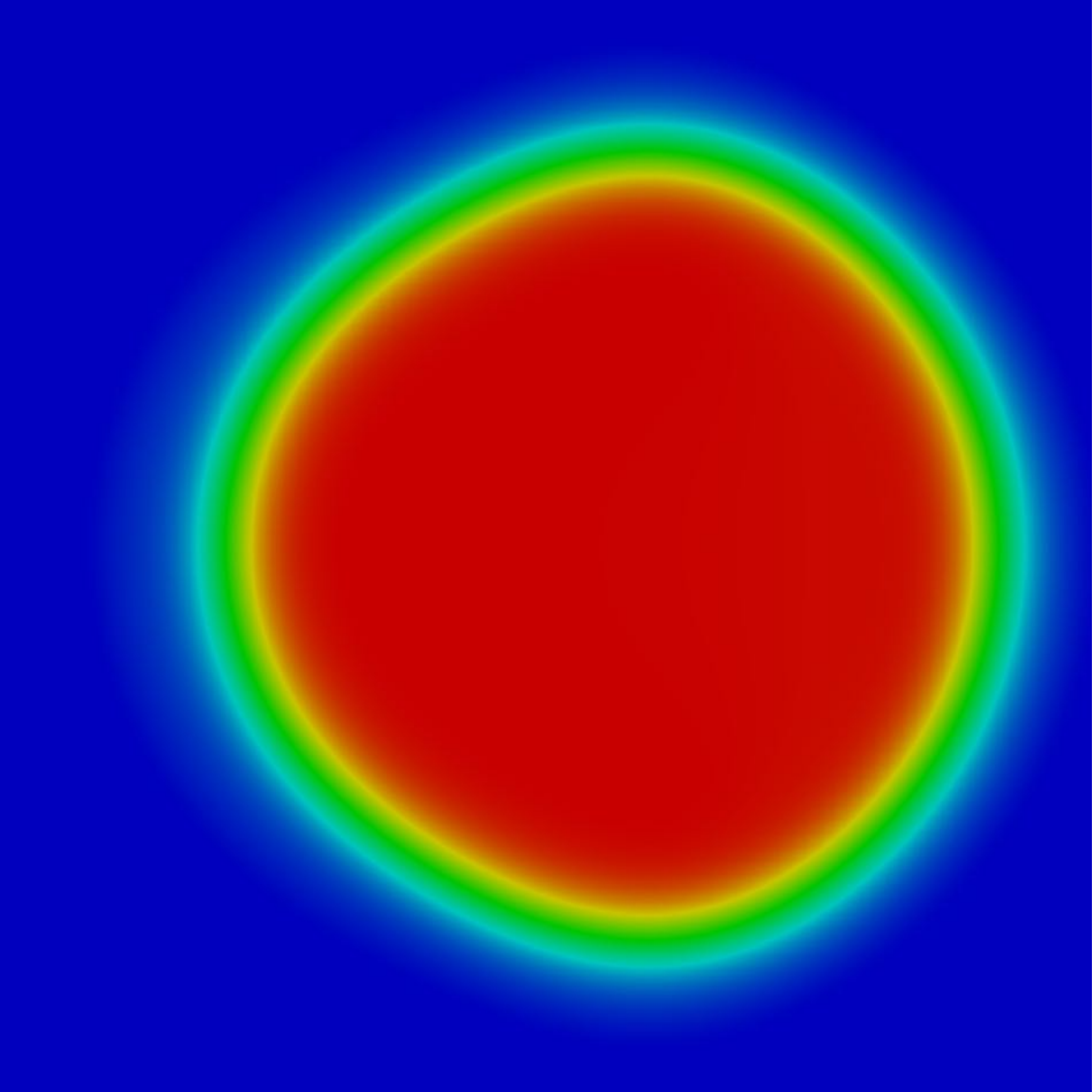} 
	\includegraphics[width=.2\textwidth]{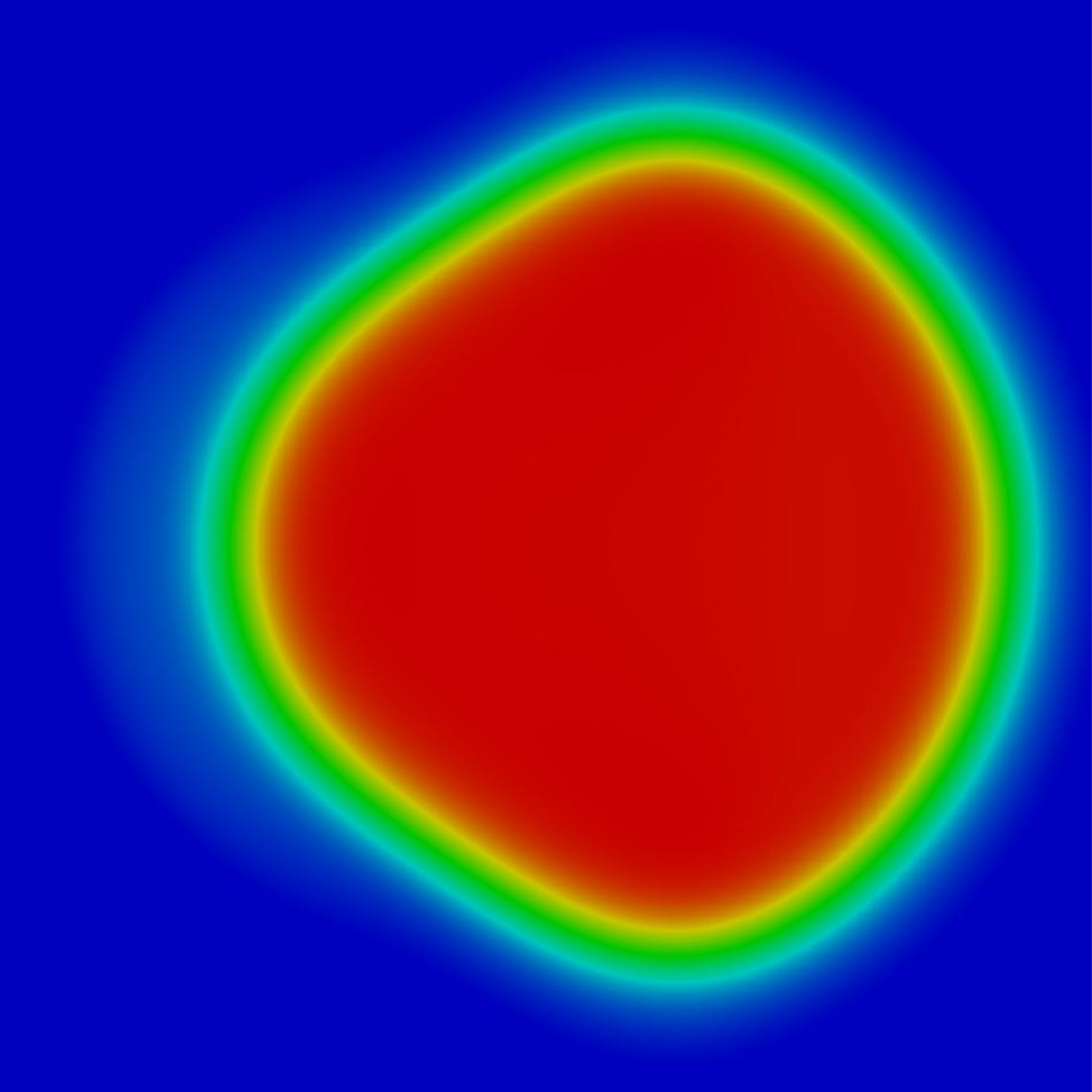} \qquad
\begin{tikzpicture}
\draw (0, 0) node[inner sep=0] {\includegraphics[width=.2\textwidth]{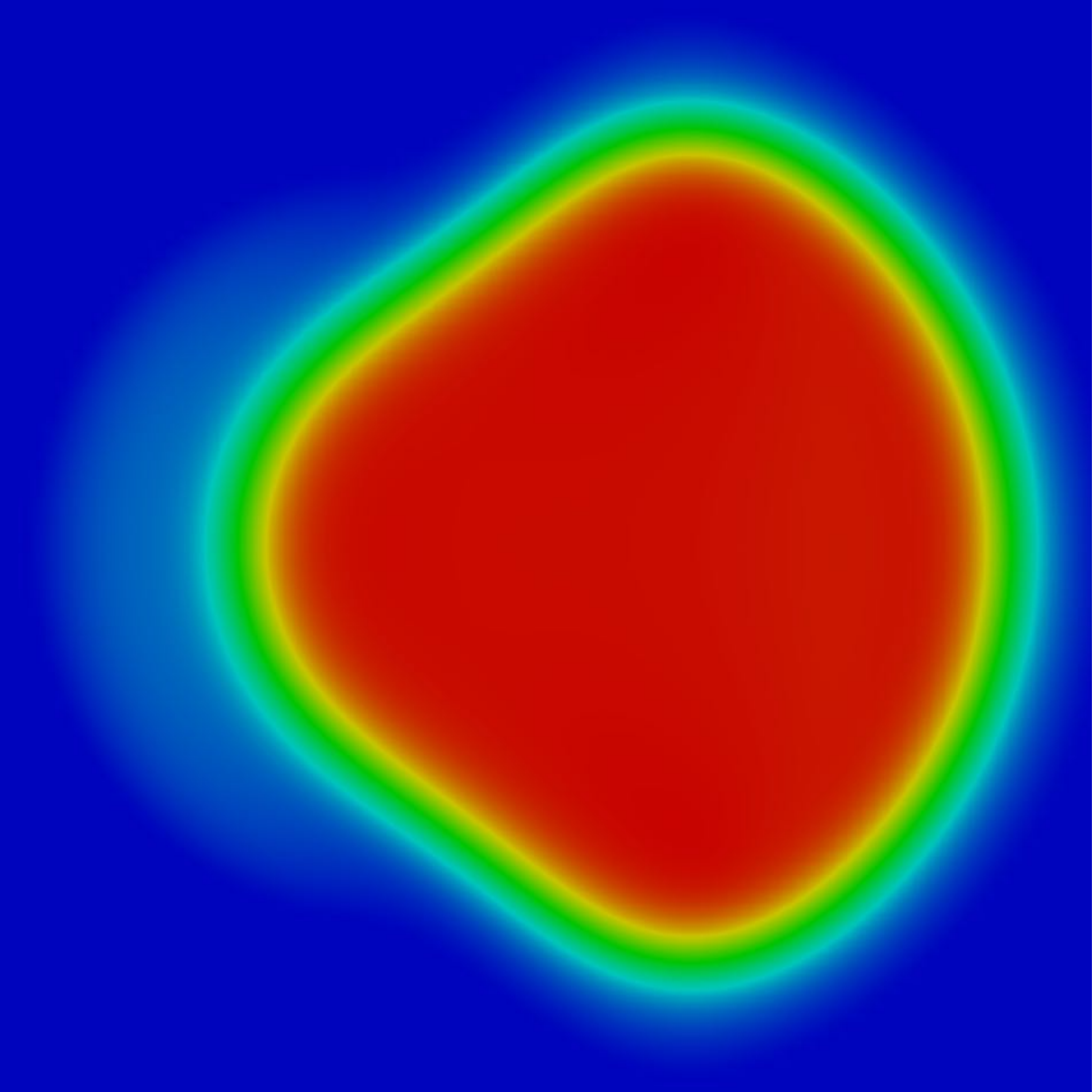}};
\draw (-.9, 1.3) node {$\textbf{\textcolor{white}{0.0030}}$};
\end{tikzpicture} \\
	\begin{tikzpicture}
	\begin{axis}[
	hide axis,
	scale only axis,
	height=0pt,
	width=0pt,
	colormap name=rainbow,
	colorbar horizontal,
	point meta min=0,
	point meta max=1,
	colorbar style={
		samples=100,
		height=.5cm,
		xtick={0,0.5,1},
		width=10cm
	},
	]
	\addplot [draw=none] coordinates {(0,0)};
	\end{axis}
	\end{tikzpicture} 
 	\caption{Simulation of the tumor volume fraction $\phi_T$ for three different haptotaxis parameters $\chi_H \in \{5\cdot 10^{-4},10^{-3},2\cdot 10^{-3}\}$ and different kernel functions $k_\eps$ for $\eps \in \{\eps_1,\eps_2,\eps_3\}:= \{0,2.75\cdot 10^{-2},5.25\cdot 10^{-2}\}$ for a fixed time $t=12$; also three different parameters $\chi_H$ are selected such that the shapes are in accordance with each other}
	\label{FigureLocalNonlocal}
\end{figure}

\section{Concluding Comments} \label{Sec_Conclusion}

 In this study, we have  presented and  analyzed new local and nonlocal mathematical models of growth and of invasion of tumors in healthy tissue  that depict the erosion of the extracellular matrix by matrix-degenerative enzymes and the affects of long-range interactions such as cell-to-cell adhesion. Under reasonable assumptions on the forms of the total energy of the system, potentials, and cell mobility behavior, we proved the existence of solutions to systems of phase-field  models characterized by nonlinear integro-partial differential equations derived  using the  balance laws of mechanics and principal  biological mechanisms know to control the  growth and decline of tumor masses. The results of several numerical experiments based  on two-and three-dimensional finite element approximations of the models are presented which demonstrate that the models provide realistic simulations of the effects of nonlocal interactions  and MDE concentrations  on erosion of the  ECM and corresponding invasion of tumor cells  for various distributions of nutrient  concentration.   
	\pagebreak
	
\section*{Appendix: An existence result for integro-differential equations}

We study the initial value problem for integro-differential systems in the form 
\begin{equation} \label{app:integro}
    \begin{aligned}
    \begin{cases}
        x^{\prime}(t)=f(t, x(t), Kx(t)), \\
        x(0)=x_0,
        \end{cases}
    \end{aligned}
\end{equation}
where $f \in C([0, T]\times \mathbb{R}^n \times \mathbb{R}^n; \mathbb{R}^n)$, $Kx(t)=\int_0^t k(s,x(s)) \, \dd s$, and prove a local existence theorem. To this end, we employ the Schauder fixed-point theorem; see, e.g.,~\cite[Chapter 9.2, Theorem 3]{evans2010partial}. The proof below can be considered as an extension of the Cauchy--Peano theorem and ~\cite[Theorem 1.1.1]{lakshmikantham1995theory}, where a similar integro-differential equation is considered. We note that since $f$ is continuous on $[0,T]$ with respect to $t$, \eqref{app:integro} can be equivalently rewritten as
\begin{equation*}
    \begin{aligned}
        x(t)= x_0+\int_0^t f(s, x(s), K x(s)) \, \textup{d}s,
    \end{aligned}
\end{equation*}
for $t \in [0,T]$.

\begin{theorem}[Local existence of solutions of \eqref{app:integro}]
Let $f \in C([0, T]\times \mathbb{R}^n \times \mathbb{R}^n; \mathbb{R}^n)$ and $k \in C([0,T] \times \mathbb{R}^n; \mathbb{R}^n)$. Then the initial value problem (\ref{app:integro}) has a solution $x$ on the interval $[0,\widetilde T]$ for some $\widetilde T \in (0,T]$. \label{app:theorem}
\end{theorem}
\begin{proof}
The proof follows the general outline of~\cite[Theorem 1.1.1]{lakshmikantham1995theory}. Let $b>0$. The continuous function $k$ is bounded on the compact set $D=[0,T]\times \overline{B}_b(x_0)$,  $$|k(t,x)| \leq C_k, \quad (t,x)\in D.$$ Here, $\overline{B}_b(x_0)$ denotes the closed ball around $x_0$ with radius $b$ in the Euclidean norm. We then have the estimate
$$|K \phi (s)| \leq \int_0^\sigma |k(\sigma,\phi(\sigma))| \, \dd \sigma \leq T C_k =:r.$$
Therefore, $f$ is a continuous function on the compact set $\widetilde D=[0, T] \times \overline B_b(x_0) \times \overline B_r(0).$ Then there exists $0 < C_f < \infty$ such that $$|f(y)| \leq C_f \ \text{for all } y \in \widetilde D.$$ At this point, we introduce the space
$$Y=\{\phi \in C([0, \widetilde T];\mathbb{R}^n) : \phi(0)=x_0 \text{ and } \|\phi-x_0\|_\infty \leq b \}$$
for $\widetilde T = \min\{T,b/C_f\}$, where $\|x\|_\infty=\max_{t \in [0,\widetilde T]} |x(t)|$ for $x \in C([0,\widetilde{T}];\mathbb{R}^n)$. This particular choice of $\widetilde{T}$ will be justified below. \\

Let $\phi \in Y$. We consider the mapping $\mathcal{T}: \phi \mapsto w$ such that
\begin{equation}
    \begin{aligned}
        w(t):= \mathcal{T}\phi(t)=x_0+\int_0^t f(s, \phi(s), K \phi(s)) \, \textup{d}s,
    \end{aligned}
\end{equation}
where $t \in [0,\widetilde T]$. We intend to prove that $\mathcal{T}$ is a continuous self-mapping on the compact and convex set $Y$, which yields the existence of a fixed point $\phi\in Y$ of $\mathcal{T}$ on account of the Schauder fixed-point theorem.

Firstly, we prove that the convexity of $Y$ holds. For arbitrary $\phi,\psi \in Y$, and $\lambda \in [0,1]$, we have $\lambda \phi + (1-\lambda) \psi \in C([0,\widetilde T];\mathbb{R}^n)$ and $\lambda \phi(0) + (1-\lambda) \psi(0)=x_0$. Furthermore, it holds that
$$\|\lambda \phi + (1-\lambda) \psi - x_0\|_\infty = \|\lambda (\phi-x_0) + (1-\lambda) (\psi-x_0) \|_\infty \leq \lambda b + (1-\lambda) b = b.$$

Secondly, we show the compactness of $Y$ by employing the theorem of Arzel\`a--Ascoli; see e.g., \cite[Theorem 4.25]{brezis2010functional}. For all $t_1,t_2 \in [0,\widetilde T]$, we have the uniform equicontinuity of $\mathcal{T}$,
$$\begin{aligned}|\mathcal{T}\phi(t_2)-\mathcal{T}\phi(t_1)| &\leq \int_{t_1}^{t_2}  |f(s,\phi(s), K\phi(s))|  \dd s
\leq C_f |t_2-t_1|.
\end{aligned}$$

Thirdly, we prove that $\mathcal{T}$ is a self-mapping, i.e. $T\phi \in Y$ for $\phi \in Y$. We have $T\phi(0)=w(0)=x_0$ by definition of $\mathcal{T}$. Thanks to our choice of $\widetilde T$, we can conclude  that
$$\begin{aligned} |\mathcal{T}\phi(t)-x_0| &\leq \int_0^t  |f(s,\phi(s), K\phi(s))| \dd s 
\leq \widetilde T C_f \leq b,
\end{aligned}$$
for all $t \in [0,\widetilde T]$.

Finally, we show the continuity of $\mathcal{T}$. Let $\eps>0$ and $\phi, \psi \in Y$ be arbitrary. Since $f$ is uniformly continuous on the compact set $\widetilde D$, there exists a $\delta>0$ with $$|\phi(s)-\psi(s)| + |K\phi(s)-K\psi(s)| <\delta,$$ such that
\begin{equation}
|f(s,\phi(s),K\phi(s))-f(s,\psi(s),K\psi(s))| < \frac{\eps}{\widetilde T},
\label{app:proof:uniform}
\end{equation}
holds true for all $s \in [0,\widetilde T]$. Moreover, $k$ is uniformly continuous on $D$ and hence, there is a $\widetilde \delta>0$ with $|\phi(s)-\psi(s)|<\widetilde \delta$ such that
$$|K\phi(s)-K\psi(s)| < \frac{\delta}{2},$$
which remains true for $|\phi(s)-\psi(s)|<\min\{\widetilde \delta, \delta/2\}=:\widehat \delta$. Hence, we have derived the existence of a parameter $\widehat \delta>0$ with $|\phi(s)-\psi(s)|<\widehat \delta$ such that (\ref{app:proof:uniform}) is fulfilled. Therefore, we conclude 
$$\begin{aligned}
    \|\mathcal{T}\phi-\mathcal{T}\psi\|_\infty \leq   \eps,
\end{aligned}$$
which completes the proof.
\end{proof}

\section*{Acknowledgements}

	The authors gratefully acknowledge the support of the German Science Foundation (DFG) for funding part of this work through grant WO 671/11-1, the Cancer Prevention Research Institute of Texas (CPRIT) under grant number RR160005, the NIH through the grants NCI U01CA174706 and NCI R01CA186193, and the US Department of Energy, Office of Science, Office of Advanced Scientific Computing Research, Mathematical Multifaceted Integrated Capability Centers (MMICCS) program, under award number  DE-SC0019393.
	
\bibliography{references}{}

\begin{thebibliography}{10}

\bibitem{anderson2000mathematical}
{\sc A.~R. Anderson, M.~A. Chaplain, E.~L. Newman, R.~J. Steele, and A.~M.
  Thompson}, {\em Mathematical modelling of tumour invasion and metastasis},
  Computational and mathematical methods in medicine, 2 (2000), pp.~129--154.

\bibitem{araujo2004history}
{\sc R.~P. Araujo and D.~S. McElwain}, {\em A history of the study of solid
  tumour growth: the contribution of mathematical modelling}, Bulletin of
  mathematical biology, 66 (2004), pp.~1039--1091.

\bibitem{armstrong2006continuum}
{\sc N.~J. Armstrong, K.~J. Painter, and J.~A. Sherratt}, {\em A continuum
  approach to modelling cell--cell adhesion}, Journal of Theoretical Biology,
  243 (2006), pp.~98--113.

\bibitem{bellomo2008foundations}
{\sc N.~Bellomo, N.~Li, and P.~K. Maini}, {\em On the foundations of cancer
  modelling: {S}elected topics, speculations, and perspectives}, Mathematical
  Models and Methods in Applied Sciences, 18 (2008), pp.~593--646.

\bibitem{bellomo2000modelling}
{\sc N.~Bellomo and L.~Preziosi}, {\em Modelling and mathematical problems
  related to tumor evolution and its interaction with the immune system},
  Mathematical and Computer Modelling, 32 (2000), pp.~413--452.

\bibitem{boyer2012mathematical}
{\sc F.~Boyer and P.~Fabrie}, {\em Mathematical tools for the study of the
  incompressible {N}avier--{S}tokes equations and related models}, vol.~183,
  Springer Science \& Business Media, 2012.

\bibitem{brezis2010functional}
{\sc H.~Brezis}, {\em Functional analysis, {S}obolev spaces and Partial
  Differential Equations}, Springer Science \& Business Media, 2010.

\bibitem{chaplain2011mathematical}
{\sc M.~A. Chaplain, M.~Lachowicz, Z.~Szyma{\'n}ska, and D.~Wrzosek}, {\em
  Mathematical modelling of cancer invasion: {T}he importance of cell--cell
  adhesion and cell--matrix adhesion}, Mathematical Models and Methods in
  Applied Sciences, 21 (2011), pp.~719--743.

\bibitem{chaplain2005mathematical}
{\sc M.~A. Chaplain and G.~Lolas}, {\em Mathematical modelling of cancer cell
  invasion of tissue: {T}he role of the urokinase plasminogen activation
  system}, Mathematical Models and Methods in Applied Sciences, 15 (2005),
  pp.~1685--1734.

\bibitem{colli2015cahn}
{\sc P.~Colli, G.~Gilardi, and D.~Hilhorst}, {\em On a {C}ahn--{H}illiard type
  phase field system related to tumor growth}, Discrete \& Continuous Dynamical
  Systems - A, 35 (2015), p.~2423.

\bibitem{dai2017analysis}
{\sc M.~Dai, E.~Feireisl, E.~Rocca, G.~Schimperna, and M.~E. Schonbek}, {\em
  Analysis of a diffuse interface model of multispecies tumor growth},
  Nonlinearity, 30 (2017), p.~1639.

\bibitem{deisboeck2010multiscale}
{\sc T.~S. Deisboeck and G.~S. Stamatakos}, {\em Multiscale Cancer Modeling},
  CRC Press, 2010.

\bibitem{della2018nonlocal}
{\sc F.~Della~Porta, A.~Giorgini, and M.~Grasselli}, {\em The nonlocal
  {C}ahn--{H}illiard--{H}ele--{S}haw system with logarithmic potential},
  Nonlinearity, 31 (2018), p.~4851.

\bibitem{della2016nonlocal}
{\sc F.~Della~Porta and M.~Grasselli}, {\em On the nonlocal
  {C}ahn--{H}illiard--{B}rinkman and {C}ahn--{H}illiard--{H}ele--{S}haw
  systems}, Communications in Mathematical Sciences, 13 (2015), pp.~1541--1567.

\bibitem{du2013nonlocal}
{\sc Q.~Du, M.~Gunzburger, R.~B. Lehoucq, and K.~Zhou}, {\em A nonlocal vector
  calculus, nonlocal volume-constrained problems, and nonlocal balance laws},
  Mathematical Models and Methods in Applied Sciences, 23 (2013), pp.~493--540.

\bibitem{ebenbeck2019analysis}
{\sc M.~Ebenbeck and H.~Garcke}, {\em Analysis of a
  {C}ahn--{H}illiard--{B}rinkman model for tumour growth with chemotaxis},
  Journal of Differential Equations, 266 (2019), pp.~5998--6036.

\bibitem{ebenbeck2019cahn}
\leavevmode\vrule height 2pt depth -1.6pt width 23pt, {\em On a
  {C}ahn--{H}illiard--{B}rinkman model for tumor growth and its singular
  limits}, SIAM Journal on Mathematical Analysis, 51 (2019), pp.~1868--1912.

\bibitem{engwer2017structured}
{\sc C.~Engwer, C.~Stinner, and C.~Surulescu}, {\em On a structured multiscale
  model for acid-mediated tumor invasion: {T}he effects of adhesion and
  proliferation}, Mathematical Models and Methods in Applied Sciences, 27
  (2017), pp.~1355--1390.

\bibitem{evans2010partial}
{\sc L.~C. Evans}, {\em Partial {D}ifferential {E}quations}, American
  Mathematical Society, 2010.

\bibitem{frigeri2015a}
{\sc S.~Frigeri, M.~Grasselli, and E.~Rocca}, {\em A diffuse interface model
  for two-phase incompressible flows with non-local interactions and
  non-constant mobility}, Nonlinearity, 28 (2015), pp.~1257--1293.

\bibitem{frigeri2017diffuse}
{\sc S.~Frigeri, K.~F. Lam, and E.~Rocca}, {\em On a diffuse interface model
  for tumour growth with non-local interactions and degenerate mobilities}, in
  Solvability, Regularity, and Optimal Control of Boundary Value Problems for
  PDEs, Springer, 2017, pp.~217--254.

\bibitem{fritz2018unsteady}
{\sc M.~Fritz, E.~A. Lima, J.~T. Oden, and B.~Wohlmuth}, {\em On the unsteady
  {D}arcy-{F}orchheimer-{B}rinkman equation in local and nonlocal tumor growth
  models}, to appear in Mathematical Models and Methods in Applied Sciences,
  (2019).

\bibitem{garcke2016global}
{\sc H.~Garcke and K.~F. Lam}, {\em Global weak solutions and asymptotic limits
  of a {C}ahn--{H}illiard--{D}arcy system modelling tumour growth}, AIMS
  Mathematics, 1 (2016), pp.~318--360.

\bibitem{garcke2017well}
\leavevmode\vrule height 2pt depth -1.6pt width 23pt, {\em Well-posedness of a
  {C}ahn--{H}illiard system modelling tumour growth with chemotaxis and active
  transport}, European Journal of Applied Mathematics, 28 (2017), pp.~284--316.

\bibitem{garcke2018cahn}
\leavevmode\vrule height 2pt depth -1.6pt width 23pt, {\em On a
  {C}ahn--{H}illiard--{D}arcy system for tumour growth with solution dependent
  source terms}, Trends in Applications of Mathematics to Mechanics,  (2018),
  pp.~243--264.

\bibitem{garcke2018multiphase}
{\sc H.~Garcke, K.~F. Lam, R.~N{\"u}rnberg, and E.~Sitka}, {\em A multiphase
  {C}ahn--{H}illiard--{D}arcy model for tumour growth with necrosis},
  Mathematical Models and Methods in Applied Sciences, 28 (2018), pp.~525--577.

\bibitem{garcke2016cahn}
{\sc H.~Garcke, K.~F. Lam, E.~Sitka, and V.~Styles}, {\em A
  {C}ahn--{H}illiard--{D}arcy model for tumour growth with chemotaxis and
  active transport}, Mathematical Models and Methods in Applied Sciences, 26
  (2016), pp.~1095--1148.

\bibitem{gatenby1995models}
{\sc R.~A. Gatenby}, {\em Models of tumor-host interaction as competing
  populations: {I}mplications for tumor biology and treatment}, Journal of
  Theoretical Biology, 176 (1995), pp.~447--455.

\bibitem{gerisch2010approximation}
{\sc A.~Gerisch}, {\em On the approximation and efficient evaluation of
  integral terms in pde models of cell adhesion}, IMA journal of numerical
  analysis, 30 (2010), pp.~173--194.

\bibitem{gerisch2008mathematical}
{\sc A.~Gerisch and M.~Chaplain}, {\em Mathematical modelling of cancer cell
  invasion of tissue: {L}ocal and non-local models and the effect of adhesion},
  Journal of Theoretical Biology, 250 (2008), pp.~684--704.

\bibitem{hawkins2012numerical}
{\sc A.~Hawkins-Daarud, K.~G. van~der Zee, and J.~Tinsley~Oden}, {\em Numerical
  simulation of a thermodynamically consistent four-species tumor growth
  model}, International journal for numerical methods in biomedical
  engineering, 28 (2012), pp.~3--24.

\bibitem{hillen2013convergence}
{\sc T.~Hillen, K.~J. Painter, and M.~Winkler}, {\em Convergence of a cancer
  invasion model to a logistic chemotaxis model}, Mathematical Models and
  Methods in Applied Sciences, 23 (2013), pp.~165--198.

\bibitem{jiang2015well}
{\sc J.~Jiang, H.~Wu, and S.~Zheng}, {\em Well-posedness and long-time behavior
  of a non-autonomous {C}ahn--{H}illiard--{D}arcy system with mass source
  modeling tumor growth}, Journal of Differential Equations, 259 (2015),
  pp.~3032--3077.

\bibitem{libMeshPaper}
{\sc B.~S. Kirk, J.~W. Peterson, R.~H. Stogner, and G.~F. Carey}, {\em
  {libMesh: A C++ Library for Parallel Adaptive Mesh Refinement/Coarsening
  Simulations}}, Engineering with Computers, 22 (2006), pp.~237--254.

\bibitem{lakshmikantham1995theory}
{\sc V.~Lakshmikantham and M.~Rama Mohana~Rao}, {\em Theory of
  Integro-Differential Equations}, Gordon and Breach Science Publications,
  1995.

\bibitem{lam2017thermodynamically}
{\sc K.~F. Lam and H.~Wu}, {\em Thermodynamically consistent
  {N}avier--{S}tokes--{C}ahn--{H}illiard models with mass transfer and
  chemotaxis}, European Journal of Applied Mathematics,  (2017), pp.~1--50.

\bibitem{leoni2007necessary}
{\sc G.~Leoni and M.~Morini}, {\em Necessary and sufficient conditions for the
  chain rule in ${W}_\text{loc}^{1,1}(\mathbb{R}^n;\mathbb{R}^d)$ and
  ${BV}_\text{loc}(\mathbb{R}^n;\mathbb{R}^d)$}, Journal of the European
  Mathematical Society, 9 (2007), pp.~219--252.

\bibitem{lieb2001analysis}
{\sc E.~H. Lieb and M.~Loss}, {\em Analysis}, Springer Science \& Business
  Media, 2001.

\bibitem{lima2014hybrid}
{\sc E.~Lima, J.~Oden, and R.~Almeida}, {\em A hybrid ten-species phase-field
  model of tumor growth}, Mathematical Models and Methods in Applied Sciences,
  24 (2014), pp.~2569--2599.

\bibitem{lima2017selection}
{\sc E.~Lima, J.~Oden, B.~Wohlmuth, A.~Shahmoradi, D.~Hormuth~II, T.~Yankeelov,
  L.~Scarabosio, and T.~Horger}, {\em Selection and validation of predictive
  models of radiation effects on tumor growth based on noninvasive imaging
  data}, Computer methods in applied mechanics and engineering, 327 (2017),
  pp.~277--305.

\bibitem{lima2015analysis}
{\sc E.~A. Lima, R.~C. Almeida, and J.~T. Oden}, {\em Analysis and numerical
  solution of stochastic phase-field models of tumor growth}, Numerical Methods
  for Partial Differential Equations, 31 (2015), pp.~552--574.

\bibitem{lions2012non}
{\sc J.~L. Lions and E.~Magenes}, {\em Non-Homogeneous Boundary Value Problems
  and Applications}, Springer Science \& Business Media, 2012.

\bibitem{madsen2015source}
{\sc D.~H. Madsen and T.~H. Bugge}, {\em The source of matrix-degrading enzymes
  in human cancer: {P}roblems of research reproducibility and possible
  solutions}, J Cell Biol, 209 (2015), pp.~195--198.

\bibitem{marchant2001travelling}
{\sc B.~Marchant, J.~Norbury, and J.~Sherratt}, {\em Travelling wave solutions
  to a haptotaxis-dominated model of malignant invasion}, Nonlinearity, 14
  (2001), p.~1653.

\bibitem{mengesha2015localization}
{\sc T.~Mengesha and D.~Spector}, {\em Localization of nonlocal gradients in
  various topologies}, Calculus of Variations and Partial Differential
  Equations, 52 (2015), pp.~253--279.

\bibitem{murat2003chain}
{\sc F.~Murat and C.~Trombetti}, {\em A chain rule formula for the composition
  of a vector-valued function by a piecewise smooth function}, Bollettino
  dell'Unione Matematica Italiana, 6 (2003), pp.~581--595.

\bibitem{nargis2016effects}
{\sc N.~Nargis and R.~Aldredge}, {\em Effects of matrix metalloproteinase on
  tumour growth and morphology via haptotaxis}, J Bioengineer \& Biomedical
  Sci, 6 (2016), p.~2.

\bibitem{oden2016toward}
{\sc J.~T. Oden, E.~A. Lima, R.~C. Almeida, Y.~Feng, M.~N. Rylander,
  D.~Fuentes, D.~Faghihi, M.~M. Rahman, M.~DeWitt, M.~Gadde, et~al.}, {\em
  Toward predictive multiscale modeling of vascular tumor growth}, Archives of
  Computational Methods in Engineering, 23 (2016), pp.~735--779.

\bibitem{peng2017multiscale}
{\sc L.~Peng, D.~Trucu, P.~Lin, A.~Thompson, and M.~A. Chaplain}, {\em A
  multiscale mathematical model of tumour invasive growth}, Bulletin of
  mathematical biology, 79 (2017), pp.~389--429.

\bibitem{perumpanani2000traveling}
{\sc A.~Perumpanani, B.~Marchant, and J.~Norbury}, {\em Traveling shock waves
  arising in a model of malignant invasion}, SIAM Journal on Applied
  Mathematics, 60 (2000), pp.~463--476.

\bibitem{perumpani1996biological}
{\sc B.~Perumpani, A.~Sherratt, J.~Norbury, and H.~Byrne}, {\em Biological
  inferences from a mathematical model for malignant invasion}, Invasion
  Metastasis, 16 (1996), pp.~209--22l.

\bibitem{robinson2001infinite}
{\sc J.~C. Robinson}, {\em Infinite-{D}imensional {D}ynamical {S}ystems: {A}n
  {I}ntroduction to {D}issipative {P}arabolic {PDE}s and the {T}heory of
  {G}lobal {A}ttractors}, Cambridge University Press, 2001.

\bibitem{roubicek}
{\sc T.~Roub{\'\i}{\v{c}}ek}, {\em Nonlinear {P}artial {D}ifferential
  {E}quations with {A}pplications}, Springer Science \& Business Media, 2013.

\bibitem{simon1986compact}
{\sc J.~Simon}, {\em Compact sets in the space ${L}^p(0, {T}; {B})$}, Annali di
  Matematica pura ed applicata, 146 (1986), pp.~65--96.

\bibitem{stinner2014global}
{\sc C.~Stinner, C.~Surulescu, and M.~Winkler}, {\em Global weak solutions in a
  {PDE-ODE} system modeling multiscale cancer cell invasion}, SIAM Journal on
  Mathematical Analysis, 46 (2014), pp.~1969--2007.

\bibitem{strauss1966continuity}
{\sc W.~Strauss}, {\em On continuity of functions with values in various
  {B}anach spaces}, Pacific Journal of Mathematics, 19 (1966), pp.~543--551.

\bibitem{tao2011chemotaxis}
{\sc Y.~Tao and M.~Winkler}, {\em A chemotaxis-haptotaxis model: the roles of
  nonlinear diffusion and logistic source}, SIAM Journal on Mathematical
  Analysis, 43 (2011), pp.~685--704.

\bibitem{walker2007global}
{\sc C.~Walker and G.~F. Webb}, {\em Global existence of classical solutions
  for a haptotaxis model}, SIAM Journal on Mathematical Analysis, 38 (2007),
  pp.~1694--1713.

\bibitem{walter1998ordinary}
{\sc W.~Walter}, {\em Ordinary Differential Equations}, Springer Science \&
  Business Media, 1998.

\bibitem{wise2008three}
{\sc S.~M. Wise, J.~S. Lowengrub, H.~B. Frieboes, and V.~Cristini}, {\em
  Three-dimensional multispecies nonlinear tumor growth {I}: {M}odel and
  numerical method}, Journal of Theoretical Biology, 253 (2008), pp.~524--543.

\bibitem{ziemer2012weakly}
{\sc W.~P. Ziemer}, {\em Weakly Differentiable Functions: Sobolev Spaces and
  Functions of Bounded Variation}, Springer Science \& Business Media, 1989.

\end{thebibliography}
\bibliographystyle{siam} 

\end{document}